\documentclass{article} 

\usepackage{amsmath}
\usepackage{amssymb}
\usepackage{amstext}
\usepackage{xspace}
\usepackage{url}
\usepackage{xy}
\xyoption{all}
\usepackage{proof}
\usepackage{stmaryrd}
\usepackage{rotating}

\newdimen\proofrulebreadth \proofrulebreadth=.05em
\newdimen\proofdotseparation \proofdotseparation=1.25ex
\newdimen\proofrulebaseline \proofrulebaseline=2ex
\newcount\proofdotnumber \proofdotnumber=3
\let\then\relax
\def\hfi{\hskip0pt plus.0001fil}
\mathchardef\squigto="3A3B
\newif\ifinsideprooftree\insideprooftreefalse
\newif\ifonleftofproofrule\onleftofproofrulefalse
\newif\ifproofdots\proofdotsfalse
\newif\ifdoubleproof\doubleprooffalse
\let\wereinproofbit\relax
\newdimen\shortenproofleft
\newdimen\shortenproofright
\newdimen\proofbelowshift
\newbox\proofabove
\newbox\proofbelow
\newbox\proofrulename
\def\shiftproofbelow{\let\next\relax\afterassignment\setshiftproofbelow\dimen0 }
\def\shiftproofbelowneg{\def\next{\multiply\dimen0 by-1 }%
\afterassignment\setshiftproofbelow\dimen0 }
\def\setshiftproofbelow{\next\proofbelowshift=\dimen0 }
\def\setproofrulebreadth{\proofrulebreadth}

\def\prooftree{
\ifnum  \lastpenalty=1
\then   \unpenalty
\else   \onleftofproofrulefalse
\fi
\ifonleftofproofrule
\else   \ifinsideprooftree
        \then   \hskip.5em plus1fil
        \fi
\fi
\bgroup
\setbox\proofbelow=\hbox{}\setbox\proofrulename=\hbox{}%
\let\justifies\proofover\let\leadsto\proofoverdots\let\Justifies\proofoverdbl
\let\using\proofusing\let\[\prooftree
\ifinsideprooftree\let\]\endprooftree\fi
\proofdotsfalse\doubleprooffalse
\let\thickness\setproofrulebreadth
\let\shiftright\shiftproofbelow \let\shift\shiftproofbelow
\let\shiftleft\shiftproofbelowneg
\let\ifwasinsideprooftree\ifinsideprooftree
\insideprooftreetrue
\setbox\proofabove=\hbox\bgroup$\displaystyle 
\let\wereinproofbit\prooftree
\shortenproofleft=0pt \shortenproofright=0pt \proofbelowshift=0pt
\onleftofproofruletrue\penalty1
}

\def\eproofbit{
\ifx    \wereinproofbit\prooftree
\then   \ifcase \lastpenalty
        \then   \shortenproofright=0pt  
        \or     \unpenalty\hfil         
        \or     \unpenalty\unskip       
        \else   \shortenproofright=0pt  
        \fi
\fi
\global\dimen0=\shortenproofleft
\global\dimen1=\shortenproofright
\global\dimen2=\proofrulebreadth
\global\dimen3=\proofbelowshift
\global\dimen4=\proofdotseparation
\global\count255=\proofdotnumber
$\egroup  
\shortenproofleft=\dimen0
\shortenproofright=\dimen1
\proofrulebreadth=\dimen2
\proofbelowshift=\dimen3
\proofdotseparation=\dimen4
\proofdotnumber=\count255
}

\def\proofover{
\eproofbit 
\setbox\proofbelow=\hbox\bgroup 
\let\wereinproofbit\proofover
$\displaystyle
}%
\def\proofoverdbl{
\eproofbit 
\doubleprooftrue
\setbox\proofbelow=\hbox\bgroup 
\let\wereinproofbit\proofoverdbl
$\displaystyle
}%
\def\proofoverdots{
\eproofbit 
\proofdotstrue
\setbox\proofbelow=\hbox\bgroup 
\let\wereinproofbit\proofoverdots
$\displaystyle
}%
\def\proofusing{
\eproofbit 
\setbox\proofrulename=\hbox\bgroup 
\let\wereinproofbit\proofusing
\kern0.3em$
}

\def\endprooftree{
\eproofbit 
  \dimen5 =0pt
\dimen0=\wd\proofabove \advance\dimen0-\shortenproofleft
\advance\dimen0-\shortenproofright
\dimen1=.5\dimen0 \advance\dimen1-.5\wd\proofbelow
\dimen4=\dimen1
\advance\dimen1\proofbelowshift \advance\dimen4-\proofbelowshift
\ifdim  \dimen1<0pt
\then   \advance\shortenproofleft\dimen1
        \advance\dimen0-\dimen1
        \dimen1=0pt
        \ifdim  \shortenproofleft<0pt
        \then   \setbox\proofabove=\hbox{%
                        \kern-\shortenproofleft\unhbox\proofabove}%
                \shortenproofleft=0pt
        \fi
\fi
\ifdim  \dimen4<0pt
\then   \advance\shortenproofright\dimen4
        \advance\dimen0-\dimen4
        \dimen4=0pt
\fi
\ifdim  \shortenproofright<\wd\proofrulename
\then   \shortenproofright=\wd\proofrulename
\fi
\dimen2=\shortenproofleft \advance\dimen2 by\dimen1
\dimen3=\shortenproofright\advance\dimen3 by\dimen4
\ifproofdots
\then
        \dimen6=\shortenproofleft \advance\dimen6 .5\dimen0
        \setbox1=\vbox to\proofdotseparation{\vss\hbox{$\cdot$}\vss}%
        \setbox0=\hbox{%
                \advance\dimen6-.5\wd1
                \kern\dimen6
                $\vcenter to\proofdotnumber\proofdotseparation
                        {\leaders\box1\vfill}$%
                \unhbox\proofrulename}%
\else   \dimen6=\fontdimen22\the\textfont2 
        \dimen7=\dimen6
        \advance\dimen6by.5\proofrulebreadth
        \advance\dimen7by-.5\proofrulebreadth
        \setbox0=\hbox{%
                \kern\shortenproofleft
                \ifdoubleproof
                \then   \hbox to\dimen0{%
                        $\mathsurround0pt\mathord=\mkern-6mu%
                        \cleaders\hbox{$\mkern-2mu=\mkern-2mu$}\hfill
                        \mkern-6mu\mathord=$}%
                \else   \vrule height\dimen6 depth-\dimen7 width\dimen0
                \fi
                \unhbox\proofrulename}%
        \ht0=\dimen6 \dp0=-\dimen7
\fi
\let\doll\relax
\ifwasinsideprooftree
\then   \let\VBOX\vbox
\else   \ifmmode\else$\let\doll=$\fi
        \let\VBOX\vcenter
\fi
\VBOX   {\baselineskip\proofrulebaseline \lineskip.2ex
        \expandafter\lineskiplimit\ifproofdots0ex\else-0.6ex\fi
        \hbox   spread\dimen5   {\hfi\unhbox\proofabove\hfi}%
        \hbox{\box0}%
        \hbox   {\kern\dimen2 \box\proofbelow}}\doll%
\global\dimen2=\dimen2
\global\dimen3=\dimen3
\egroup 
\ifonleftofproofrule
\then   \shortenproofleft=\dimen2
\fi
\shortenproofright=\dimen3
\onleftofproofrulefalse
\ifinsideprooftree
\then   \hskip.5em plus 1fil \penalty2
\fi
}

\newif\ifignore 
\ignorefalse

\newcommand{\auxproof}[1]{
\ifignore\mbox{}\newline
\textbf{PROOF:} \dotfill\newline
{\it #1}\mbox{}\newline
\textbf{ENDPROOF}\dotfill
\fi}

\newtheorem{theorem}{Theorem}
\newtheorem{lemma}[theorem]{Lemma}

\newtheorem{corollary}[theorem]{Corollary}
\newtheorem{definition}[theorem]{Definition}

\newenvironment{proof}[1][Proof]%
   { \begin{trivlist}%
     \item[\hskip \labelsep {\bfseries #1}]%
   }%
   { \end{trivlist}%
   }
\newcommand{\QEDbox}{\square}
\newcommand{\QED}{\hspace*{\fill}$\QEDbox$}

\newcommand{\after}{\mathrel{\circ}}
\newcommand{\klafter}{\mathrel{\bullet}} 
\newcommand{\cat}[1]{\ensuremath{\mathbf{#1}}}
\newcommand{\Cat}[1]{\ensuremath{\mathbf{#1}}}

\newcommand{\op}{\ensuremath{^{\mathrm{op}}}}
\newcommand{\idmap}{\ensuremath{\mathrm{id}}}
\newcommand{\id}{\idmap}

\newcommand{\support}{\ensuremath{\mathrm{supp}}}
\newcommand{\st}{\ensuremath{\mathsf{st}}\xspace}
\newcommand{\bc}{\ensuremath{\mathsf{bc}}\xspace}
\newcommand{\dst}{\ensuremath{\mathsf{dst}}\xspace}
\newcommand{\coo}{\ensuremath{\mathsf{co}}\xspace}
\newcommand{\sotimes}{\mathrel{\raisebox{.05pc}{$\scriptstyle \otimes$}}}
\newcommand{\Alg}{\textsl{Alg}\xspace}

\newcommand{\NNO}{\ensuremath{\mathbb{N}}}
\newcommand{\congrightarrow}{\mathrel{\stackrel{
           \raisebox{.5ex}{$\scriptstyle\cong\,$}}{
           \raisebox{0ex}[0ex][0ex]{$\rightarrow$}}}}

\newcommand{\powerset}{\mathcal{P}}
\newcommand{\leftScottint}{[{\kern-.3ex}[}
\newcommand{\rightScottint}{]{\kern-.3ex}]}

\newcommand{\Sets}{\Cat{Sets}\xspace}
\newcommand{\Mon}{\Cat{Mon}\xspace}
\newcommand{\CMon}{\Cat{CMon}\xspace}
\newcommand{\Mnd}{\Cat{Mnd}\xspace}
\newcommand{\CMnd}{\Cat{CMnd}\xspace}
\newcommand{\ACMnd}{\Cat{ACMnd}\xspace}
\newcommand{\IACMnd}{\Cat{IACMnd}\xspace}
\newcommand{\StMnd}{\Cat{StMnd}\xspace}
\newcommand{\SRng}{\Cat{SRng}\xspace}
\newcommand{\CSRng}{\Cat{CSRng}\xspace}
\newcommand{\ICSRng}{\Cat{ICSRng}\xspace}
\newcommand{\Law}{\Cat{Law}\xspace}

\newcommand{\SMLaw}{\Cat{SMLaw}\xspace}
\newcommand{\SMBLaw}{\Cat{SMBLaw}\xspace}
\newcommand{\DSMBLaw}{\Cat{DSMBLaw}\xspace}

\newcommand{\cotuple}[2]{\ensuremath{[ #1, #2 ]}}
\newcommand{\tuple}[2]{\ensuremath{\langle #1, #2 \rangle}}
\newcommand{\Kl}{\mathcal{K}{\kern-.2ex}\ell}
\newcommand{\Ev}{\mathcal{E}}
\newcommand{\Evc}{\mathcal{E}}
\newcommand{\Mod}{\mathcal{M}od}
\newcommand{\Evx}{\mathcal{A}d}
\newcommand{\Mat}{\mathcal{M}at}
\newcommand{\LM}{\mathcal{T}}

\newcommand{\set}[2]{\{#1\;|\;#2\}}
\newcommand{\setin}[3]{\{#1\in#2\;|\;#3\}}

\newcommand{\lamin}[3]{\lambda{#1\in#2}.\,#3}
\newcommand{\lam}[2]{\lambda{#1}.\,#2}

\newcommand{\pullback}[1][dr]{\save*!/#1-1.2pc/#1:(-1,1)@^{|-}\restore}

\newenvironment{bijectivecorrespondence}
  {\newif\ifbijnotfirst
   \global\bijnotfirstfalse
   \global\def\bijprev{}
   \normalsize
   \begin{tabular}{cl}}
  {\end{tabular}
   }
\newcommand{\correspondence}[2][]{%
  \ifbijnotfirst%
    \rule{0pt}{5.8pt}%
    \smash{\ensuremath{\infer={\hphantom{#2}}{\hphantom{\bijprev}}}} \\%
  \fi%
  \global\bijnotfirsttrue%
  \global\def\bijprev{#2}%
  \ensuremath{#2} & #1 \\%
  }

\title{Scalars, Monads, and Categories}

\author{\small\begin{tabular}{ccc}
{\large Dion Coumans}
& \qquad\qquad &
{\large Bart Jacobs} \\
Inst.\ for Mathematics, Astrophysics
& &
Inst.\ for Computing and \\
and Particle Physics (IMAPP) & &
Information Sciences (ICIS) \\
\url{www.math.ru.nl/~coumans}
& &
\url{www.cs.ru.nl/~bart}
\end{tabular} \\
{\small Radboud University Nijmegen, The Netherlands}}

\date{\small \today}

\begin{document}
\maketitle

\begin{abstract}
  The paper describes interrelations between: (1)~algebraic
  structure on sets of scalars, (2)~properties of monads associated
  with such sets of scalars, and (3)~structure in categories
  (esp.\ Lawvere theories) associated with these monads. These
  interrelations will be expressed in terms of ``triangles of
  adjunctions'', involving for instance various kinds of monoids
  (non-commutative, commutative, involutive) and semirings as
  scalars. It will be shown to which kind of monads and categories
  these algebraic structures correspond via adjunctions.
\end{abstract}

%

\section{Introduction}\label{IntroSec}

Scalars are the elements $s$ used in scalar multiplication $s\cdot v$,
yielding for instance a new vector for a given vector $v$. Scalars are
elements in some algebraic structure, such as a field (for vector
spaces), a ring (for modules), a group (for group actions), or a
monoid (for monoid actions).

A categorical description of scalars can be given in a monoidal
category $\cat{C}$, with tensor $\otimes$ and tensor unit $I$, as the
homset $\cat{C}(I,I)$ of endomaps on $I$. In~\cite{KellyL80} it is
shown that such homsets $\cat{C}(I,I)$ always form a commutative
monoid; in~\cite[\S3.2]{AbramskyC09} this is called the `miracle' of
scalars. More recent work in the area of quantum computation has led
to renewed interest in such scalars, see for
instance~\cite{AbramskyC04,AbramskyC09}, where it is shown that the
presence of biproducts makes this homset $\cat{C}(I,I)$ of scalars a
semiring, and that daggers $\dag$ make it involutive.  These are first
examples where categorical structure (a category which is monoidal or
has biproducts or daggers) gives rise to algebraic structure (a set
with a commutative monoid, semiring or involution structure).  Such
correspondences form the focus of this paper, not only those between
categorical and algebraic structure, but also involving a third
element, namely structure on endofunctors (especially monads).  Such
correspondences will be described in terms of triangles of
adjunctions.

To start, we describe the basic triangle of adjunctions that we shall
build on. At this stage it is meant as a sketch of the setting, and
not as an exhaustive explanation.  Let $\aleph_{0}$ be the category
with natural numbers $n\in\NNO$ as objects. Such a number $n$ is
identified with the $n$-element set
$\underline{n}=\{0,1,\ldots,n-1\}$. Morphisms $n\rightarrow m$ in
$\aleph_0$ are ordinary functions $\underline{n}\rightarrow
\underline{m}$ between these finite sets.  Hence there is a full and
faithful functor $\aleph_{0} \hookrightarrow \Sets$. The underline
notation is useful to avoid ambiguity, but we often omit it when no
confusion arises and write the number $n$ for the set $\underline{n}$.

\begin{figure}
\label{SetsTriangle}
$$
\vcenter{\xymatrix@R-0pc@C+.5pc{
& & \Sets\ar@/_2ex/ [ddll]_{\begin{array}{c}\scriptstyle A\mapsto \\[-.7pc] 
                            \scriptstyle A\times(-)\end{array}}
   \ar@/_2ex/ [ddrr]_(0.2){\begin{array}{c}\scriptstyle A\mapsto \\[-.7pc] 
                       \scriptstyle A\times(-)\end{array}\hspace*{-1.5pc}} \\
& \dashv\;\; & & \dashv & \\
\Sets^{\Sets}\ar @/_2ex/[rrrr]_{\textrm{restrict}}
   \ar@/_2ex/ [uurr]_(0.6){(-)(1)} & & & &  
   \Sets^{\aleph_0}\ar@/_2ex/ [uull]_{(-)(1)}
      \ar @/_2ex/[llll]_{\textrm{left Kan}}^{\raisebox{-.7pc}{$\bot$}}
}}
$$
\caption{Basic triangle of adjunctions.}
\end{figure}
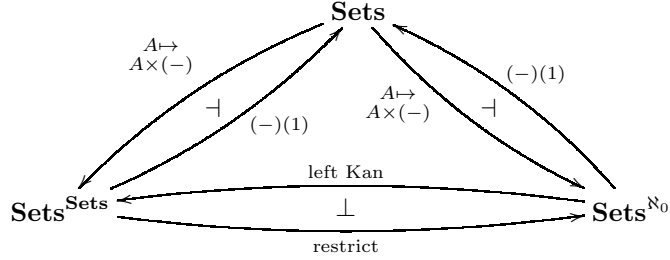

Now consider the triangle in Figure~\ref{SetsTriangle}, with functor
categories at the two bottom corners.  We briefly explain the arrows
(functors) in this diagram. The downward arrows
$\Sets\rightarrow\Sets^{\Sets}$ and $\Sets\rightarrow\Sets^{\aleph_0}$
describe the functors that map a set $A\in\Sets$ to the functor $X
\mapsto A\times X$. In the other, upward direction right adjoints are
given by the functors $(-)(1)$ describing ``evaluate at unit 1'', that
is $F\mapsto F(1)$. At the bottom the inclusion $\aleph_{0}
\hookrightarrow \Sets$ induces a functor $\Sets^{\Sets} \rightarrow
\Sets^{\aleph_0}$ by restriction: $F$ is mapped to the functor
$n\mapsto F(n)$. In the reverse direction a left adjoint is obtained
by left Kan extension~\cite[Ch.~X]{MacLane71}. Explicitly, this left
adjoint maps a functor $F\colon\aleph_{0}\rightarrow\Sets$ to the
functor $\mathcal{L}(F)\colon\Sets\rightarrow\Sets$ given by:
$$\begin{array}{rcl}
\mathcal{L}(F)(X) 
& = &
\Big(\coprod_{i\in\NNO}F(i)\times X^{i}\Big)/\!\sim,
\end{array}$$

\noindent where $\sim$ is the least equivalence relation such that,
for each $f\colon n\rightarrow m$ in $\aleph_0$,
$$\begin{array}{rcl}
\kappa_m(F(f)(a), v) 
& \sim &
\kappa_n(a, v\after f),
\qquad\mbox{where }a\in F(n)\mbox{ and }v\in X^{m}.
\end{array}$$

\auxproof{ 
We first prove the adjunction
  $\Sets\leftrightarrows\Sets^{\aleph_0}$.  It involves a bijective
  correspondence:
$$\begin{bijectivecorrespondence}
\correspondence[in $\Sets^{\aleph_0}$]
   {\xymatrix{A\times(-)\ar[r]^-{\sigma} & F}}
\correspondence[in \Sets]
   {\xymatrix{A\ar[r]_-{f} & F(1)}}
\end{bijectivecorrespondence}$$

\noindent Given a natural transformation $\sigma$ one defines
$\overline{\sigma}(a) = \sigma_{1}(a,0)$, where $0\in 1$. In the
other direction, given $f$ one takes, for $a\in A$ and $i\in n$,
$$\begin{array}{rcl}
\overline{f}_{n}(a,i)
& = &
F\Big(1\stackrel{i}{\rightarrow}n\Big)(f(a)) \;\in\; F(n).
\end{array}$$

\noindent It is not hard to see that $\overline{f}$ is natural, since
for $g\colon n\rightarrow m$ in $\aleph_0$ we get:
$$\begin{array}{rcl}
\big(\overline{f}_{m} \after A\times g\big)(a, i)
& = &
\overline{f}_{m}(a, g(i)) \\
& = &
F\big(1\stackrel{g(i)}{\rightarrow}m\big)(f(a)) \\
& = &
F\big(1\stackrel{i}{\rightarrow}n\stackrel{g}{\rightarrow} m\big)(f(a)) \\
& = &
F(g)\Big(F\big(1\stackrel{i}{\rightarrow}n\big)(f(a))\Big) \\
& = &
\big(F(g) \after \overline{f}_{n}\big)(a,i).
\end{array}$$

\noindent Further,
$$\begin{array}{rcl}
\overline{\overline{\sigma}}_{n}(a,i)
& = &
F\big(1\stackrel{i}{\rightarrow}n\big)(\overline{\sigma}(a)) \\
& = &
F\big(1\stackrel{i}{\rightarrow}n\big)(\sigma_{1}(a,0)) \\
& = &
\sigma_{n}((A\times i)(a, 0)) \\
& = &
\sigma_{n}(a, i) \\
\overline{\overline{f}}(a)
& = &
\overline{f}_{1}(a,0) \\
& = &
F\big(1\stackrel{0}{\rightarrow}1\big)(f(a)) \\
& = &
f(a).
\end{array}$$

We turn to the adjunction $\Sets^{\Sets} \leftrightarrows \Sets^{\aleph_0}$.
It involves:
$$\begin{bijectivecorrespondence}
\correspondence[in $\Sets^{\Sets}$]
   {\xymatrix{\mathcal{L}(F)\ar[r]^-{\sigma} & G}}
\correspondence[in $\Sets^{\aleph_0}$]
   {\xymatrix{F\ar[r]_-{\tau} & G}}
\end{bijectivecorrespondence}$$

\noindent where $\mathcal{L}(F)$ describes the left Kan extension described
above. Given $\sigma$ define $\overline{\sigma}_{n}\colon F(n)
\rightarrow G(n)$ by $\overline{\sigma}_{n}(a) =
\sigma_{n}([\kappa_{n}(a,\idmap_{n})])$, where $[\kappa_{n}(a,\idmap_{n})] 
\in \mathcal{L}(F)(n) = \big(\coprod_{i}F(i)\times n^{i}\big)/\!\sim$. This
yields a natural transformation, since for $f\colon n\rightarrow m$
in $\aleph_0$,
$$\begin{array}{rcl}
\big(G(f) \after \overline{\sigma}_{n}\big)(a)
& = &
G(f)\big(\sigma_{n}([\kappa_{n}(a,\idmap_{n})])\big) \\
& = &
\sigma_{m}\big(\mathcal{L}(F)(f)([\kappa_{n}(a,\idmap_{n})])\big) \\
& = &
\sigma_{m}([\kappa_{n}(a,f\after\idmap_{n})]) \\
& = &
\sigma_{m}([\kappa_{n}(a,\idmap_{m} \after f)]) \\
& = &
\sigma_{m}([\kappa_{m}(F(f)(a),\idmap_{m})]) \\
& = &
\big(\overline{\sigma}_{m} \after F(f)\big)(a).
\end{array}$$

\noindent In the other direction, given $\tau$ we take $\overline{\tau}_{X}
\colon \mathcal{L}(F)(X) \rightarrow G(X)$ by:
$$\begin{array}{rcl}
\overline{\tau}_{X}([\kappa_{i}(a,g)])
& = &
G\big(i\stackrel{g}{\rightarrow}X\big)(\tau_{i}(a)) \;\in\; G(X).
\end{array}$$

\noindent This yields a natural transformation since for $f\colon
X\rightarrow Y$ in $\Sets$ we have:
$$\begin{array}{rcl}
\big(G(f) \after \overline{\tau}_{X}\big)([\kappa_{i}(a, g)])
& = &
G(f)\Big(G\big(i\stackrel{g}{\rightarrow}X\big)(\tau_{i}(a))\Big) \\
& = &
\Big(G\big(i\stackrel{f\after g}{\rightarrow}Y\big)(\tau_{i}(a))\Big) \\
& = &
\overline{\tau}_{Y}([\kappa_{i}(a,f\after g)]) \\
& = &
\big(\overline{\tau}_{Y} \after \mathcal{L}(F)(f)\big)([\kappa_{i}(a,g)]).
\end{array}$$

\noindent Finally,
$$\begin{array}{rcl}
\overline{\overline{\sigma}}_{X}([\kappa_{i}(a,g)])
& = &
G\big(i\stackrel{g}{\rightarrow}X\big)(\overline{\sigma}_{i}(a)) \\
& = &
G\big(i\stackrel{g}{\rightarrow}X\big)(\sigma_{i}([\kappa_{i}(a,\idmap_{i})])) \\
& = &
\sigma_{n}(\mathcal{L}(F)(g)([\kappa_{i}(a,\idmap_{i})])) \\
& = &
\sigma_{n}([\kappa_{i}(a,g)]) \\
\overline{\overline{\tau}}_{n}(a)
& = &
\overline{\tau}_{n}([\kappa_{n}(a,\idmap_{n})]) \\
& = &
G\big(n\stackrel{\idmap}{\rightarrow}n\big)(\tau_{n}(a)) \\
& = &
\tau_{n}(a).
\end{array}$$
}

\noindent The adjunction on the left in Figure ~\ref{SetsTriangle} is
then in fact the composition of the other two. The adjunctions in
Figure~\ref{SetsTriangle} are not new. For instance, the one at the
bottom plays an important role in the description of analytic functors
and species~\cite{Joyal86}, see
also~\cite{Hasegawa02,AdamekV08,Curien08}. The category of presheaves
$\Sets^{\aleph_0}$ is used to provide a semantics for binding,
see~\cite{FiorePT99}. What is new in this paper is the systematic
organisation of correspondences in triangles like the one in
Figure~\ref{SetsTriangle} for various kinds of algebraic structures
(instead of sets).
\begin{itemize}
\item There is a triangle of adjunctions for monoids, monads, and
  Lawvere theories, see Figure~\ref{MonoidTriangleFig}.

\item This triangle restricts to commutative monoids, commutative
  monads, and symmetric monoidal Lawvere theories, see
  Figure~\ref{ComMonoidTriangleFig}.

\item There is also a triangle of adjunctions for commutative
  semirings, commutative additive monads, and symmetric monoidal
  Lawvere theories with biproducts, see Figure~\ref{CSRngTriangleFig}.

\item This last triangle restricts to involutive commutative
  semirings, involutive commutative additive monads, and dagger
  symmetric monoidal Lawvere theories with dagger biproducts, see
  Figure~\ref{ICSRngTriangleFig} below.
\end{itemize}

\noindent These four figures with triangles of adjunctions provide a
quick way to get an overview of the paper (the rest is just hard
work). The triangles capture fundamental correspondences between basic
mathematical structures. As far as we know they have not been made
explicit at this level of generality.

The paper is organised as follows. It starts with a section containing
some background material on monads and Lawvere theories. The triangle
of adjunctions for monoids, much of which is folklore, is developed in
Section~\ref{MonoidSec}.  Subsequently, Section~\ref{AMndSec} forms an
intermezzo; it introduces the notion of additive monad, and proves
that a monad $T$ is additive if and only if in its Kleisli category
$\Kl(T)$ coproducts form biproducts, if and only if in its
category $\Alg(T)$ of algebras products form
biproducts. These additive monads play a crucial role in
Sections~\ref{SemiringMonadSec} and~\ref{Semiringcatsec} which develop
a triangle of adjunctions for commutative semirings. Finally,
Section~\ref{InvolutionSec} introduces the refined triangle with
involutions and daggers.

The triangles of adjunctions in this paper are based on many detailed
verifications of basic facts. We have chosen to describe all
constructions explicitly but to omit most of these verifications,
certainly when these are just routine. Of course, one can continue and
try to elaborate deeper (categorical) structure underlying the
triangles. In this paper we have chosen not to follow that route, but
rather to focus on the triangles themselves.

\section{Preliminaries}\label{PrelimSec}

We shall assume a basic level of familiarity with category theory,
especially with adjunctions and monads. This section recalls some
basic facts and fixes notation. For background information we refer
to~\cite{Awodey06,Borceux94,MacLane71}.

In an arbitrary category \Cat{C} we write finite products as
$\times,1$, where $1\in\cat{C}$ is the final object. The projections
are written as $\pi_{i}$ and tupling as $\tuple{f_{1}}{f_{2}}$. Finite
coproducts are written as $+$ with initial object $0$, and with
coprojections $\kappa_i$ and cotupling $[f_{1},f_{2}]$. We write $!$,
both for the unique map $X \to 1$ and the unique map $0 \to X$. A
category is called distributive if it has both finite products and
finite coproducts such that functors $X\times(-)$ preserve these
coproducts: the canonical maps $0\rightarrow X\times 0$, and $(X\times
Y)+(X\times Z) \rightarrow X\times (Y+Z)$ are isomorphisms. Monoidal
products are written as $\otimes, I$ where $I$ is the tensor unit,
with the familiar isomorphisms: $\alpha\colon X\otimes (Y\otimes Z)
\congrightarrow (X\otimes Y)\otimes Z$ for associativity, $\rho\colon
X\otimes I\congrightarrow X$ and $\lambda\colon I\otimes
X\congrightarrow X$ for unit, and in the symmetric case also
$\gamma\colon X\otimes Y\congrightarrow Y\otimes X$ for swap. 

We write $\Mnd(\Cat{C})$ for the category of monads on a category
\Cat{C}. For convenience we write $\Mnd$ for $\Mnd(\Sets)$. Although
we shall use strength for monads mostly with respect to finite
products $(\times, 1)$ we shall give the more general definition
involving monoidal products $(\otimes, I)$. A monad $T$ is called
strong if it comes with a `strength' natural transformation $\st$ with
components $\st\colon T(X)\otimes Y\rightarrow T(X\otimes Y)$,
commuting with unit $\eta$ and multiplication $\mu$, in the sense that
$\st \after \eta\otimes\idmap = \eta$ and $\st \after \mu\otimes\idmap
= \mu \after T(\st) \after \st$. Additionally, for the familiar
monoidal isomorphisms $\rho$ and $\alpha$,
$$\hspace*{-1em}\xymatrix@C-1pc{
T(Y)\otimes I\ar[r]^-{\st}\ar[dr]_{\rho} & T(Y\otimes I)\ar[d]^{T(\rho)} 
&
T(X)\otimes (Y\otimes Z)\ar[rr]^-{\st} \ar[d]_-{\alpha} & & 
   T(X\otimes (Y\otimes Z)) \ar[d]^{T(\alpha)}  \\
& T(Y)
&
(T(X)\otimes Y)\otimes Z\ar[r]^-{\st\otimes\idmap}&
  T(X\otimes Y)\otimes Z\ar[r]^-{\st} & 
  T((X\otimes Y)\otimes Z)
}$$

\noindent Also, when the tensor $\otimes$ is a cartesian product
$\times$ we sometimes write these $\rho$ and $\alpha$ for the obvious
maps.

The category $\StMnd(\Cat{C})$ has monads with strength $(T,\st)$ as
objects. Morphisms are monad maps commuting with strength.  The
monoidal structure on \Cat{C} is usually clear from the context.

\begin{lemma}
\label{SetsStrengthLem}
Monads on \Sets are always strong w.r.t.\ finite products, in a canonical
way, yielding a functor $\Mnd(\Sets) = \Mnd \rightarrow \StMnd =
\StMnd(\Sets)$.
\end{lemma}

\begin{proof}
For every functor $T\colon\Sets\rightarrow\Sets$, there exists a
strength map $\st\colon T(X)\times Y \rightarrow T(X\times Y)$, namely
$\st(u,y) = T(\lam{x}{\tuple{x}{y}})(u)$. It makes the above diagrams
commute, and also commutes with unit and multiplication in case $T$ is
a monad. Additionally, strengths commute with natural transformations
$\sigma\colon T\rightarrow S$, in the sense that $\sigma \after \st =
\st \after (\sigma\times\idmap)$. \QED
\end{proof}

\auxproof{
For $u\in T(X)$ and $y\in Y$,
$$\begin{array}{rcl}
\big(\sigma \after \st\big)(u, y)
& = &
\sigma(T(\lam{x}{\tuple{x}{y}})(u)) \\
& = &
\big(\sigma \after T(\lam{x}{\tuple{x}{y}})\big)(u) \\
& = &
\big(S(\lam{x}{\tuple{x}{y}}) \after \sigma\big)(u) \\
& = &
\st(\sigma(u), y) \\
& = &
\big(\st \after (\sigma\times\idmap)\big)(u,y).
\end{array}$$
}

Given a general strength map $\st\colon T(X)\otimes Y \rightarrow
T(X\otimes Y)$ in a \textit{symmetric} monoidal category one can
define a swapped $\st'\colon X\otimes T(Y) \rightarrow T(X\otimes Y)$
as $\st' = T(\gamma) \after \st \after \gamma$, where $\gamma \colon
X\otimes Y \congrightarrow Y\otimes X$ is the swap map.  There are now
in principle two maps $T(X)\otimes T(Y) \rightrightarrows T(X\otimes
Y)$, namely $\mu \after T(\st') \after \st$ and $\mu \after T(\st)
\after \st'$. A strong monad $T$ is called commutative if these two
composites $T(X)\otimes T(Y) \rightrightarrows T(X\otimes Y)$ are the
same. In that case we shall write $\dst$ for this (single) map, which
is a monoidal transformation, see also~\cite{Kock71a}. The powerset
monad $\powerset$ is an example of a commutative monad, with $\dst
\colon \powerset(X)\times\powerset(Y)\rightarrow \powerset(X\times Y)$
given by $\dst(U,V) = U\times V$. Later we shall see other examples.

We write $\Kl(T)$ for the Kleisli category of a monad $T$, with
$X\in\Cat{C}$ as objects, and maps $X\rightarrow T(Y)$ in \Cat{C} as
arrows. For clarity we sometimes write a fat dot $\klafter$ for
composition in Kleisli categories, so that $g \klafter f = \mu \after
T(g) \after f$. The inclusion functor $\Cat{C}\rightarrow \Kl(T)$ is
written as $J$, where $J(X) = X$ and $J(f) = \eta \after f$. A map of
monads $\sigma\colon T\rightarrow S$ yields a functor
$\Kl(\sigma)\colon \Kl(T) \rightarrow \Kl(S)$ which is the identity on
objects, and maps an arrow $f$ to $\sigma\after f$. This functor
$\Kl(\sigma)$ commutes with the $J$'s. One obtains a functor $\Kl
\colon \Mnd(\cat{C}) \to \Cat{Cat}$, where $\cat{Cat}$ is the category
of (small) categories.

We will use the following standard result.

\begin{lemma}
\label{KleisliStructLem}
For $T\in\Mnd(\Cat{C})$, consider the generic statement ``if $\Cat{C}$
has $\diamondsuit$ then so does $\Kl(T)$ and $J\colon\Cat{C}
\rightarrow \Kl(T)$ preserves $\diamondsuit$'s'', where $\diamondsuit$ is
some property. This holds for:
{\renewcommand{\theenumi}{(\roman{enumi})}
\begin{enumerate}
\item $\diamondsuit$ = (finite coproducts $+, 0$), or in fact any colimits;

\item $\diamondsuit$ = (monoidal products $\otimes, I$), in case
the monad $T$ is commutative;

\end{enumerate}}
\end{lemma}

\begin{proof}
Point \textit{(i)} is obvious; for \textit{(ii)}
one defines the tensor on morphisms in $\Kl(T)$ as:
$$\begin{array}{rcl}
\big(X\stackrel{f}{\rightarrow} T(U)\big) \otimes 
  \big(Y\stackrel{g}{\rightarrow} T(V)\big)
& = &
\big(X\otimes Y \stackrel{f\otimes g}{\longrightarrow} T(U)\otimes T(V)
   \stackrel{\dst}{\longrightarrow} T(U\otimes V)\big).
\end{array}$$

\noindent Then: $J(f)\otimes J(g) = \dst \after ((\eta \after
f)\otimes (\eta \after g)) = \eta \after (f\otimes g) = J(f\otimes
g)$. \QED
\end{proof}

\auxproof{
Proof of the fact that $\Kl(\sigma)$ preserves coproducts:

For $\kappa_1 \colon X \to X+Y$ in $\Kl(T)$,
	\[\Kl(\sigma)(\kappa_1) = \sigma_{X+Y} \after \eta^T \after \kappa_1 = \eta^S \after \kappa_1\]
as $\sigma$ commutes with $\eta$.

For $f\colon X \to Z$, $g\colon Y \to Z$ in $\Kl(T)$, i.e. $f\colon X \to T(Z)$, $g\colon Y \to T(Z)$ in $\cat{C}$, 
	\[\Kl(\sigma)(\cotuple{f}{g}) = \sigma_Z \after \cotuple{f}{g} = \cotuple{\sigma_Z \after f}{\sigma_Z \after g} = \cotuple{\Kl(\sigma)(f)}{\Kl(\sigma)(g)}.\]

The functor $\Kl(\sigma)$ preserves monoidal products since:
$$\begin{array}{rcl}
\Kl(\sigma)(f) \otimes \Kl(\sigma)(g)
& = &
\dst \after ((\sigma \after f) \otimes (\sigma \after g)) \\
& \stackrel{(*)}{=} &
\sigma \after \dst \after (f\otimes g) \\
& = &
\Kl(\sigma)(f\otimes g),
\end{array}$$

\noindent where the marked equality holds since $\sigma$ commutes
with strength:
$$\begin{array}{rcll}
\sigma \after \dst
& = &
\sigma \after \mu \after T(\st') \after \st \\
& = &
\mu \after \sigma \after T(\sigma) \after T(\st') \after \st 
   & \mbox{$\sigma$ is a map of monads} \\
& = &
\mu \after \sigma \after T(\st' \after \idmap\otimes\sigma) \after \st 
   & \mbox{$\sigma$ commutes with strength} \\
& = &
\mu \after S(\st') \after \sigma \after \st \after (\idmap\otimes\sigma) \\
& = &
\mu \after S(\st') \after \st \after (\sigma\otimes\idmap) 
   \after (\idmap\otimes\sigma) \\
& = &
\dst \after (\sigma\otimes\sigma).
\end{array}$$
	
}

As in this lemma we sometimes formulate results on monads in full
generality, \textit{i.e.}~for arbitrary categories, even though our
main results---see Figures~\ref{MonoidTriangleFig},
\ref{ComMonoidTriangleFig}, \ref{CSRngTriangleFig}
and~\ref{ICSRngTriangleFig}---only deal with monads on \Sets. These
results involve algebraic structures like monoids and semirings, which
we interpret in the standard set-theoretic universe, and not in
arbitrary categories. Such greater generality is possible, in
principle, but it does not seem to add enough to justify the
additional complexity.

Often we shall be interested in a ``finitary'' version of the Kleisli
construction, corresponding to the Lawvere
theory~\cite{Lawvere63a,HylandP07} associated with a monad. For a
monad $T\in\Mnd$ on $\Sets$ we shall write $\Kl_{\NNO}(T)$ for the
category with natural numbers $n\in\NNO$ as objects, regarded as
finite sets $\underline{n} = \{0,1,\ldots, n-1\}$. A map $f\colon
n\rightarrow m$ in $\Kl_{\NNO}(T)$ is then a function $\underline{n}
\rightarrow T(\underline{m})$. This yields a full inclusion
$\Kl_{\NNO}(T) \hookrightarrow \Kl(T)$. It is easy to see that a map
$f\colon n\rightarrow m$ in $\Kl_{\NNO}(T)$ can be identified with an
$n$-cotuple of elements $f_{i}\in T(m)$, which may be seen as $m$-ary
terms/operations.

By the previous lemma the category $\Kl_{\NNO}(T)$ has coproducts
given on objects simply by the additive monoid structure $(+, 0)$ on
natural numbers. There are obvious coprojections $n\rightarrow n+m$,
using $\underline{n+m} \cong \underline{n}+\underline{m}$.  The
identities $n+0 = n = 0+n$ and $(n+m)+k = n + (m+k)$ are in fact the
familiar monoidal isomorphisms. The swap map is an isomorphism $n+m
\cong m+n$ rather than an identity $n+m = m+n$.

In general, a Lawvere theory is a small category $\cat{L}$ with
natural numbers $n\in\NNO$ as objects, and $(+,0)$ on $\NNO$ forming
finite coproducts in $\cat{L}$. It forms a categorical version of a
term algebra, in which maps $n\rightarrow m$ are understood as
$n$-tuples of terms $t_i$ each with $m$ free variables. Formally a
Lawvere theory involves a functor $\aleph_{0}\rightarrow\cat{L}$ that
is the identity on objects and preserves finite coproducts ``on the
nose'' (up-to-identity) as opposed to up-to-isomorphism. A morphism of
Lawvere theories $F\colon\cat{L}\rightarrow\cat{L'}$ is a functor that
is the identity on objects and strictly preserves finite coproducts.
This yields a category \Law. 

\begin{corollary}
\label{Mnd2FCCatCor}
The finitary Kleisli construction $\Kl_{\NNO}$ for monads on $\Sets$,
yields a functor $\Kl_{\NNO}: \Mnd \rightarrow \Law$. \QED
\end{corollary}

\section{Monoids}\label{MonoidSec}

The aim of this section is to replace the category \Sets of sets at
the top of the triangle in Figure~\ref{SetsTriangle} by the category
\Mon of monoids $(M,\cdot,1)$, and to see how the corners at the
bottom change in order to keep a triangle of adjunctions. Formally,
this can be done by considering monoid objects in the three categories
at the corners of the triangle in Figure~\ref{SetsTriangle} (see
also~\cite{FiorePT99,Curien08}) but we prefer a more concrete
description. The results in this section, which are summarised in
Figure~\ref{MonoidTriangleFig}, are not claimed to be new, but are
presented in preparation of further steps later on in this paper.

\begin{figure}
$$\begin{array}{c}
{\xymatrix@R-.5pc@C+.5pc{
& & \Mon\ar@/_2ex/ [ddll]_{\cal A}\ar@/_2ex/ [ddrr]_(0.4){\Kl_{\NNO}\mathcal{A}} \\
& \dashv & & \dashv & \\
\Mnd\ar @/_2ex/[rrrr]_{\Kl_\NNO}
   \ar@/_2ex/ [uurr]_(0.6){\;{\Ev} \cong \mathcal{H}\Kl_{\NNO}} & & \bot & &  
   \Law\ar@/_2ex/ [uull]_{\mathcal{H}}\ar @/_2ex/[llll]_{\LM}
}} \\ \\[-1em]
\mbox{where}\quad
\left\{\begin{array}{ll}
{\cal A}(M) = M\times (-) & \mbox{action monad} \\
\Ev(T) = T(1) & \mbox{evaluation at singleton set 1} \\
\mathcal{H}(\Cat{L}) = \Cat{L}(1,1) & 
  \mbox{endo-homset of $1\in\Cat{L}$} \\
\Kl_{\NNO}(T) & \mbox{Kleisli category restricted to objects $n\in\NNO$} \\
\LM(\cat{L}) = T_{\cat{L}} & \mbox{monad associated with Lawvere theory \cat{L}.}
\end{array}\right.
\end{array}$$
\caption{Basic relations between monoids, monads and Lawvere theories.}
\label{MonoidTriangleFig}
\end{figure}

We start by studying the interrelations between monoids and monads.  In
principle this part can be skipped, because the adjunction on the left
in Figure~\ref{MonoidTriangleFig} between monoids and monads follows
from the other two by composition. But we do make this adjunction
explicit in order to completely describe the situation.

The following result is standard. We only sketch the proof.

\begin{lemma}
\label{Mon2MndLem}
Each monoid $M$ gives rise to a monad ${\cal A}(M) = M\times(-)\colon
\Sets \rightarrow \Sets$. The mapping $M\mapsto {\cal A}(M)$ yields
a functor $\Mon\rightarrow\Mnd$.
\end{lemma}

\begin{proof}
  For a monoid $(M,\cdot,1)$ the unit map $\eta \colon X\rightarrow M\times X =
  {\cal A}(M)$ is $x\mapsto (1,x)$. The multiplication $\mu \colon M\times
  (M\times X)\rightarrow M\times X$ is $(s,(t,x)) \mapsto (s\cdot
  t,x)$. The standard strength map $\st\colon (M\times X)\times Y
  \rightarrow M\times (X\times Y)$ is given by $\st((s,x),y) =
  (s,(x,y))$. Each monoid map $f\colon M\rightarrow N$ gives rise to a
  map of monads with components $f\times \idmap\colon M\times
  X\rightarrow N\times X$. These components commute with
  strength. \QED
\end{proof}

The monad $\mathcal{A}(M) = M\times(-)$ is called the `action monad',
as its category of Eilenberg-Moore algebras consists of $M$-actions
$M\times X\rightarrow X$ and their morphisms. The monoid elements act
as scalars in such actions. 

Conversely, each monad (on $\Sets$) gives rise to a monoid. In the
following lemma we prove this in more generality. For a category
$\cat{C}$ with finite products, we denote by
$\cat{Mon}(\cat{C})$ the category of monoids in $\cat{C}$,
\textit{i.e.}~the category of objects $M$ in $\cat{C}$ carrying a
monoid structure $1 \rightarrow M \leftarrow M \times M$ with
structure preserving maps between them.

\begin{lemma}
\label{Mnd2MonLem}
Each strong monad $T$ on a category \Cat{C} with finite products, gives rise to a monoid $\Ev(T) = T(1)$ in $\Cat{C}$. The
mapping $T \mapsto T(1)$ yields a functor $\StMnd(\Cat{C}) \to
\Mon(\Cat{C})$
\end{lemma}

\begin{proof}
  For a strong monad $(T, \eta, \mu, \st)$, we define a multiplication
  on $T(1)$ by $\mu \after T(\pi_2) \after \st \colon T(1) \times T(1)
  \to T(1)$, with unit $\eta_1 \colon 1 \to T(1)$. Each monad map
  $\sigma\colon T \to S$ gives rise to a monoid map $T(1) \to S(1)$ by
  taking the component of $\sigma$ at $1$. \QED
\end{proof}

\auxproof{
We check the unit laws diagrammatically:
$$\xymatrix@C-.5pc{
T(1)\ar[r]^-{\cong}\ar@{=}[ddr] & 
T(1)\times 1\ar[r]^-{\idmap\times\eta_1}\ar[d]^{\st} &
T(1)\times T(1)\ar[d]^{\st}
&
T(1)\ar[r]^-{\cong}\ar@{=}[ddr] & 
1\times T(1)\ar[r]^-{\eta_{1}\times\idmap}\ar[dr]_-{\eta}\ar[dd]^{\pi_2} &
T(1)\times T(1)\ar[d]^{\st} \\
& 
T(1\times 1)\ar[r]^-{T(\idmap\times\eta)}\ar[d]^{T(\pi_{2})} &
T(1\times T(1))\ar[d]^{T(\pi_{2})}
&
&
&
T(1\times T(1))\ar[d]^{T(\pi_{2})} \\
&
T(1)\ar[r]^-{T(\eta)}\ar@{=}[dr] &
T^{2}(1)\ar[d]^{\mu}
& 
&
T(1)\ar[r]^-{\eta}\ar@{=}[dr] &
T^{2}(1)\ar[d]^{\mu} \\
&
&
T(1)
&
&
&
T(1)
}$$

The associativity of multiplication follows from the following diagram.
$$\xymatrix{
(T(1)\times T(1))\times T(1)\ar[r]^-{\st\times\idmap}\ar[d]_{\alpha}^{\cong} &
   T(1\times T(1))\times T(1)\ar[r]^-{T(\pi_{2})\times\idmap}\ar[d]^{\st} &
   T^{2}(1)\times T(1)\ar[r]^-{\mu}\ar[d]^{\st} &
   T(1)\times T(1)\ar[ddd]^{\st} \\
T(1)\times (T(1)\times T(1))\ar[dd]_{\idmap\times\st} & 
   T((1\times T(1))\times T(1))\ar[r]^-{T(\pi_{2}\times\idmap)}
      \ar[d]_{T(\alpha)}^{\cong} & 
   T(T(1)\times T(1))\ar[dd]^{T(\st)} & \\
& 
   T(1\times (T(1)\times T(1)))\ar[d]^{T(\idmap\times\st)} \\
T(1)\times T(1\times T(1))\ar[d]_{\idmap\times T(\pi_{2})}\ar[r]^-{\st} & 
   T(1\times T(1\times T(1)))\ar[d]_{T(\idmap\times\st)}\ar[r]^-{T(\pi_{2})} & 
   T^{2}(1\times T(1))\ar[r]^-{\mu}\ar[d]^{T^{2}(\pi_{2})} & 
   T(1\times T(1))\ar[d]^{T(\pi_{2})} \\
T(1)\times T^{2}(1)\ar[d]_{\idmap\times\mu}\ar[r]^-{\st} & 
   T(1\times T^{2}(1))\ar[d]^{T(\idmap\times\mu)}\ar[r]^-{T(\pi_{2})} & 
   T^{3}(1)\ar[r]^-{\mu}\ar[d]^{T(\mu)} & 
   T^{2}(1)\ar[d]^-{\mu} \\
T(1)\times T(1)\ar[r]_-{\st} & 
   T(1\times T(1))\ar[r]_-{T(\pi_{2})} & 
   T^{2}(1)\ar[r]_-{\mu} & T(1)
}$$
}

The swapped strength map $\st'$ gives rise to a swapped multiplication
on $T(1)$, namely $\mu \after T(\pi_1) \after \st' \colon T(1) \times
T(1) \to T(1)$, again with unit $\eta_1$. It corresponds to $(a,b)
\mapsto b\cdot a$ instead of $(a,b)\mapsto a\cdot b$ like in the
lemma. In case $T$ is a commutative monad, the two multiplications
coincide as we prove in Lemma \ref{CMnd2CMonLem}.

The functors defined in the previous two Lemmas~\ref{Mon2MndLem} and
\ref{Mnd2MonLem} form an adjunction. This result goes back
to~\cite{Wolff73}.

\begin{lemma}
\label{AdjMndMonLem}
The pair of functors $\mathcal{A} \colon \Mon \rightleftarrows \Mnd
\colon \Ev$ forms an adjunction $\mathcal{A} \dashv \Ev$, as on the
left in Figure~\ref{MonoidTriangleFig}.
\end{lemma}

\begin{proof}
For a monoid $M$ and a (strong) monad $T$ on \Sets there are
(natural) bijective correspondences:
$$\begin{bijectivecorrespondence}
  \correspondence[in \Mnd]{\xymatrix{\mathcal{A}(M)\ar[r]^-{\sigma} & T}}
  \correspondence[in \Mon]{\xymatrix{M\ar[r]_-{f} & T(1)}}
\end{bijectivecorrespondence}$$

\noindent Given $\sigma$ one defines a monoid map $\overline{\sigma}
\colon M\rightarrow T(1)$ as:
$$\xymatrix{
\overline{\sigma} = \Big(M\ar[r]^-{\rho^{-1}}_-{\cong} & 
M\times 1 = \mathcal{A}(M)(1)\ar[r]^-{\sigma_1} & T(1)\Big),
}$$

\noindent where $\rho^{-1} = \tuple{\idmap}{!}$ in this cartesian case. 
Conversely, given $f$ one gets a monad map $\overline{f}
\colon \mathcal{A}(M) \rightarrow T$ with components:
$$\xymatrix{
\overline{f}_{X} = \Big(M\times X\ar[r]^-{f\times\idmap} &
T(1)\times X\ar[r]^-{\st} & 
T(1\times X)\ar[r]^-{T(\lambda)}_-{\cong} & T(X)\Big),
}$$

\noindent where $\lambda = \pi_{2} \colon 1\times X\congrightarrow
X$. Straightforward computations show that these assignments indeed
give a natural bijective correspondence. \QED
\end{proof}

\auxproof{
We briefly check the bijective correspondences:
$$\begin{array}{rcll}
\overline{\overline{\sigma}}_{X}
& = &
T(\pi_{2}) \after \st \after (\overline{\sigma}\times\idmap) \\
& = &
T(\pi_{2}) \after \st \after (\sigma_{1}\times\idmap) \after
   (\tuple{\idmap}{!}\times\idmap) \\
& = &
T(\pi_{2}) \after \sigma_{1\times X} \after \st \after 
   (\tuple{\idmap}{!}\times\idmap)
   & \mbox{since $\sigma$ commutes with strength} \\
& = &
\sigma_{X} \after (M\times\pi_{2}) \after \st \after 
   (\tuple{\idmap}{!}\times\idmap)
   & \mbox{by naturality of $\sigma$} \\
& \stackrel{(*)}{=} &
\sigma_{X} \after \idmap \\
& = &
\sigma_{X},
\end{array}$$

\noindent where the marked equation holds since for $a\in M$
and $x\in X$
$$\begin{array}{rcl}
\big((M\times\pi_{2}) \after \st \after 
   (\tuple{\idmap}{!}\times\idmap)\big)(a,x)
& = &
\big((M\times\pi_{2}) \after \st\big)((a,*),x) \\
& = &
(M\times\pi_{2})(a,(*,x)) \\
& = &
(a,x).
\end{array}$$

\noindent Next,
$$\begin{array}{rcll}
\overline{\overline{f}}
& = &
\overline{f}_{1} \after \tuple{\idmap}{!} \\
& = &
T(\pi_{2}) \after \st \after (f\times\idmap)\after \tuple{\idmap}{!} \\
& = &
T(\pi_{1}) \after \st \after (f\times\idmap)\after \tuple{\idmap}{!}
   & \mbox{since $\pi_{1} = \pi_{2} \colon 1\times 1\rightarrow 1$} \\
& = &
\pi_{1} \after (f\times\idmap)\after \tuple{\idmap}{!}
   & \mbox{by~\eqref{StrengthMonoidal}} \\
& = &
f \after \pi_{1} \after \tuple{\idmap}{!} \\
& = &
f.
\end{array}$$

For $f \colon M \to T(1)$, $\overline{f}$ is a monad morphism $\mathcal{A}(M) \to T$:\\
\begin{itemize}
	\item Naturality: let $h \colon X \to Y$ be a set-map
	$$
	\begin{array}{rcl}
		T(h) \after \overline{f}_X &=&T(h) \after T(\pi_2) \after st \after f \times id \\
		&=&
		T(\pi_2) \after T(id \times h) \after st \after f \times id \\
		&=&
		T(\pi_2) \after st \after id \times h \after f \times id\\
		&=&
		T(\pi_2) \after st \after f\times id \after id \times h\\
		&=&
		\overline{f}_Y \after \mathcal{A}(M)(h)
	\end{array}
	$$
	
	\item Commutativity with $\eta$:
	\xymatrix{
	X \ar[r]^{\eta} \ar[d]_{\lambda^{-1}} & M \times X \ar[d]^{f \times id}\\
	1 \times X \ar[r]^{\eta \times id} \ar[dd]_{\lambda} \ar[rd]_{\eta} & T(1) \times X \ar[d]^{st} \\
	&T(1 \times X) \ar[d]^{T(\pi_2) = T(\lambda)}\\
	X \ar[r]_{\eta} &T(X)
	}
	
	\item Commutativity with $\mu$: \\
	\begin{sideways}
        \xymatrix{
	M\times(M\times X) \ar[dddddd]_{\mu}\ar[rr]^{id \times (f \times id)} && M \times (T(1) \times X) \ar[d]_{f \times id} \ar[rr]^{id \times \st} && M \times T(1 \times X) \ar[d]_{f \times id} \ar[r]^{id \times T(\pi_2)} & M \times T(X) \ar[d]^{f \times id}\\
	&& T(1)\times(T(1)\times X) \ar[rr]^{id \times st} \ar[dr]^{\st} \ar[dl]_{\alpha} && T(1) \times T(1 \times X) \ar[r]^{id \times T(\pi_2)} \ar[d]{\st} & T(1) \times T(X) \ar[d]^{\st}\\
	& (T(1) \times T(1))\times X \ar[d]_{\st} && T(1 \times (T(1) \times X)) \ar[ddl]^{T(\alpha)} \ar[ddd]^{T(\pi_2)} \ar[r]^{T(id \times \st)} & T(1 \times T(1 \times X)) \ar[r]^{T(id \times \pi_2)} \ar[ddd]_{T(\pi_2)} & T(1 \times T(X)) \ar[ddd]^{T(\pi_2)} \\
	& T(1 \times T(1)) \times X \ar[d]_{T(\pi_2) \times id}\ar[dr]^{\st} &&&& \\
	& T^2(1) \times X \ar[dd]_{\mu \times id} \ar[rrd]^{\st} & T((1 \times T(1)) \times X) \ar[dr]^{T(\pi_2 \times id)} &&& \\
	&&&T(T(1) \times X) \ar[r]^{T(\st)} & T^2(1 \times X) \ar[r]^{T^2(\pi_2)} \ar[d]_{\mu} & T^2(X) \ar[d]^{\mu}\\
	M \times X \ar[r]^{f \times id} & T(1) \times X \ar[rrr]^{\st} &&& T(1 \times X) \ar[r]^{T(\pi_2)}& T(X)
	}  
	\end{sideways}
\end{itemize}	

For $\sigma: \mathcal{A}(M) \to T$, $\overline{\sigma}$ is a monoid morphism:
\begin{itemize}
	\item Preservation of 1 follows from the fact that $\sigma$ commutes with $\eta_1$ \\
	\item Preservation of the multiplication: \\
	\xymatrix{
	M \times M \ar[r] \ar[ddd]_{\cdot} \ar[rdd] & (M\times 1) \times (M \times 1) \ar[d]_{\st} \ar[r]^{\sigma \times id} & T(1) \times (M\times 1) \ar[r]^{id \times \sigma} \ar[d]_{\st} & T(1) \times T(1) \ar[d]^{\st} \\
	&M \times (1 \times (M \times 1)) \ar[r]^{\sigma} \ar[d]_{id \times \pi_2} & T(1 \times (M \times 1)) \ar[d]_{T(\pi_2)} \ar[r]^{T(id \times \sigma)} & T(1 \times T(1)) \ar[d]^{T(\pi_2)}\\
	& M \times (M \times 1) \ar[r]^{\sigma} \ar[d]_{\mu} & T(M \times 1) \ar[r]^{T(\sigma)} & T^2(1) \ar[d]^{\mu}\\
	M \ar[r] & M \times 1 \ar[rr]^{\sigma} &&T(1)
	}
\end{itemize}

For naturality, consider:
$$\begin{bijectivecorrespondence}
  \correspondence[in \Mnd]{\xymatrix{\mathcal{A}(N)\ar[r]^-{\mathcal{A}(f)} & \mathcal{A}(M) \ar[r]^-{\sigma} & T \ar[r]^-{\tau} & S}}
  \correspondence[in \Mon]{\xymatrix{N\ar[r]_-{f} & M \ar[r]_-{\overline{\sigma}} & T(1) \ar[r]_-{\tau_1} & S(1)}}
\end{bijectivecorrespondence}$$

Then,
$$
\begin{array}{rcl}
	\overline{\tau \after \sigma \after \mathcal{A}(f)} &=& (\tau \after \sigma \after \mathcal{A}(f))_1 \after \tuple{id}{!} \\
	&=&
	\tau_1 \after \sigma_1 \after (\mathcal{A}(f))_1 \after \tuple{id}{!}\\
	&=&
	\tau_1 \after \sigma_1 \after f \times id \after \tuple{id}{!}\\
	&=&
	\tau_1 \after \sigma_1 \after \tuple{id}{!} \after f \\
	&=&
	\tau_1 \after \overline{\sigma} \after f
\end{array}
$$
}

Notice that, for a monoid $M$, the counit of the above adjunction is
the projection $\smash{(\Ev \after \mathcal{A})(M) = \mathcal{A}(M)(1)
  = M\times 1 \stackrel{\cong}{\rightarrow} M}$. Hence the adjunction
is a reflection.

We now move to the bottom of Figure \ref{MonoidTriangleFig}. The
finitary Kleisli construction yields a functor from the category of
monads to the category of Lawvere theories (Corollary
\ref{Mnd2FCCatCor}). This functor has a left adjoint, as is proven in
the following two standard lemmas.

\begin{lemma}
\label{GLaw2MndLem}
Each Lawvere theory $\cat{L}$, gives rise to a monad $T_{\cat{L}}$ on
$\Sets$, which is defined by
\begin{equation}
\label{LMEqn}
\begin{array}{rcl}
	T_{\cat{L}}(X) 
	& = &
	\Big(\coprod_{i\in\NNO}\cat{L}(1,i)\times X^{i}\Big)/\!\sim,
\end{array}
\end{equation}

\noindent where $\sim$ is the least equivalence relation such that,
for each $f\colon i\rightarrow m$ in $\aleph_{0} \hookrightarrow
\cat{L}$,
$$\begin{array}{rcl}
	\kappa_m (f \after g, v) 
	& \sim &
	\kappa_i(g, v\after f),
	\qquad\mbox{where }g\in \cat{L}(1,i)\mbox{ and }v\in X^{m}.
\end{array}$$

\noindent Finally, the mapping $\cat{L} \to T_{\cat{L}}$ yields a
functor $\LM\colon\Law \to \Mnd$.
\end{lemma}

\begin{proof}
For a Lawvere theory $\cat{L}$, the unit map $\eta \colon X \to
T_{\cat{L}}(X) = \Big(\coprod_{i\in\NNO}\cat{L}(1,i)\times
X^{i}\Big)/\!\sim$ is given by
$$
\begin{array}{rcl}
 x &\mapsto& [\kappa_1(id_1,x)].
\end{array}
$$  
The multiplication $\mu\colon
T_{\cat{L}}^2(X)\rightarrow T_{\cat{L}}(X)$ is given by:
$$\begin{array}{rcl}
\mu([\kappa_{i}(g,v)])
& = &
[\kappa_{j}((g_{0}+\cdots+g_{i-1})\after g, [v_{0},\ldots, v_{i-1}])] \\
& & \qquad \mbox{where }g\colon 1\rightarrow i, \mbox{ and } 
    v\colon i\rightarrow T_{\cat{L}}(X) \mbox{ is written as} \\
& & \qquad\qquad v(a) = \kappa_{j_{a}}(g_{a}, v_{a}), \mbox{ for }a<i, \\
& & \qquad \mbox{and } j = j_{0} + \cdots + j_{i-1}.
\end{array}$$

\noindent It is straightforward to show that this map $\mu$ is
well-defined and that $\eta$ and $\mu$ indeed define a monad structure
on $T_{\cat{L}}$.

For each morphism of Lawvere theories $F\colon \cat{L} \to \cat{K}$,
one may define a monad morphism $\LM(F) \colon T_{\cat{L}} \to
T_{\cat{K}}$ with components $\LM(F)_{X} \colon [\kappa_i(g,v)] \mapsto
[\kappa_i(F(g),v)]$. This yields a functor $\LM \colon \Law \to
\Mnd$. Checking the details is left to the reader.\QED
\end{proof}

\auxproof{
\begin{enumerate}
\item $\mu$ is well-defined\\
Let $f \colon i \to m$ in $\aleph_0$, $g \in \cat{L}(1,i)$, $v \colon m \to \LM(\cat{L})(X)$.\\
Write $v(a) = \kappa_{j_a}(h_a,v_a)$
$$
\begin{array}{rcll}
\lefteqn{\mu([\kappa_m(f \after g,v)])}\\ 
&=& [\kappa_{j_0 + j_{m-1}}((h_0 + \ldots h_{m-1}) \after (f \after g), [v_0, \ldots v_{m-1}])]\\
&=&
[\kappa_{j_0 + j_{m-1}}([\kappa_{f(0)}, \ldots \kappa_{f(i-1)}] \after (h_{f(0)} + \ldots h_{f(i-1)}) \after g), [v_0, \ldots v_{m-1}])]\\
&=&
[\kappa_{j_{f(0)} + \ldots +j_{f(i-1)}}((h_{f(0)} + \ldots h_{f(i-1)}) \after g, [v_0, \ldots v_{m-1}] \after [\kappa_{f(0)}, \ldots \kappa_{f(i-1)}])] &\text{eq.rel.}\\
&=&
[\kappa_{j_{f(0)} + \ldots +j_{f(i-1)}}((h_{f(0)} + \ldots h_{f(i-1)}) \after g, [v_{f(0)}, \ldots v_{f(m-1)}])]\\
&=&
\mu([\kappa_i(g,v \after f)])
\end{array}
$$
\item $\mu \after \eta \colon \LM(\cat{L})(X) \to \LM(\cat{L})(X) = id$\\
$$
\begin{array}{rcll}
(\mu \after \eta)([\kappa_i(g,v)]) &=& \mu([\kappa_1(\id_1,[\kappa_i(g,v)])]) \\
&=& [\kappa_i(g \after \id,v)] = [\kappa_i(g,v)]
\end{array}
$$
\item $\mu \after T(\eta) \colon \LM(\cat{L})(X) \to \LM(\cat{L})(X) = id$\\
$$
\begin{array}{rcll}
(\mu \after T\eta)([\kappa_i(g,v)]) &=& \mu([\kappa_i(g, \eta \after v)])\\
&=&
[\kappa_{1+\ldots+1}((\id+\ldots+\id)\after g, [v_0 + \ldots v_{i-1}])]\\
&=&
[\kappa_i(g,v)]
\end{array}
$$

\item $\mu \after T\mu = \mu \after \mu \colon \LM(\cat{L})^3(X) \to \LM(\cat{L})(X)$\\
Let $[\kappa_i(g,v)] \in \LM(\cat{L})^3(X)$, write
$$
\begin{array}{rcl}
	v(a) &=& [\kappa_{j_a}(g_a, v_a)] \,\text{where}\, g_a \colon 1 \to j_a \,\text{and}\, v_a \colon j_a \to \LM(\cat{L})(X)
\end{array}
$$
$$
\begin{array}{rcll}
	v_a(b) &=& [\kappa_{m^{a}_b}(h^a_b, w^{a}_b)]
\end{array}
$$	
$$
\begin{array}{rcll}
\lefteqn{(\mu \after T\mu)([\kappa_i(g,v)])}\\ 
&=& \mu([\kappa_i(g, \mu \after v)])\\
&=&
\mu([\kappa_i(g, \lambda a.[\kappa_{m^a_0+\ldots+m^a_{j_a-1}}((h^a_0+\ldots h^a_{j_a-1})\after g_a, [w^a_0,\ldots,w^a_{j_a-1}])])\\
&=&
[\kappa_{m^0_0+\ldots+m^0_{j_1-1} + \ldots +m^{i-1}_{j_{i-1}-1}}(((h^0_0 + \ldots h^0_{j_0-1})\after g_0 + \ldots)\after g, [[w^0_0, \ldots w^0_{j_0-1}], \ldots])\\
&=&
[\kappa_{m^0_0+\ldots+m^0_{j_1-1} + \ldots +m^{i-1}_{j_{i-1}-1}}(((h^0_0 + \ldots h^i_{j_i-1}) \after (g_0 + \ldots g_{i-1})\after g, [w^0_0, \ldots w^{i-1}_{j_{i-1}-1}])\\
&=&
\mu([\kappa_{j_0 + \ldots j_{i-1}}((g_0 + \ldots g_{i-1}) \after g, [v_0, \ldots v_{i-1}])])\\
&=&
(\mu \after \mu)([\kappa_i(g,v)])
\end{array}
$$
\item For $G \colon \cat{L} \to \cat{K}$, $\LM(G)$ is a monad morphism.\\
Preservation of $\eta$:
$$
\begin{array}{rcll}
(\LM(G) \after \eta)(x) &=& \LM(G)([\kappa_1(\id_1,x)])\\
&=& 
[\kappa_1(G(\id_1),x)]\\
&=&
[\kappa_1(\id_1,x)]\\
&=&
\eta(x)
\end{array}
$$
Commutates with $\mu$:
$$
\begin{array}{rcl}
\lefteqn{(\mu \after \LM(G) \after \LM(\cat{L})(\LM(G)))([\kappa_i(g,v)])}\\ 
&=& 
(\mu \after \LM(G))([\kappa_i(g,\LM(G)\after v)])\\ 
&=& 
\mu([\kappa_i(G(g),\LM(G)\after v)])\\ 
&=& 
[\kappa_{j_0+\ldots j_{i-1}}((G(g_0)+\ldots+G(g_{i-1}))\after G(g),[v_0,\ldots,v_{i-1}])]\\
&=&
[\kappa_{j_0+\ldots j_{i-1}}((G(g_0+\ldots+g_{i-1})\after g),[v_0,\ldots,v_{i-1}])]\\
&=&
\LM(G)([\kappa_{j_0+\ldots j_{i-1}}((g_0+\ldots+g_{i-1})\after g,[v_0,\ldots,v_{i-1}])])\\
&=&
(\LM(G)\after\mu)([\kappa_i(g,\LM(G)\after v)])
\end{array}
$$
Naturality:\\
For a map $h \colon X \to Y$ in $\Sets$, 
$$	\LM(\cat{L})(h) \colon \Big(\coprod_{i\in\NNO}\cat{L}(1,i)\times X^{i}\Big)/\!\sim \,\to \Big(\coprod_{i\in\NNO}\cat{L}(1,i)\times Y^{i}\Big)/\!\sim,$$ 
is given by
$$
	[\kappa_i(g,w)]  \mapsto  [\kappa_i(g,h \after w)].
$$
$$
\begin{array}{rcll}
(\LM(G)_Y \after \LM(\cat{L})(h))([\kappa_i(g,v)]) 
&=& 
(\LM(G)_Y([\kappa_i(g,h \after v)])\\
&=&
[\kappa_i(G(g),h \after v)]\\
&=&
\LM(\cat{K})(h)([\kappa_i(G(g),v)])\\
&=&
(\LM(\cat{K})(h) \after \LM(G)_X)([\kappa_i(g,v)])  
\end{array}
$$

\item Functoriality of $\LM$\\
$$
\begin{array}{rcll}
\LM(\id)_X([\kappa_i(g,v)]) &=& [\kappa_i(id(g),v)] = [\kappa_i(g,v)]
\end{array}
$$
and
$$
\begin{array}{rcll}
\LM(F \after G)_X([\kappa_i(g,v)]) = [\kappa_i((F \after G)(g),v)] = (\LM(F) \after \LM(G))([\kappa_i(g,v)])
\end{array}
$$
\end{enumerate}
}

\begin{lemma}
\label{AdjMndLvTLem}
The pair of functors $\LM \colon \Law \rightleftarrows \cat{Mnd}
\colon \Kl_{\NNO}$ forms an adjunction $\LM \dashv \Kl_{\NNO}$, as at
the bottom in Figure~\ref{MonoidTriangleFig}.
\end{lemma}

\begin{proof}
For a Lawvere theory $\cat{L}$ and a monad $T$ there are
(natural) bijective correspondences:
$$\begin{bijectivecorrespondence}
  \correspondence[in \Mnd]{\xymatrix{\LM(\cat{L})\ar[r]^-{\sigma} & T}}
  \correspondence[in \Law]{\xymatrix{\cat{L}\ar[r]_-{F} & \Kl_{\NNO}(T)}}
\end{bijectivecorrespondence}$$

\noindent Given $\sigma$, one defines a $\cat{\Law}$-map
$\overline{\sigma} \colon \cat{L} \to \Kl_{\NNO}(T)$ which is the
identity on objects and sends a morphism $f \colon n \to m$ in
$\cat{L}$ to the morphism
$$
\xymatrix{
n \ar[rrr]^-{\lam{i<n}{[\kappa_m(f \after \kappa_i,\id_m)]}} &&& 
   \LM(\cat{L})(m) \ar[r]^-{\sigma_m} & T(m)
}$$

\noindent in $\Kl_{\NNO}(T)$.

Conversely, given $F$, one defines a monad morphism $\overline{F}$
with components $\overline{F}_X \colon \LM(\cat{L})(X) \to T(X)$
given, for $i \in \NNO$, $g \colon 1 \to i \in \cat{L}$ and $v \in
X^i$, by:
$$
\begin{array}{rcll}
	[\kappa_i(g,v)] &\mapsto& (T(v) \after F(g))(*),
\end{array}
$$
where $*$ is the unique element of $1$. \QED
\end{proof}

\auxproof{
\begin{enumerate}
\item $\overline{\sigma}$ is a morphism in $\Law$\\
\begin{itemize}
\item $\overline{\sigma}$ is a functor\\
For $n \in \NNO$, $i \in n$,
$$
\begin{array}{rcll}
\overline{\sigma}(\id_n)(i) &=& \sigma_n ([\kappa_n(\id_n \after \kappa_i, \id_n)])\\
&=&
\sigma_n([\kappa_n(\kappa_i, \id_n)])\\
&=&
\sigma_n([\kappa_1(\id_1, \id_n \after \kappa_i)])&\text{def. equivalence rel.}\\
&=&
\sigma_n([\kappa_1(\id_1, i)])\\
&=&
(\sigma_n \after \eta )(i)\\
&=&
\eta(i)&\text{$\sigma$ preserves $\eta$}\\
&=&
\id_n^{\Kl_{\NNO}(T)}(i)
\end{array}
$$
For $n \xrightarrow{f} m \xrightarrow{g} k$,
$$
\begin{array}{rcll}
(\overline{\sigma}(g) \klafter \overline{\sigma}(f))(i) 
&=&
(\mu \after T(\overline{\sigma}(g)) \after \overline{\sigma}(f))(i)\\
&=&
(\mu \after T(\overline{\sigma}(g) \after \sigma_m)(\kappa_m(f \after \kappa_i, \id_m))\\
&=&
(\mu \after \sigma_{T(k)} \after \LM(\cat{L})(\overline{\sigma}(g)))(\kappa_m(f \after \kappa_i, \id_m)) \\
&=&
(\mu \after \sigma_{T(k)})(\kappa_m(f \after \kappa_i, \overline{\sigma}(g)))\\
&=&
(\mu \after \sigma_{T(k)})(\kappa_m(f \after \kappa_i, \sigma_k \after \lambda j \in m [\kappa_k(g,\kappa_j,\id_k)]))\\
&=&
(\mu \after \sigma_{T(k)}\after \LM(\cat{L})(\sigma))(\kappa_m(f \after \kappa_i, \lambda j \in m [\kappa_k(g,\kappa_j,\id_k)]))\\
&=&
(\sigma_k \after \mu)(\kappa_m(f \after \kappa_i, \lambda j \in m [\kappa_k(g,\kappa_j,\id_k)]))\\
&=&
\sigma_k([\kappa_{\coprod_m k}\big((g\after\kappa_1 + \ldots + g \after \kappa_m)\after(f\after \kappa_i),[\id_k,\ldots \id_k]\big)\\
&=&
\sigma_k([\kappa_{\coprod_m k}\big((g+\ldots+g) \after (\kappa_1 + \ldots +\kappa_m)\after(f\after \kappa_i\big), \nabla)])\\
&=&
\sigma_k([\kappa_k\big(\nabla \after (g+\ldots+g) \after (\kappa_1 + \ldots +\kappa_m)\after(f\after \kappa_i\big), \id_k)])\\
&&
\text{equivalence relation}\\
&=&
\sigma_k([\kappa_k(g \after f \after \kappa_i, \id_k)])\\
&=&
\overline{\sigma}(g \after f)(i)
\end{array}
$$

\item $\overline{\sigma}$ preserves coproducts\\
Consider $\tilde{\kappa_1} \colon n \to n+m$ in $\cat{L}$
$$
\begin{array}{rcll}
\overline{\sigma}(\kappa_1)(i) &=& \sigma_{n+m}([\kappa_{n+m}(\tilde{\kappa_1} \after \kappa_i, \id_{n+m})])\\
&=&
\sigma_{n+m}([\kappa_1(\id_1, \tilde{\kappa_1}(i))]) &\text{equiv. rel.}\\
&=&
(\sigma_{n+m}\after \eta \after \tilde{\kappa_1})(i)\\
&=&
(\eta \after \tilde{\kappa_1}) (i)\\
&=&
\kappa_1^{\Kl_{\NNO}}(i)
\end{array}
$$
\end{itemize}

\item $\overline{F}\colon \LM(\cat{L}) \to T$ is monad morphism.\\
\begin{itemize}
\item $\overline{F}$ is well-defined.\\
Let $f \colon i \to m$ in $\aleph_0$, $g \colon 1 \to i$ in $\cat{L}$, $w\colon m \to X$ in $\Sets$. Then
$$
\begin{array}{rcll}
	T(w) \after F(f_{\cat{L}} \after g) &=& T(w) \after (F(f_{\cat{L}}) \klafter F(g))\\
	&=&
	T(w) \after \mu \after T(F(f_{\cat{L}})) \after F(g)\\
	&=&
	T(w) \after \mu \after T(\eta) \after T(f) \after F(g)\\
	&=&
	T(w) \after T(f) \after F(g)\\  
	&=&	
	T(w \after f) \after F(g)
\end{array}
$$
\item Naturality of $\overline{F}$.\\
Let $h \colon X \to Y$,
$$
\begin{array}{rcll}
(T(h) \after \overline{F}_X)([\kappa_i(g,w)]) &=& (T(h) \after T(w) \after F(g))(*)\\
&=&
(T(h \after w) \after F(g))(*)\\
&=&
\overline{F}_Y([\kappa_i(g, h \after w)])\\
&=&
(\overline{F}_Y \after \LM(\cat{L})(h))([\kappa_i(g,w)])
\end{array}
$$
\item Preservation of $\eta$
$$
\begin{array}{rcll}
(\overline{F}_X \after \eta)(x) &=& \overline{F}_X([\kappa_1(\id_1,x)])\\
&=&
(T(x) \after F(\id_1))(*) \\
&=&
(T(x) \after \eta_1)(*) &(\eta_1 = \id_1^{\Kl_{\NNO}})\\
&=&
\eta(x) &\text{naturality of $\eta$}
\end{array}
$$
\item Preservation of $\mu$\\
Let $[\kappa_i(g,v)] \in \LM(\cat{L})^2(X) = \coprod_i \cat{L}(1,i) \times \LM(\cat{L})(X)^i$ and write, for $a < i$, $v(a) = [\kappa_{j_a}(g_a,v_a)]$, then
$$
\begin{array}{rcl}
\lefteqn{(\overline{F}_X \after \mu)([\kappa_i(g,v)])}\\
&=& \overline{F}_X([\kappa_{j_0+\ldots+j_{i-1}}((g_0 + \ldots + g_{i-1}) \after g,[v_0,\ldots,v_{i-1}])])\\
&=&
(T([v_0,\ldots,v_{i-1}])\after F((g_0 + \ldots + g_{i-1}) \after g))(*)\\
&=&
(T([v_0,\ldots,v_{i-1}])\after F(g_0 + \ldots + g_{i-1}) \klafter F(g))(*)\\
&=&
(T([v_0,\ldots,v_{i-1}])\after \mu^T \after T(F(g_0 + \ldots + g_{i-1})) \after F(g))(*)\\
&=&
(T([v_0,\ldots,v_{i-1}])\after \mu^T \after T(F([\kappa_0 \after g_0 + \ldots + \kappa_{i-1} \after g_{i-1}])) \after F(g))(*)\\
&=&
(T([v_0,\ldots,v_{i-1}])\after \mu^T \after T([F(\kappa_0) \klafter F(g_0) + \ldots + F(\kappa_{i-1}) \klafter F(g_{i-1})]) \after F(g))(*)\\
&=&
(T([v_0,\ldots,v_{i-1}])\after \mu^T \after T([\mu \after T(F(\kappa_0)) \after F(g_0) + \ldots ]) \after F(g))(*)\\
&=&
(T([v_0,\ldots,v_{i-1}])\after \mu^T \after T([\mu \after T(\eta \after \kappa_0) \after F(g_0) + \ldots ]) \after F(g))(*)\\
&=&
(T([v_0,\ldots,v_{i-1}])\after \mu^T \after T([T(\kappa_0) \after F(g_0) + \ldots + T(\kappa_{i-1}) \after F(g_{i-1})]) \after F(g))(*)\\
&=&
(T([v_0,\ldots,v_{i-1}])\after \mu^T \after T([T(\kappa_0) + \ldots + T(\kappa_{i-1})] \after (F(g_0) + \ldots F(g_{i-1}))) \after F(g))(*)\\
&=&
(T([v_0,\ldots,v_{i-1}])\after \mu^T \after T([T(\kappa_0) + \ldots + T(\kappa_{i-1})] \after (F(g_0) + \ldots F(g_{i-1}))) \after F(g))(*)\\
&=&
(T([v_0,\ldots,v_{i-1}])\after \mu^T \after T([T(\kappa_0) + \ldots + T(\kappa_{i-1})] \after (F(g_0) + \ldots F(g_{i-1}))) \after F(g))(*)\\
&=&
(\mu^T \after T^2([v_0,\ldots,v_{i-1}]) \after T([T(\kappa_0) + \ldots + T(\kappa_{i-1})] \after (F(g_0) + \ldots F(g_{i-1}))) \after F(g))(*)\\
&=&
(\mu^T \after T(T([v_0,\ldots,v_{i-1}]) \after [T(\kappa_0) + \ldots + T(\kappa_{i-1})] \after (F(g_0) + \ldots F(g_{i-1}))) \after F(g))(*)\\
&=&
(\mu^T \after T([T(v_0),\ldots,T(v_{i-1})] \after (F(g_0) + \ldots F(g_{i-1}))) \after F(g))(*)\\
&=&
(\mu^T \after T([T(v_0) \after F(g_0),\ldots,T(v_{i-1}) \after F(g_{i-1})]) \after F(g))(*)\\
&=&
(\mu^T \after T([* \mapsto (T(v_0) \after F(g_0))(*),\ldots,\star \mapsto (T(v_{i-1}) \after F(g_{i-1}))(*)]) \after F(g))(*)\\
&=&
(\mu^T \after T(\overline{F}_X \after [* \mapsto [\kappa_{j_0}(g_0,v_0)],\ldots,\star \mapsto [\kappa_{j_{i-1}}(g_{i-1},v_{i-1})]]) \after F(g))(*)\\
&=&
(\mu^T \after T(\overline{F}_X \after [* \mapsto v(0),\ldots,\star \mapsto v(i-1)]) \after F(g))(*)\\
&=&
(\mu^T \after T(\overline{F}_X \after v) \after F(g))(*)\\
&=&
(\mu^T \after \overline{F}_{T(X)}([\kappa_i(g,\overline{F}_X \after v)])\\
&=&
(\mu^T \after \overline{F}_{T(X)} \after \LM(\cat{L})(\overline{F}_X)([\kappa_i(g,v)])\\
\end{array}
$$ 
\end{itemize}
\item $\overline{\overline{\sigma}} = \sigma$\\
$\sigma_X \colon \LM(\cat{L})(X) = \coprod_i \cat{L}(1,i) \times X^i \to T(X)$
$$
\begin{array}{rcll}
\overline{\overline{\sigma}}([\kappa_i(g,w)]) &=& (T(w) \after \overline{\sigma}(g))(*)\\
&=&
(T(w) \after \sigma_i)([\kappa_i(g \after \kappa_0,\id)])\\
&=&
(T(w) \after \sigma_i)([\kappa_i(g,\id)])&(\kappa_0 \colon 1 \to 1 = \id)\\
&=&
(\sigma_X \after \LM(\cat{L})(w))([\kappa_i(g,\id)])\\
&=&
\sigma_X([\kappa_i(g,w)])
\end{array}
$$ 
\item $\overline{\overline{F}} = F$\\
$F \colon \cat{L} \to \Kl_{\NNO}(T)$. Let $f \colon n \to m$ in $\cat{L}$.
Then $\overline{\overline{F}}(f) \colon n \to T(m)$.
$$
\begin{array}{rcll}
\overline{\overline{F}}(f)(i) &=& \overline{F}_m([\kappa_m(f \after \kappa_i),\id_m])\\
&=&
(T(\id_m) \after F(f \after \kappa_i))(*)\\
&=&
(F(f) \klafter F(\kappa_i))(*)\\
&=&
(\mu \after T(F(f)) \after \eta \after \kappa_i)(*)\\
&=&
(\mu \after \eta \after F(f))(i)\\
&=&
F(f)(i)
\end{array}
$$
\item Naturality\\
For naturality consider:
$$\begin{bijectivecorrespondence}
  \correspondence[in \Mnd]{\xymatrix{\LM(\cat{K})\ar[r]^-{\LM(G)} & \LM(\cat{L}) \ar[r]^-{\sigma} & T \ar[r]^-{\tau} & S}}
  \correspondence[in \Law]{\xymatrix{\cat{K}\ar[r]_-{G} & \cat{L} \ar[r]_-{\overline{\sigma}} & \Kl_{\NNO}(T) \ar[r]_-{\Kl_{\NNO}(\tau)} & \Kl_{\NNO}(S)}}
\end{bijectivecorrespondence}$$
Let $f \colon n \to m$ in $\cat{K}$
\end{enumerate}
$$
\begin{array}{rcll}
\overline{\tau \after \sigma \after \LM(G)}(f)(i)
&=&
(\tau \after \sigma \after \LM(G))_m([\kappa_m(f \after \kappa_i^{\cat{K}},\id_m)])\\
&=&
(\tau_m \after \sigma_m)([\kappa_m(G(f \after \kappa_i^{\cat{K}}), \id_m)])\\
&=&
(\tau_m \after \sigma_m)([\kappa_m(G(f) \after \kappa_i^{\cat{L}}, \id_m)])\\
&=&
(\Kl_{\NNO}(\tau) \after \sigma_m)([\kappa_m(G(f) \after \kappa_i^{\cat{L}}, \id_m)])\\
&=&
(\Kl_{\NNO}(\tau) \after \overline{\sigma})(G(f))(i)\\
&=&
(\Kl_{\NNO}(\tau) \after \overline{\sigma} \after G)(f)(i)\\
\end{array}
$$
}

Finally, we consider the right-hand side of Figure \ref{MonoidTriangleFig}.
For each category $\cat{C}$ and object $X$ in $\cat{C}$, the homset
$\cat{C}(X,X)$ is a monoid, where multiplication is given by
composition with the identity as unit. The mapping $\cat{L} \mapsto
\mathcal{H}(\cat{L}) = \cat{L}(1,1)$, defines a functor
$\Law \to \Mon$. This functor is right adjoint to the
composite functor $\Kl_{\NNO} \circ \mathcal{A}$.

\begin{lemma}
\label{AdjCatMonLem}
The pair of functors $\Kl_{\NNO} \after \mathcal{A} \colon
\Mon \rightleftarrows \Law \colon \mathcal{H}$ forms an
adjunction $\Kl_{\NNO} \after \mathcal{A} \dashv \mathcal{H}$, as
on the right in Figure~\ref{MonoidTriangleFig}.
\end{lemma}

\begin{proof}
For a monoid $M$ and a Lawvere theory $\cat{L}$ there are
(natural) bijective correspondences:
$$\begin{bijectivecorrespondence}
  \correspondence[in \Law]{\xymatrix{\Kl_{\NNO}\mathcal{A}(M)\ar[r]^-{F} & \cat{L}}}
  \correspondence[in \Mon]{\xymatrix{M\ar[r]_-{f} & \mathcal{H}(\cat{L})}}
\end{bijectivecorrespondence}$$

\noindent Given $F$ one defines a monoid map $\overline{F}
\colon M\rightarrow \mathcal{H}(\cat{L}) = \cat{L}(1,1)$ by 
$$s \mapsto F(1 \xrightarrow{\tuple{\lam{x}{s}}{!}} M \times 1).$$ 
Note that $1 \xrightarrow{\tuple{\lam{x}{s}}{!}} M \times 1 =
\mathcal{A}(M)(1)$ is an endomap on $1$ in
$\Kl_{\NNO}\mathcal{A}(M)$. Since $F$ is the identity on objects it
sends this endomap to an element of $\cat{L}(1,1)$.

Conversely, given a monoid map $f\colon M\rightarrow \cat{L}(1,1)$ one
defines a \Law-map $\overline{f} \colon \Kl_{\NNO}\mathcal{A}(M)
\rightarrow \cat{L}$. It is the identity on objects and sends a
morphism $h \colon n \to m$ in $\Kl_{\NNO}\mathcal{A}(M)$,
\textit{i.e.} $h \colon n \to M \times m$ in $\Sets$, to the morphism
$$\xymatrix@C+1pc{
\overline{f}(h) = \Big(n 
   \ar[rr]^-{\big[\kappa_{h_{2}(i)} \after f(h_{1}(i))\big]_{i < n}} 
   &&m \Big)}.
$$

\noindent Here we write $h(i)\in M\times m$ as pair $(h_{1}(i),
h_{2}(i))$.  We leave further details to the reader. \QED
\end{proof}

\auxproof{
\begin{enumerate}
\item $\overline{F} \colon M \to \cat{L}(1,1)$ is monoid morphism.\\
For $s \in M$, define $\underline{s} = 1 \xrightarrow{\tuple\lam{x}{s}{!}} M \times 1$.
$$
\begin{array}{rcll}
\overline{F}(s) \cdot \overline{F}(t) &=& F(\underline{s}) \cdot F(\underline{t})\\
&=&
F(\underline{t}) \after F(\underline{s})\\
&=&
F(\underline{t} \after \underline{s})\\
&=&
F(\mu \after (\id \cdot \underline{t}) \after \underline{s})\\
&=&
F(\underline{s \cdot t}) = \overline{F}(s \cdot t).
\end{array}
$$
and 
$$
\overline{F}(1) = F(\underline{1}) = F(\id_1) = \id_{F(1)} = \id_I.
$$
\item $\overline{f}\colon (\Kl_{\NNO} \after \mathcal{A})(M) \to \cat{L}$ is a morphism in \Law.

First we show that $\overline{f}$ is a functor.\\
Let $h \colon n \to M\times m$ and $g \colon m \to M \times p$ be morphisms in $(\Kl_{\NNO} \after \mathcal{A})(M)$. 
$$
\begin{array}{rcl}
g \klafter h &=& n \xrightarrow{h} M \times m \xrightarrow{\id \times g} M \times (M \times p) \xrightarrow{\mu} M \times p\\
&=&
i \mapsto (h_1(i),h_2(i)) \mapsto (h_1(i), (g_1(h_2(i)), g_2(h_2(i)))) \mapsto (h_1(i) \cdot g_1(h_2(i)), g_2(h_2(i)))
\end{array}
$$
Using this, we see that
$$
\begin{array}{rcl}
\overline{f}(g) \after \overline{f}(h) &=& \big[\kappa_{g_{2}(j)} \after f(g_{1}(j))\big]_{j \in m}\after \big[\kappa_{h_{2}(i)} \after f(h_{1}(i))\big]_{i \in n}\\
&=&
\big[\big[\kappa_{g_{2}(j)} \after f(g_{1}(j))\big]_{j \in m}\after \kappa_{h_{2}(i)} \after f(h_{1}(i))\big]_{i \in n}\\
&=&
\big[\kappa_{g_{2}(h_2(i))} \after f(g_{1}(h_2(i))) \after f(h_{1}(i))\big]_{i \in n}\\
&=&
\big[\kappa_{g_{2}(h_2(i))} \after f(h_{1}(i)\cdot g_{1}(h_2(i)))\big]_{i \in n}\\
&=&
\big[\kappa_{(g\after h)_2(i)} \after f((g\after h)_2(i))\big]_{i \in n}\\
\end{array}
$$
$\overline{f}$ also preserves the identity: $\id \colon n \to M \times n, i \mapsto (1,i)$ and therefore
$$
\begin{array}{rcl}
\overline{f}(\id_n) &=& \big[\kappa_{\id_{2}(i)} \after f(\id_{1}(i))\big]_{i \in n}\\
&=&
\big[\kappa_i \after f(1)\big]_{i \in n}\\
&=&
\big[\kappa_i \after \id\big]_{i \in n}\\
&=&
\id
\end{array}
$$
So $\overline{f}$ is a functor.

$\overline{f}$ preserves the coproduct structure as the canonical map
$$
\overline{f}(n) + \overline{f}(m) \xrightarrow{\tuple{\overline{f}(\kappa_1)}{\overline{f}(\kappa_2)}} \overline{f}(n+m)
$$
is the identity map, and therefore certainly an isomorphism.

\item $\overline{\overline{f}} = f$
$$
\overline{\overline{f}}(s) = \overline{f}(\underline{s}) = f(\underline{s}_1(*)) = f(s)
$$

\item $\overline{\overline{F}}=F$
It is clear that $\overline{\overline{F}}=F$ on objects. Now let $h \colon n \to M \times m$.
$$
\begin{array}{rcll}
\overline{\overline{F}}(h) &=& \big[\kappa_{h_{2}(i)} \after \overline{F}(h_{1}(i))\big]_{i \in n}\\
&=&
\big[\kappa_{h_{2}(i)} \after F(\underline{h_{1}(i)})\big]_{i \in n}\\
&=&
\big[F(\kappa_{h_{2}(i)}) \after F(\underline{h_{1}(i)})\big]_{i \in n}&\text{$F$ preserves coproducts}\\
&=&
\big[F(\kappa_{h_{2}(i)} \after \underline{h_{1}(i)})\big]_{i \in n}\\
&=&
\big[F(h \after \kappa_i)\big]_{i \in n}&\text{(1)}\\
&=&
F(h)
\end{array}
$$
where (1) relies on the fact that
$$
\begin{array}{rcl}
\kappa_{h_{2}(i)} \after \underline{h_{1}(i)} &=& 1 \xrightarrow{\underline{h_{1}(i)}} M\times `1 \xrightarrow{\id \times \kappa_{h_{2}(i)}} M \times (M \times m) \xrightarrow{\mu} M \times m \\
&&
* \mapsto (h_1(i),*) \mapsto (h_1(i),(1,h_2(i))) \mapsto (h_1(i),h_2(i)) = h(i)
\end{array}
$$

\item Naturality\\
Consider:
$$\begin{bijectivecorrespondence}
  \correspondence[in \Law]{\xymatrix{\Kl_{\NNO}\mathcal{A}(N)\ar[r]^-{\Kl_{\NNO}\mathcal{A}(f)} & \Kl_{\NNO}\mathcal{A}(M) \ar[r]^-{F} & \cat{L} \ar[r]^-{G} & \cat{K}}}
  \correspondence[in \Mon]{\xymatrix{N\ar[r]_-{f} & M \ar[r]_-{\overline{F}} & \cat{L}(1,1) \ar[r]_-{\mathcal{H}(G)} & \cat{K}(1,1)}}
\end{bijectivecorrespondence}$$
To prove: $\overline{G \after F \after \Kl_{\NNO}\mathcal{A}(f)} = \mathcal{H}(G) \after \overline{F} \after f$.
$$
\begin{array}{rcll}
\overline{G \after F \after \Kl_{\NNO}\mathcal{A}(f)}(s) 
&=&
(G \after F \after \Kl_{\NNO}\mathcal{A}(f))(\underline{s})\\
&=&
(G \after F)((f \times \id) \after \underline{s}) &\text{definition $\Kl_{\NNO}\mathcal{A}$}\\
&=&
(G \after F)(\underline{f(s)})\\
&=&
(\mathcal{H}(G)\after \overline{F}\after f)(s)
\end{array}
$$
\end{enumerate}
}

Given a monad $T$ on $\Sets$, $\mathcal{H}\Kl_{\NNO}(T) =
\Kl_{\NNO}(T)(1,1) = \Sets(1, T(1))$ is a monoid, where the multiplication is
given by
$$
(1\xrightarrow{a} T(1))\cdot (1 \xrightarrow{b} T(1)) = 
\big( 1 \xrightarrow{a} T(1) \xrightarrow{T(b)} T^2(1) 
\xrightarrow{\mu} T(1)\big).
$$

\noindent The functor $\Ev: \Mnd(\cat{C}) \to \Mon(\cat{C})$, defined
in Lemma \ref{Mnd2MonLem} also gives a multiplication on
$\Sets(1,T(1)) \cong T(1)$, namely $\mu \after T(\pi_2) \after \st
\colon T(1) \times T(1) \to T(1)$. These two multiplications coincide
as is demonstrated in the following diagram,
$$
\xymatrix@C+1pc{
	1 \ar[r]_-{a}\ar @/^4ex/ [rrr]^-{\tuple{a}{b}} & 
   T(1) \ar[r]^-{\rho^{-1}} \ar[rd]^{T(\rho^{-1})} 
  \ar@/_{6ex}/[ddrr]_{T(b)} & 
  T(1) \times 1 \ar[r]_-{\idmap \times b} \ar[d]^{\st} & 
  T(1) \times T(1) \ar[d]^{\st}\ar @/^2ex/ [ddr]^{\cdot} & \\
	&& T(1 \times 1) \ar[r]_-{T(\idmap \times b)} & T(1 \times T(1)) 
   \ar[d]^{T(\lambda)} & \\
	&&&T^2(1) \ar[r]^{\mu}&T(1)
}
$$	

\noindent In fact, $\Ev \cong \mathcal{H}\Kl_{\NNO}$, which completes
the picture from Figure \ref{MonoidTriangleFig}.

\subsection{Commutative monoids}\label{ComMonoidSubsec}

In this subsection we briefly summarize what will change in the triangle
in Figure~\ref{MonoidTriangleFig} when we restrict ourselves to
commutative monoids (at the top). This will lead to commutative
monads, and to tensor products. The latter are induced by
Lemma~\ref{KleisliStructLem}. The new situation is described in
Figure~\ref{ComMonoidTriangleFig}. For the adjunction between
commutative monoids and commutative monads we start with the following
basic result.

\begin{lemma}
\label{CMnd2CMonLem}
Let $T$ be a commutative monad on a category \Cat{C} with finite
products. The monoid $\Ev(T) = T(1)$ in $\Cat{C}$ from
Lemma~\ref{Mnd2MonLem} is then commutative.
\end{lemma}

\begin{proof}
  Recall that the multiplication on $T(1)$ is given by $\mu \after
  T(\lambda) \after \st \colon T(1) \times T(1) \to T(1)$ and
  commutativity of the monad $T$ means $\mu \after T(\st') \after \st
  = \mu \after T(\st) \after \st'$ where $\st' = T(\gamma) \after \st
  \after \gamma$, for the swap map $\gamma$, see
  Section~\ref{PrelimSec}. Then:
$$\hspace*{-2em}\begin{array}[b]{rcl}
\mu \after T(\lambda) \after \st \after \gamma
& = &
\mu \after T(T(\lambda) \after \st') \after \st \after \gamma \\
& = &
T(\lambda) \after \mu \after T(\st') \after \st \after \gamma \\
& = &
T(\rho) \after \mu \after T(\st) \after \st' \after \gamma 
   \quad\mbox{by commutativity of \rlap{$T$,}}\\
& & \qquad\mbox{and because $\rho = \lambda
   \colon 1\times 1\rightarrow 1$} \\
& = &
\mu \after T(T(\rho) \after \st \after \gamma) \after \st \\
& = &
\mu \after T(\rho \after \gamma) \after \st \\
& = &
\mu \after T(\lambda) \after \st.
\end{array}\eqno{\QEDbox}$$
\end{proof}

The proof of the next result is easy and left to the reader.

\begin{lemma}
\label{CMon2CMndLem}
A monoid $M$ is commutative (Abelian) if and only if the associated
monad $\mathcal{A}(M) = M\times (-)\colon \Sets\rightarrow\Sets$ is
commutative (as described in Section~\ref{PrelimSec}). \QED
\end{lemma}

Next, we wish to define an appropriate category \SMLaw of Lawvere
theories with symmetric monoidal structure $(\otimes, I)$. In order to
do so we need to take a closer look at the category $\aleph_0$
described in the introduction. Recall that $\aleph_0$ has $n\in\NNO$
as objects whilst morphisms $n\rightarrow m$ are functions
$\underline{n} \rightarrow \underline{m}$ in \Sets, where, as
described earlier $\underline{n} = \{0,1,\ldots,n-1\}$. This category
$\aleph_0$ has a monoidal structure, given on objects by
multiplication $n\times m$ of natural numbers, with $1\in\NNO$ as
tensor unit. Functoriality involves a (chosen) coordinatisation, in
the following way. For $f\colon n\rightarrow p$ and $g\colon
m\rightarrow q$ in $\aleph_{0}$ one obtains $f\otimes g \colon n\times
m\rightarrow p\times q$ as a function:
$$\begin{array}{rcl}
f\otimes g
& = &
\coo_{p,q}^{-1} \after (f\times g) \after \coo_{n,m}
   \;\colon\; \underline{n\times m}\longrightarrow \underline{p\times q},
\end{array}$$

\noindent where $\coo$ is a coordinatisation function
$$\xymatrix@C+.5pc{
\underline{n\times m} = \{0,\ldots, (n\times m)-1\}
   \ar[r]^-{\textsf{co}_{n,m}}_-{\cong} &
   \{0,\ldots, n-1\} \times \{0,\ldots, m-1\} = 
   \underline{n}\times\underline{m},
}$$

\noindent given by
\begin{equation}
\label{CooEqn}
\begin{array}{rcl}
\coo(c) = (a,b)
& \Leftrightarrow &
c = a\cdot m + b.
\end{array}
\end{equation}

\noindent We may write the inverse $\coo^{-1} \colon \overline{n}
\times\overline{m} \rightarrow \overline{n\times m}$ as a small
tensor, as in $a\sotimes b = \coo^{-1}(a,b)$. Then: $(f\otimes
b)(a\sotimes b) = f(a)\sotimes g(b)$. The monoidal isomorphisms in
$\aleph_0$ are then obtained from $\Sets$, as in
$$\xymatrix{
\gamma^{\aleph_0} = \Big(\underline{n\times m}\ar[r]^-{\coo} &
   \underline{n}\times\underline{m}\ar[r]^-{\gamma^{\Sets}} &
   \underline{m}\times\underline{n}\ar[r]^-{\coo^{-1}} & 
   \underline{m\times n}\Big).
}$$

\noindent Thus $\gamma^{\aleph_0}(a\sotimes b) = b\sotimes
a$. Similarly, the associativity map $\alpha^{\aleph_0} \colon n\otimes
(m\otimes k) \rightarrow (n\otimes m)\otimes k$ is determined as
$\alpha^{\aleph_0}(a\sotimes (b\sotimes c)) = (a\sotimes b)\sotimes
c$.  The maps $\rho\colon n\times 1\rightarrow n$ in $\aleph_0$ are
identities.

\auxproof{
Formally,
$$\begin{array}{rcl}
\alpha^{\aleph_0}
& = &
\coo_{n\times m,k}^{-1} \after (\coo_{n,m}^{-1}\times\idmap) \after
   \alpha^{\Sets} \after (\idmap{}\times\coo_{m,k})\after \coo_{n,m\times k}.
\end{array}$$
}

This tensor $\otimes$ on $\aleph_0$ distributes over sum: the
canonical distributivity map $(n\otimes m)+(n\otimes k) \rightarrow
n\otimes (m+k)$ is an isomorphism. Its inverse maps $a\sotimes b \in
n\otimes (m+k)$ to $a\sotimes b \in \underline{n\times m}$ if $b<m$,
and to $a\sotimes (b-m)\in\underline{n\times k}$ otherwise. 

We thus define the objects of the category \SMLaw to be symmetric
monoidal Lawvere theories $\cat{L}\in\Law$ for which the map
$\aleph_{0}\rightarrow\cat{L}$ strictly preserves the monoidal
structure that has just been described via multiplication $(\times,
1)$ of natural numbers; additionally the coproduct structure must be
preserved, as in \Law. Morphisms in \SMLaw are morphisms in \Law that
strictly preserve this tensor structure. We note that for
$\cat{L}\in\SMLaw$ we have a distributivity $n\otimes m + n\otimes k
\congrightarrow n\otimes (m+k)$, since this isomorphism lies in the
range of the functor $\aleph_{0}\rightarrow\cat{L}$.

By Lemma~\ref{KleisliStructLem} we know that the Kleisli category
$\Kl(T)$ is symmetric monoidal if $T$ is commutative. In order to see
that also the finitary Kleisli category $\Kl_{\NNO}(T)\in\Law$ is
symmetric monoidal, we have to use the coordinatisation map described
in~\eqref{CooEqn}. For $f\colon n\rightarrow p$ and $g\colon m
\rightarrow q$ in $\Kl_{\NNO}(T)$ we then obtain $f\otimes g \colon
n\times m\rightarrow p\times q$ as
$$\xymatrix@C-.3pc{
f\otimes g = \Big(\underline{n\times m}\ar[r]^-{\coo} &
   \underline{n}\times\underline{m}\ar[r]^-{f\times g} &
   T(\underline{p})\times T(\underline{q})\ar[r]^-{\dst} &
   T(\underline{p}\times\underline{q})\ar[r]^-{T(\coo^{-1})} &
   T(\underline{p\times q})\Big).
}$$

We recall from~\cite{KellyL80} (see
also~\cite{AbramskyC04,AbramskyC09}) that for a monoidal category
$\cat{C}$ the homset $\cat{C}(I,I)$ of endomaps on the tensor unit
forms a commutative monoid. This applies in particular to Lawvere
theories $\cat{L}\in\SMLaw$, and yields a functor $\mathcal{H}\colon
\SMLaw \rightarrow \Mon$ given by $\mathcal{H}(\cat{L}) =
\cat{L}(1,1)$, where $1\in\cat{L}$ is the tensor unit. Thus we almost
have a triangle of adjunctions as in
Figure~\ref{ComMonoidTriangleFig}.  We only need to check the
following result.

\auxproof{ Starting from the middle $I$ on the left, $t\after s\colon
  I\rightarrow I$ is the upper path $t \after \rho^{-1} \after \lambda
  \after s$ to the middle $I$ on the right. Similarly $s\after t$ is
  the lower path.
$$\xymatrix@C13ex@R-1ex{
    I \ar_-{\lambda}^-{\cong}[r] & I \otimes I \ar@{=}[r] & I \otimes
    I \ar^-{\idmap \otimes t}[d] \ar^-{\cong}_-{\rho^{-1}}[r] & I \ar^-{t}[d] \\
    I \ar^-{\cong}_-{\lambda=\rho}[r] \ar^-{s}[u] \ar_-{t}[d] & I
    \otimes I \ar^-{s \otimes \idmap}[u] 
    \ar_-{\idmap \otimes t}[d] \ar^-{s \otimes t}[r] & I \otimes I
    \ar^-{\cong}_-{\lambda^{-1}=\rho^{-1}}[r] & I \\
    I \ar_-{\rho}^-{\cong}[r] & I \otimes I \ar@{=}[r] & I \otimes I
    \ar_-{s \otimes \idmap}[u] \ar^-{\cong}_-{\lambda^{-1}}[r] & I. \ar_-{s}[u]
  }$$
}

\begin{lemma}
\label{LMCommLem}
The functor $\LM\colon\Law\rightarrow\Mnd$ defined in~\eqref{LMEqn}
restricts to $\SMLaw \rightarrow \CMnd$. Further, this restriction is
left adjoint to $\Kl_{\NNO}\colon \CMnd\rightarrow \SMLaw$.
\end{lemma}

\begin{proof}
For $\cat{L}\in\SMLaw$ we define a map
$$\xymatrix@R-1.8pc{
\LM(\cat{L})(X)\times \LM(\cat{L})(Y)\ar[rr]^-{\dst} & &
   \LM(\cat{L})(X\times Y) \\
\big([\kappa_{i}(g,v)], [\kappa_{j}(h,w)]\big)\ar@{|->}[rr] & &
   [\kappa_{i\times j}(g\otimes h, (v\times w) \after \coo_{i,j})],
}$$

\noindent where $g\colon 1\rightarrow i$ and $h\colon 1\rightarrow j$
in $\cat{L}$ yield $g\otimes h\colon 1 = 1\otimes 1 \rightarrow
i\otimes j = i\times j$, and $\coo$ is the coordinatisation
function~\eqref{CooEqn}. Then one can show that both $\mu \after
\LM(\cat{L})(\st') \after \st$ and $\mu \after \LM(\cat{L})(\st)
\after \st'$ are equal to $\dst$. This makes $\LM(\cat{L})$ a
commutative monad.

In order to check that the adjunction $\LM \dashv \Kl_{\NNO}$
restricts, we only need to verify that the unit $\cat{L} \rightarrow
\Kl_{\NNO}(\LM({\cat{L}}))$ strictly preserves tensors.  This is
easy. \QED

\auxproof{
We shall use the following formulation of multiplication $\mu\colon
\LM(\cat{L})(\LM(\cat{L})(X))\rightarrow \LM(\cat{L})(X)$:
$$\begin{array}{rcl}
\mu([\kappa_{i}(g,v)])
& = &
[\kappa_{j}((g_{0}+\cdots+g_{i-1})\after g, [v_{0},\ldots, v_{i-1}])] \\
& & \qquad \mbox{where }g\colon 1\rightarrow i, \mbox{ and } 
    v\colon i\rightarrow \LM(\cat{L})(X) \mbox{ is written as} \\
& & \qquad\qquad v(a) = [\kappa_{j_{a}}(g_{a}, v_{a})], \mbox{ for }a<i, \\
& & \qquad \mbox{and } j = j_{0} + \cdots + j_{i-1}.
\end{array}$$

\noindent We note that the strength map $\st\colon \LM(\cat{L})(X)
\times Y \rightarrow \LM(\cat{L})(X\times Y)$ is given by:
$$\begin{array}{rcl}
[(\kappa_{i}(g, v), y)]
& \longmapsto &
[\kappa_{i}(g, \lam{a<i}{(g(a), y)})].
\end{array}$$

We use that the following ``left distributivity'' map is an identity
in $\aleph_0$.
\begin{equation}
\label{LeftDistrEqn}
\xymatrix@R-2pc@C-1pc{
\mathsf{ld} = \Big(i\times j = j+\cdots+j\ar[r]^-{=} &
   1\otimes j+\cdots+ 1\otimes j\ar[r]^-{\cong} &
   (1+\cdots+1)\otimes j = i\otimes j\Big)
}
\end{equation}

\noindent Thus, omitting equivalence brackets,
$$\begin{array}{rcl}
\lefteqn{\big(\mu \after \LM(\cat{L})(\st') \after \st)
   (\kappa_{i}(g,v), \kappa_{j}(h,w))} \\
& = &
\big(\mu \after \LM(\cat{L})(\st')\big)
   (\kappa_{i}(g, \lam{a<i}{\kappa_{i}(v(a), \kappa_{j}(h,w))})\big) \\
& = &
\mu\big(\kappa_{i}(g, \lam{a<i}{\st'(\kappa_{i}(v(a), \kappa_{j}(h,w)))})\big) \\
& = &
\mu\big(\kappa_{i}(g, \lam{a<i}{\kappa_{j}(h, 
   \lam{b<j}{(v(a),w(b))})})\big) \\
& = &
\kappa_{i\times j}((h+\cdots+h)\after g, [\lam{b<j}{(v(a),w(b))}]_{a<i}) \\
& = &
\kappa_{i\times j}((h+\cdots+h)\after g, (v\times w)\after \coo_{i,j}) \\
& = &
\kappa_{i\times j}(\mathsf{ld} \after (h+\cdots+h)\after g, 
   (v\times w)\after \coo_{i,j}) \\
& & \qquad \mbox{because this $\mathsf{ld}$ from~\eqref{LeftDistrEqn} 
    is the identity in $\aleph_0$} \\
& = &
\kappa_{i\times j}(g\otimes h, (v\times w)\after \coo_{i,j}) 
   \qquad \mbox{see below} \\
& = &
\dst(\kappa_{i}(g,v), \kappa_{j}(h,w)).
\end{array}$$

\noindent We still need to check:
$$\xymatrix@C+1.8pc{
1\ar[r]^-{g}\ar[dd]_{\lambda^{-1}}^{=} &
   i = 1+\cdots+1\ar[r]^-{h+\cdots+h}\ar[d]^{\lambda^{-1}+\cdots+\lambda^{-1}}
      \ar @/_12ex/ [dd]_{\lambda^{-1}} &
   j+\cdots+j\ar[d]_{\lambda^{-1}+\cdots+\lambda^{-1}}^{=}
      \ar @/^12ex/[dd]^{\mathsf{ld}} \\
& 
   1\otimes 1+\cdots+ 1\otimes 1\ar[d]^{[\kappa_{a}\otimes\idmap]_{a<i}}_{\cong}
      \ar[r]^-{\idmap\otimes h+\cdots+ \idmap{}\otimes h} & 
   1\otimes j+\cdots+ 1\otimes j\ar[d]_{[\kappa_{a}\otimes\idmap]_{a<i}}^{\cong} \\
1\otimes 1\ar[r]^-{g\otimes\idmap} & 
   (1+\cdots+1)\otimes 1\ar[r]^-{\idmap\otimes h} & 
   (1+\cdots+1)\otimes j\rlap{$\;=i\times j$}
}$$

\noindent Next, for the other composite we get:
$$\begin{array}{rcl}
\lefteqn{\big(\mu \after \LM(\cat{L})(\st) \after \st')
   (\kappa_{i}(g,v), \kappa_{j}(h,w))} \\
& = &
\big(\mu \after \LM(\cat{L})(\st)\big)
   (\kappa_{j}(h, \lam{b<j}{\kappa_{i}(g,v), w(b)})\big) \\
& = &
\mu\big(\kappa_{j}(h, \lam{b<j}{\st(\kappa_{i}(g,v), w(b))})\big) \\
& = &
\mu\big(\kappa_{j}(h, \lam{b<j}{\kappa_{i}(g, 
   \lam{a<i}{(v(a),w(b))})})\big) \\
& = &
\kappa_{j\times i}((g+\cdots+g)\after h, [\lam{a<i}{(v(a),w(b))}]_{b<j}) \\
& = &
\kappa_{j\times i}((g+\cdots+g)\after h, (v\times w) \after \gamma^{\Sets} 
   \after \coo_{j,i}) \\
& = &
\kappa_{j\times i}((g+\cdots+g)\after h, (v\times w) \after 
   \coo_{i,j} \after \coo_{i,j}^{-1}  \after \gamma^{\Sets} 
   \after \coo_{j,i}) \\
& = &
\kappa_{j\times i}((g+\cdots+g)\after h, (v\times w) \after 
   \coo_{i,j} \after \gamma^{\aleph_{0}}) \\
& = &
\kappa_{i\times j}(\gamma^{\aleph_{0}} \after (g+\cdots+g)\after h, 
   (v\times w) \after \coo_{i,j}) \\
& = &
\kappa_{i\times j}(\gamma^{\aleph_{0}} \after \mathsf{ld} \after 
   (g+\cdots+g)\after h, (v\times w) \after \coo_{i,j}) \\
& = &
\kappa_{i\times j}(\gamma^{\aleph_{0}} \after (h\otimes g), 
   (v\times w) \after \coo_{i,j}) \\
& = &
\kappa_{i\times j}((g\otimes h) \after \gamma^{\aleph_{0}},
   (v\times w) \after \coo_{i,j}) \\
& = &
\kappa_{i\times j}(g\otimes h, (v\times w) \after \coo_{i,j}) \\
& & \qquad \mbox{because $\gamma^{\aleph_{0}}\colon 1\otimes 1\rightarrow 
   1\otimes 1$ is the identity} \\
& = &
\dst(\kappa_{i}(g,v), \kappa_{j}(h,w)).
\end{array}$$

For the adjunction $\LM\colon \SMLaw \leftrightarrows \CMnd \colon
\Kl_{N}$ we only need to check that the unit $\eta\colon \cat{L}
\rightarrow \Kl_{\NNO}(\LM(\cat{L}))$ is a map in \SMLaw. On
objects it is of course the identity. It sends a morphism
$f\colon n\rightarrow p$ in \cat{L} to the map
$\eta(f)\colon n\rightarrow T_{\cat{L}}(p)$ given by:
$$\begin{array}{rcl}
\eta(f)(i)
& = &
\kappa_{p}(f\after \kappa_{i}, \idmap_{p}),
\end{array}$$

\noindent see the proof of Lemma~\ref{AdjMndLvTLem}. We need to check
that it preserves tensor. For an additional map $g\colon m\rightarrow
q$ in \cat{L} we get, for $i<n\times m$, say $i = a\sotimes b = a\cdot
m + b$, for $a<n$ and $b<m$,
$$\begin{array}{rcl}
\big(\eta(f)\otimes\eta(g)\big)(i)
& = &
\big(T_{\cat{L}}(\coo_{p,q}^{-1}) \after \dst \after (\eta(f)\times \eta(g)) 
   \after \coo_{n,m}\big)(i) \\
& = &
T_{\cat{L}}(\coo_{p,q}^{-1})(\dst(\eta(f)(a), \eta(g)(b))) \\
& = &
T_{\cat{L}}(\coo_{p,q}^{-1})(\dst(\kappa_{p}(f\after\kappa_{a}, \idmap_{p}), 
   \kappa_{q}(g \after \kappa_{b}, \idmap_{q}))) \\
& = &
T_{\cat{L}}(\coo_{p,q}^{-1})(\kappa_{p\times q}(
   (f\after\kappa_{a})\otimes (g\after\kappa_{b}), \coo_{p,q})) \\
& = &
\kappa_{p\times q}((f\otimes g) \after (\kappa_{a}\otimes\kappa_{b}), 
   \coo_{p,q}^{-1} \after \coo_{p,q}) \\
& \smash{\stackrel{*}{=}} &
\kappa_{p\times q}((f\otimes g)\after \kappa_{i}, \idmap_{p\times q})  \\
& = &
\eta(f\otimes g)(i)
\end{array}$$

\noindent where the marked equation holds because in $\aleph_0$
(and hence in $\cat{L}$),
$$\begin{array}{rcl}
\kappa_{a}\otimes \kappa_{b}
& = &
\coo_{n,m}^{-1} \after (\kappa_{a}\times \kappa_{b}) \after \coo_{1,1} \\
& = &
\sotimes\; \after (\kappa_{a}\times \kappa_{b}) \after \coo_{1,1} \\
& = &
\kappa_{i}.
\end{array}$$
}
\end{proof}

\begin{figure}
$$\xymatrix@R-.5pc@C+.5pc{
& & \CMon\ar@/_2ex/ [ddll]_{\cal A}\ar@/_2ex/ [ddrr]_(0.4){\Kl_{\NNO}\mathcal{A}} \\
& \dashv & & \dashv & \\
\CMnd\ar @/_2ex/[rrrr]_{\Kl_\NNO}
   \ar@/_2ex/ [uurr]_(0.6){\;{\Ev} \cong \mathcal{H}\Kl_{\NNO}} & & \bot & &  
   \SMLaw\ar@/_2ex/ [uull]_{\mathcal{H}}\ar @/_2ex/[llll]_{\LM}
}$$
\caption{Commutative version of Figure~\ref{MonoidTriangleFig}, with
  commutative monoids, commutative monads and symmetric monoidal
  Lawvere theories.}
\label{ComMonoidTriangleFig}
\end{figure}
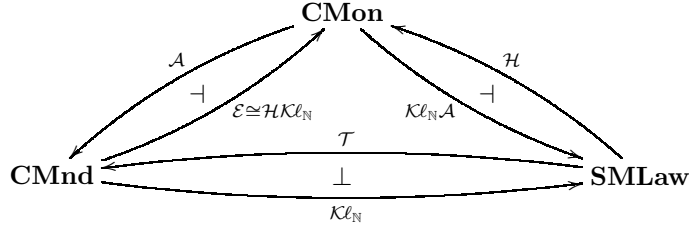

\section{Additive monads}\label{AMndSec}

Having an adjunction between commutative monoids and commutative
monads (Figure~\ref{ComMonoidTriangleFig}) raises the question whether
we may also define an adjunction between commutative semirings and some
specific class of monads. It will appear that so-called additive
commutative monads are needed here. In this section we will define and
study such additive (commutative) monads and see how they relate to
biproducts in their Kleisli categories and categories of algebras.

We consider monads on a category $\cat{C}$ with both
finite products and coproducts.  If, for a monad
$T$ on $\cat{C}$, the object $T(0)$ is final---\textit{i.e.}~satisfies
$T(0) \cong 1$---then $0$ is both initial and final in the Kleisli
category $\Kl(T)$. Such an object that is both initial and final is
called a \emph{zero object}. 

Also the converse is true, if $0 \in \Kl(T)$ is a zero object, then
$T(0)$ is final in $\cat{C}$. Although we don't use this in the
remainder of this paper, we also mention a related result on
the category of Eilenberg-Moore algebras. The proofs are simple and
are left to the reader.

\begin{lemma}\label{zerolem}
For a monad $T$ on a category $\cat{C}$ with finite products $(\times,
1)$ and coproducts $(+,0)$, the following statements are equivalent.
\begin{enumerate}
\renewcommand{\theenumi}{(\roman{enumi})}
\item $T(0)$ is final in $\cat{C}$;
\item $0 \in \Kl(T)$ is a zero object;
\item $1 \in \Alg(T)$ is a zero object. \QED
\end{enumerate}
\end{lemma}

\auxproof{
(i) $\Rightarrow$ (ii):

For each object $X$, there is a unique map $0 \to T(X)$ in $\cat{C}$
(i.e. $0 \to X$ in $\Kl(T)$), by initiality of $0$. Furthermore there
is a unique map $X \to T(0)$ in $\cat{C}$ (i.e. $X \to 0$ in
$\Kl(T)$), as by (i) $T(0)$ is final.\\

(ii) $\Rightarrow$ (i):

As 0 is a zero object in $\Kl(T)$, there is for each object $X \in
\cat{C}$, a unique map $X \to 0$ in $\Kl(T)$ (i.e. a unique map $X \to
T(0)$ in $\cat{C}$) so $T(0)$ is final.\\

(i) $\Rightarrow$ (iii):

As the free construction $\cat{C} \to \Alg(T), X \mapsto (T^2(X)
\xrightarrow{\mu} T(X))$ preserves coproducts $(T^2(0)
\xrightarrow{\mu} T(0))$ is initial in $\Alg(T)$ By (i), $T(0) = 1$,
so $1$ is initial in $\Alg(T)$. Clearly $1$ is also final, so it is a
zero object.\\

(iii) $\Rightarrow$ (i):

As $T(0)$ is initial in $\Alg(T)$ and (by assumption) 1 is a
zero-object in $\Alg(T)$ (hence in particular also initial), $T(0)
\cong 1$ and $T(0)$ is final.  
}

A zero object yields, for any pair of objects $X, Y$, a unique ``zero
map'' $0_{X,Y}\colon X \to 0 \to Y$ between them. In a Kleisli
category $\Kl(T)$ for a monad $T$ on $\cat{C}$, this zero map $0_{X,Y}
\colon X \to Y$ is the following map in $\cat{C}$
\begin{equation}
\label{KleisliZeroMap}
\xymatrix{
 0_{X,Y} = 
   \Big(X\ar[r]^-{!_X} & 1 \cong T(0) \ar[r]^-{T(!_Y)} & T(Y)\Big).
}
\end{equation}

\auxproof{
We check that this map is the composite of the two unique maps
$X\dashrightarrow 0 \dashrightarrow Y$ in $\Kl(T)$, where
$X\dashrightarrow 0$ is $X \stackrel{!}{\rightarrow} 1 \congrightarrow 
T(0)$, and $0\dashrightarrow Y$ is $!_{T(Y)}$, as in:
$$
\xymatrix{
X \ar[r]^{!} & 1 \ar[r]^-{\cong} &T(0) \ar[rr]^-{T(!_{TY})} \ar[rd]_-{T(!_{Y})} && T^2(Y) \ar[r]^-{\mu} & T(Y) \\
&&&T(Y) \ar[ru]^-{T(\eta)} \ar[rru]_-{\idmap}
}
$$
} 

\noindent For convenience, we make some basic properties of this zero
map explicit.

\begin{lemma}
\label{zeroproplem}
Assume $T(0)$ is final, for a monad $T$ on $\cat{C}$. The resulting
zero maps $0_{X,Y} \colon X \to T(Y)$ from~\eqref{KleisliZeroMap} make
the following diagrams in \cat{C} commute
$$\xymatrix@C-0.2pc@R-0.5pc{
X\ar[r]^-{0}\ar[dr]_{0} & T^{2}(Y)\ar[d]^{\mu}
  & T(X)\ar[l]_-{T(0)}\ar[dl]^{0} 
 & 
X\ar[r]^-{0}\ar[d]_{f}\ar[dr]_-{0} & T(Y)\ar[d]^{T(f)} 
 & 
X\ar[r]^-{0}\ar[dr]_{0} & T(Y)\ar[d]^{\sigma_Y}  
\\
& T(Y) &
 & 
Y\ar[r]_-{0} & T(Z) &
& S(Y)
}$$

\noindent where $f\colon Y\rightarrow Z$ is a map in \cat{C} and
$\sigma\colon T\rightarrow S$ is a map of monads. \QED
\end{lemma}

\auxproof{
We write $\langle\rangle$ for the unique map 
$X\rightarrow 1\congrightarrow T(0)$ to the final object $0$ in
$\Kl(T)$ in:
$$\begin{array}{rcll}
\mu \after 0_{X,TY}
& = &
\mu \after T(!_{T(Y)}) \after \langle\rangle \\
& = &
\mu \after T(\eta) \after T(!_{Y}) \after \langle\rangle \\
& = &
T(!_{Y}) \after \langle\rangle \\
& = &
0_{X,Y} \\
\mu \after T(0_{X,Y}) 
& = &
\mu_{Y} \after T(T(!_{Y}) \after \langle\rangle) \\
& = &
T(!_{Y}) \after \mu_{0} \after \langle\rangle \\
& = &
T(!_{Y}) \after \langle\rangle 
   & \mbox{since $T(0)$ is final} \\
& = &
0_{T(X), Y} \\
0_{Y,Z} \after f
& = &
T(!_{Z}) \after \langle\rangle \after f \\
& = &
T(!_{Z}) \after \langle\rangle \\
& = &
0_{X,Z} \\
T(f) \after 0_{X,Y}
& = &
T(f) \after T(!_{Y}) \after \langle\rangle \\
& = &
T(!_{Z}) \after \langle\rangle \\
& = &
0_{X,Z} \\
\sigma_{Y} \after 0_{X,Y}
& = &
\sigma_{Y} \after T(!_{Y}) \after \langle\rangle \\
& = &
S(!_{Y}) \after \sigma_{0} \after \langle\rangle \\
& = &
S(!_{Y}) \after \langle\rangle 
   & \mbox{by uniqueness of maps to 1 in} \\
& & & T(0) \stackrel{\sigma_0}{\rightarrow} S(0) \congrightarrow 1 \\
& = &
0_{X,Y}.
\end{array}$$
}

Still assuming that $T(0)$ is final, the zero
map~\eqref{KleisliZeroMap} enables us to define a canonical map
\begin{equation}
\label{bcDef}
\xymatrix@C+1pc{
\bc \stackrel{\textrm{def}}{=} \Big(T(X+Y)\ar[rr]^-{
    \tuple{\mu \after T(p_1)}{\mu \after T(p_2)}} 
& & T(X) \times T(Y)\Big),
}
\end{equation}

\noindent where
\begin{equation}
\label{kleisliprojdef}
\xymatrix@C-0.5pc{
p_{1} \stackrel{\textrm{def}}{=} 
   \Big(X+Y \ar[rr]^-{\cotuple{\eta}{0_{Y,X}}} & & T(X)\Big),
&
p_{2} \stackrel{\textrm{def}}{=} 
   \Big(X+Y \ar[rr]^-{\cotuple{0_{X,Y}}{\eta}} & & T(Y)\Big).
}
\end{equation}

\noindent Here we assume that the underlying category \cat{C} has both
finite products and finite coproducts. The abbreviation ``\bc'' stands
for ``bicartesian'', since this maps connects the coproducts and
products. The auxiliary maps $p_{1},p_{2}$ are sometimes called
projections, but should not be confused with the (proper) projections
$\pi_{1},\pi_{2}$ associated with the product $\times$ in \cat{C}.

We continue by listing a series of properties of this map \bc that
will be useful in what follows.

\begin{lemma}
\label{bcproplem}
In the context just described, the map $\bc \colon T(X+Y) \to
T(X)\times T(Y)$ in~\eqref{bcDef} has the following properties.
{\renewcommand{\theenumi}{(\roman{enumi})}
\begin{enumerate}
\item\label{natbcprop} This \bc is a natural transformation,
and it commutes with any monad map $\sigma\colon T\rightarrow S$, as in:
$$\xymatrix@C-.5pc{
T(X+Y)\ar[r]^-{\bc}\ar[d]|{T(f+g)} & T(X)\times T(Y)\ar[d]|{T(f)\times T(g)} 
&
T(X+Y)\ar[r]^-{\bc}\ar[d]|{\sigma_{X+Y}} 
   & T(X)\times T(Y)\ar[d]|{\sigma_{X}\times\sigma_{Y}}
\\
T(U+V)\ar[r]^-{\bc} & T(U)\times T(V)
&
S(X+Y)\ar[r]^-{\bc} & S(X)\times S(Y)
}$$

\item\label{natmonoidalbcprop} It also commutes with the monoidal
  isomorphisms (for products and coproducts in $\cat{C}$):
$$\begin{array}{c}
\xymatrix@C-.5pc{
T(X+0)\ar[r]^-{\bc}\ar[dr]_{T(\rho)}^{\cong} & T(X)\times T(0)\ar[d]^{\rho}_{\cong}
&
T(X+Y)\ar[r]^-{\bc}\ar[d]_{T(\cotuple{\kappa_2}{\kappa_1})}^{\cong}
   & T(X)\times T(Y)\ar[d]^{\tuple{\pi_2}{\pi_1}}_{\cong}
\\
& T(X)
&
T(Y+X)\ar[r]^-{\bc} & T(Y)\times T(X)
} \\
\\[-1em]
\xymatrix@C-.5pc{
T((X+Y)+Z) \ar[r]^-{\bc} \ar[d]_-{T(\alpha)}^-{\cong} & T(X+Y)\times T(Z) \ar[r]^-{\bc \times id} & (T(X)\times T(Y))\times T(Z) \ar[d]^-{\alpha}_-{\cong}
\\
T(X+(Y+Z)) \ar[r]^-{\bc} & T(X)\times T(Y+Z) \ar[r]^-{id \times \bc} & T(X)\times (T(Y) \times T(Z))
}
\end{array}$$

\item\label{comEtaMubcprop} The map \bc interacts with $\eta$ and
  $\mu$ in the following manner:
$$
\xymatrix@C-.5pc{
X+Y \ar[d]_-{\eta} \ar[dr]^-{\tuple{p_1}{p_2}} &\\
T(X+Y) \ar[r]_-{\bc} & T(X) \times T(Y)
}
$$
$$\hspace*{-1em}\xymatrix@C-.5pc{
T^2(X+Y) \ar[r]^-{\mu} \ar[d]_-{T(\bc)}& T(X+Y) \ar[dd]^{\bc} 
\hspace*{-1em} & \hspace*{-1em}
T(T(X) + T(Y)) \ar[r]^-{\bc} \ar[d]|{T(\cotuple{T(\kappa_1)}{T(\kappa_2)})}& T^2(X) \times T^2(Y) \ar[dd]^{\mu \times \mu}
\\
T(T(X)\times T(Y)) \ar[d]|{{\tuple{T(\pi_1)}{T(\pi_2)}}} & 
\hspace*{-1em} & \hspace*{-1em}
T^2(X+Y) \ar[d]_-{\mu}&
\\
T^2(X) \times T^2(Y)\ar[r]^-{\mu \times \mu}& T(X) \times T(Y)
\hspace*{-1em} & \hspace*{-1em}
T(X+Y) \ar[r]^-{\bc} & T(X) \times T(Y)
}
$$

\item\label{comStbcprop} If $\cat{C}$ is a distributive category,
  $\bc$ commutes with strength $\st$ as follows:
$$\hspace*{-1em}\xymatrix@C-1pc{
T(X+Y)\times Z \ar[r]^-{\bc \times \id} \ar[d]_-{\st}& 
(T(X) \times T(Y))\times Z \ar[r]^-{\mathit{dbl}} & 
(T(X) \times Z) \times (T(Y) \times Z) \ar[d]^-{\st \times \st}
\\
T((X+Y) \times Z) \ar[r]^-{\cong} &T((X\times Z) + (Y\times Z)) \ar[r]^-{\bc} &T(X \times Z) \times T(Y \times Z)
}
$$ 

\noindent where $\mathit{dbl}$ is the ``double'' map
$\tuple{\pi_1 \times \id}{\pi_2 \times \id}\colon (A\times B)\times C
\rightarrow (A\times C)\times (B\times C)$.
\end{enumerate}
}
\end{lemma}

\begin{proof}
These properties are easily verified, using Lemma~\ref{zeroproplem}
and the fact that the projections $p_i$ are natural, both in $\cat{C}$
and in $\Kl(T)$. \QED
\end{proof}

\auxproof{
Naturality of the projections in $\cat{C}$: for $f\colon X\rightarrow U$ and
$g\colon Y\rightarrow V$ one has:
$$\begin{array}{rcl}
p_{1} \after (f+g)
& = &
\cotuple{\eta}{0} \after (f+g) \\
& = &
\cotuple{\eta \after f}{0 \after g} \\
& = &
\cotuple{T(f) \after \eta}{T(f) \after 0} \\
& = &
T(f) \after \cotuple{\eta}{0} \\
& = &
T(f) \after p_{1}.
\end{array}$$

\noindent Similarly, for $f\colon X\rightarrow T(U)$ and $g\colon
Y\rightarrow T(V)$ one has naturality wrt.\ the coproducts in
$\Kl(T)$ since:
$$\begin{array}{rcl}
p_{1} \klafter (f+g) 
& = &
\mu \after T(\cotuple{\eta}{0}) \after 
   \cotuple{T(\kappa_{1}) \after f}{T(\kappa_{2}) \after g} \\
& = &
\mu \after \cotuple{T(\eta) \after f}{T(0) \after g} \\
& = &
\cotuple{\mu \after T(\eta) \after f}{\mu \after T(0) \after g} \\
& = &
\cotuple{\mu \after \eta \after f}{0 \klafter g} \\
& = &
\cotuple{\mu \after T(f) \after \eta}{0} \\
& = &
\cotuple{\mu \after T(f) \after \eta}{\mu \after T(f) \after 0}
\\
& = &
\mu \after T(f) \after \cotuple{\eta}{0} \\
& = &
f \klafter p_{1}
\end{array}$$

For naturality of \bc assume $f\colon X\rightarrow U$ and
$g\colon Y\rightarrow V$. Then:
$$\begin{array}{rcl}
(T(f)\times T(g)) \after \bc
& = &
(T(f) \times T(g)) \after \tuple{\mu \after T(p_{1})}{\mu \after T(p_{2})} \\
& = &
\tuple{T(f) \after \mu \after T(p_{1})}{T(g) \after \mu \after T(p_{2})} \\
& = &
\tuple{\mu \after T(T(f) \after p_{1})}{\mu \after T(T(g) \after p_{2})} \\
& = &
\tuple{\mu \after T(p_{1} \after (f+g))}{\mu \after T(p_{2} \after (f+g))} \\
& = &
\tuple{\mu \after T(p_{1})}{\mu \after T(p_{2})} \after T(f+g) \\
& = &
\bc \after T(f+g).
\end{array}$$

The map $\bc$ commutes with a monad map $\sigma$ since:
$$\xymatrix@C+1pc{
	T(X+Y) \ar[rr]^-{\tuple{T(\cotuple{\eta^T}{0})}{T(\cotuple{0}{\eta^T})}} \ar[rrd]|{\tuple{T(\cotuple{\eta^S}{0})}{T(\cotuple{0}{\eta^S})}} \ar[dd]_{\sigma} && T^2(X) \times T^2(Y) \ar[r]^{\mu \times \mu} \ar[d]^{T\sigma \times T\sigma} & T(X) \times T(Y) \ar[dd]^{\sigma \times \sigma} \\
	&& TS(X) \times TS(Y) \ar[d]^{\sigma_{S(X)} \times \sigma_{S(Y)}} &\\
	S(X+Y) \ar[rr]^-{\tuple{S(\cotuple{\eta^S}{0})}{S(\cotuple{0}{\eta^S})}} && S^2(X) \times S^2(Y) \ar[r]^{\mu \times \mu} & S(X) \times S(Y)
}$$

Commutation with $\rho$'s:
$$\begin{array}{rcl}
\pi_{1} \after \bc
& = &
\pi_{1} \after \tuple{\mu \after T(p_{1})}{\mu \after T(p_{2})} \\
& = &
\mu \after T(\cotuple{\eta}{0_{0,X}}) \\
& = &
\mu \after T(\cotuple{\eta}{\eta \after \,!_{X}}) \\
& = &
\mu \after T(\eta \after \cotuple{\idmap}{!}) \\
& = &
\cotuple{\idmap}{!}.
\end{array}$$

Commutation with swap maps:
$$\begin{array}{rcl}
\tuple{\pi_2}{\pi_1} \after \bc
& = &
\tuple{\pi_2}{\pi_1} \after \tuple{\mu \after T(p_{1})}{\mu \after T(p_{2})} \\
& = &
\tuple{\mu \after T(p_{2})}{\mu \after T(p_{1})} \\
& = &
\tuple{\mu \after T(\cotuple{0}{\eta})}{\mu \after T(\cotuple{\eta}{0})} \\
& = &
\tuple{\mu \after T(\cotuple{\eta}{0} \after \cotuple{\kappa_2}{\kappa_1})}
   {\mu \after T(\cotuple{0}{\eta} \after \cotuple{\kappa_2}{\kappa_1})} \\
& = &
\tuple{\mu \after T(p_{1})}{\mu \after T(p_{2})} \after 
   T(\cotuple{\kappa_2}{\kappa_1}) \\
& = &
\bc \after T(\cotuple{\kappa_2}{\kappa_1})
\end{array}$$

Commutation with $\alpha$':
$$\begin{array}{rcl}
\lefteqn{\tuple{\pi_{1} \after \pi_{1}}{
   \tuple{\pi_{2}\after \pi_{1}}{\pi_2}} \after (\bc\times\idmap) \after \bc} \\
& = &
\tuple{\pi_{1} \after \bc \after \pi_{1}}{
   \tuple{\pi_{2} \after \bc \after \pi_{1}}{\pi_{2}}} \after \bc \\
& = &
\tuple{\mu \after T(p_{1}) \after \pi_{1} \after \bc}{
   \tuple{\mu \after T(p_{2}) \after \pi_{1} \after \bc}
   {\pi_{2} \after \bc}}  \\
& = &
\tuple{\mu \after T(p_{1}) \after \mu \after T(p_{1})}{
   \tuple{\mu \after T(p_{2}) \after \mu \after T(p_{1})}
   {\mu \after T(p_{2})}}  \\
& = &
\tuple{\mu \after \mu \after T(T(p_{1}) \after p_{1})}{
   \tuple{\mu \after \mu \after T(T(p_{2}) \after p_{1})}
   {\mu \after T(p_{2})}}  \\
& = &
\tuple{\mu \after T(\mu) \after T(T(p_{1}) \after \cotuple{\eta}{0})}{
   \tuple{\mu \after T(\mu) \after T(T(p_{2}) \after \cotuple{\eta}{0})}
   {\mu \after T(p_{2})}}  \\
& = &
\tuple{\mu \after T(\mu \after \cotuple{\eta \after p_{1}}{0})}{
   \tuple{\mu \after T(\mu \after \cotuple{\eta\after p_{2}}{0})}
   {\mu \after T(p_{2})}}  \\
& = &
\tuple{\mu \after T(\cotuple{p_{1}}{0})}{
   \tuple{\mu \after T(\cotuple{p_{2}}{0})}
   {\mu \after T(p_{2})}}  \\
& \stackrel{(*)}{=} &
\tuple{\mu \after T(\cotuple{p_1}{0})}
         {\tuple{\mu \after T(\cotuple{\cotuple{0}{\eta}}{0})}
              {\mu \after T(\cotuple{\cotuple{0}{0}}{\eta})}} \\
& = &
\tuple{\mu \after T(\cotuple{p_1}{0})}
         {\tuple{\mu \after T(\mu) \after 
             T(\cotuple{\cotuple{0}{\eta\after\eta}}{0 \after \eta})}
              {\mu \after T(\mu) \after 
             T(\cotuple{\cotuple{0}{0\after\eta}}{\eta\after\eta})}} \\
& = &
\tuple{\mu \after T(\cotuple{p_1}{0})}
         {\tuple{\mu \after \mu \after T^{2}(p_{1})}
              {\mu \after \mu \after T^{2}(p_{2})}
     \after T(\cotuple{\cotuple{0}{T(\kappa_{1}) \after \eta}}
        {T(\kappa_{2}) \after \eta})} \\
& = &
\tuple{\mu \after T(\cotuple{p_1}{0})}
         {\tuple{\mu \after T(p_{1}) \after \mu}{\mu \after T(p_{2}) \after \mu}
     \after T(\cotuple{\cotuple{0}{T(\kappa_{1}) \after \eta}}
        {T(\kappa_{2}) \after \eta})} \\
& = &
\tuple{\mu \after T(\cotuple{p_1}{0})}
         {\bc \after \mu \after T(\cotuple{\cotuple{0}{\eta\after\kappa_{1}}}
        {\eta\after\kappa_{2}})} \\
& = &
(\idmap\times\bc) \after 
   \tuple{\mu \after T(
      \cotuple{\cotuple{\eta}{0\after\kappa_{1}}}
        {0\after\kappa_{2}})}
         {\mu \after T(
      \cotuple{\cotuple{0}{\eta\after\kappa_{1}}}
        {\eta\after\kappa_{2}})} \\
& = &
(\idmap\times\bc) \after \\
& & \quad
   \tuple{\mu \after T(p_{1} \after 
      \cotuple{\cotuple{\kappa_1}{\kappa_{2}\after\kappa_{1}}}
        {\kappa_{2}\after\kappa_{2}})}
         {\mu \after T(p_{2} \after 
      \cotuple{\cotuple{\kappa_1}{\kappa_{2}\after\kappa_{1}}}
        {\kappa_{2}\after\kappa_{2}})} \\
& = &
(\idmap\times\bc) \after \bc \after 
   T(\cotuple{\cotuple{\kappa_1}{\kappa_{2}\after\kappa_{1}}}
     {\kappa_{2}\after\kappa_{2}}).
\end{array}$$

\noindent The marked equation $\stackrel{(*)}{=}$ use $[0,0] = 0$,
which holds since $[0,0] \after \kappa_{i} = 0 = 0 \after \kappa_{i}$.

Commutation with $\eta$:
$$\begin{array}{rcl}
\bc \after \eta
& = &
\tuple{\mu \after T(p_{1})}{\mu \after T(p_{2})} \after \eta \\
& = &
\tuple{\mu \after T(p_{1}) \after \eta }{\mu \after T(p_{2}) \after \eta } \\
& = &
\tuple{\mu \after \eta \after p_{1}}{\mu \after \eta \after p_{2}} \\
& = &
\tuple{p_1}{p_2}.
\end{array}$$

Commutation with $\mu$, involving two diagrams:
$$\begin{array}{rcl}
\bc \after \mu
& = &
\tuple{\mu \after T(p_{1})}{\mu \after T(p_{2})} \after \mu \\
& = &
\tuple{\mu \after T(p_{1}) \after \mu}{\mu \after T(p_{2})  \after \mu} \\
& = &
\tuple{\mu \after \mu \after T^{2}(p_{1})}
   {\mu \after \mu \after T^{2}(p_{2})} \\
& = &
\tuple{\mu \after T(\mu) \after T^{2}(p_{1})}
   {\mu \after T(\mu) \after T^{2}(p_{2})} \\
& = &
\tuple{\mu \after T(\mu \after T(p_{1}))}{\mu \after T(\mu \after T(p_{2}))} \\
& = &
\mu\times\mu \after \tuple{T(\mu \after T(p_{1}))}{T(\mu \after T(p_{2}))} \\
& = &
\mu\times\mu \after \tuple{T(\pi_{1})}{T(\pi_{2})} \after 
   T(\tuple{\mu \after T(p_{1})}{\mu \after T(p_{2})}) \\
& = &
\mu\times\mu \after \tuple{T(\pi_{1})}{T(\pi_{2})} \after T(\bc).
\end{array}$$

\noindent And similarly,
$$\begin{array}{rcl}
\lefteqn{\bc \after \mu \after T(\cotuple{T(\kappa_{1})}{T(\kappa_{2})})} \\
& = &
\tuple{\mu \after T(p_{1})}{\mu \after T(p_{2})} \after 
   \mu \after T(\cotuple{T(\kappa_{1})}{T(\kappa_{2})}) \\
& = &
\tuple{\mu \after T(p_{1}) \after \mu \after 
   T(\cotuple{T(\kappa_{1})}{T(\kappa_{2})})}
 {\mu \after T(p_{2}) \after \mu \after 
   T(\cotuple{T(\kappa_{1})}{T(\kappa_{2})})} \\
& = &
\tuple{\mu \after \mu \after T^{2}(p_{1}) \after
   T(\cotuple{T(\kappa_{1})}{T(\kappa_{2})})}
 {\mu \after \mu \after T^{2}(p_{2}) \after 
   T(\cotuple{T(\kappa_{1})}{T(\kappa_{2})})} \\
& = &
\tuple{\mu \after T(\mu) \after 
   T(\cotuple{T(p_{1}\after\kappa_{1})}{T(p_{1}\after\kappa_{2})})}
 {\mu \after T(\mu) \after 
   T(\cotuple{T(p_{2} \after \kappa_{1})}{T(p_{2}\after\kappa_{2})})} \\
& = &
\tuple{\mu \after T(\mu \after \cotuple{T(\eta)}{T(0)})}
 {\mu \after T(\mu \after \cotuple{T(0)}{T(\eta)})} \\
& = &
\tuple{\mu \after T(\cotuple{\mu \after T(\eta)}{\mu \after T(0)})}
 {\mu \after T(\cotuple{\mu \after T(0)}{\mu \after T(\eta)})} \\
& = &
\tuple{\mu \after T(\cotuple{\mu \after \eta}{\mu \after 0})}
 {\mu \after T(\cotuple{\mu \after 0}{\mu \after \eta})} \\
& = &
\tuple{\mu \after T(\mu) \after T(\cotuple{\eta}{0})}
 {\mu \after T(\mu) \after T(\cotuple{0}{\eta})} \\
& = &
\tuple{\mu \after \mu \after T(p_{1})}
 {\mu \after \mu \after T(p_{2})} \\
& = &
\mu\times\mu \after \tuple{\mu \after T(p_{1})}{\mu \after T(p_{2})} \\
& = &
\mu\times\mu \after \bc.
\end{array}$$

Finally, for commutation of the diagram with the doubling map $\mathit{dbl}$,
we first note that in a distributive category the following diagram
commutes.
$$\xymatrix{
(X+Y)\times Z\ar[rr]^-{p_{1}\times\idmap} & & T(X)\times Z\ar[d]^{\st} \\
(X\times Z)+(Y\times Z)\ar[u]^{\mathit{dist}}_{\cong}\ar[rr]^-{p_1}
  & & T(X\times Z)
}$$

\noindent where $\mathit{dist} =
\cotuple{\kappa_{1}\times\idmap}{\kappa_{2}\times\idmap}$ is the
canonical distribution map. Commutation of this diagram holds because:
$$\begin{array}{rcl}
\st \after p_{1}\times\idmap \after \mathit{dist} 
& = &
\cotuple{\st \after p_{1}\times\idmap \after \kappa_{1}\times\idmap}
   {\st \after p_{1}\times\idmap \after \kappa_{2}\times\idmap} \\
& = &
\cotuple{\st \after \eta\times\idmap}{\st \after 0\times\idmap} \\
& \stackrel{(*)}{=} &
\cotuple{\eta}{0} \\
& = &
p_{1},
\end{array}$$

\noindent where the marked equation holds because $\st \after
0\times\idmap = 0$, which follows from distributivity $0\times Z \cong
0$, so that $T(0\times Z)$ is final in:
$$\xymatrix@C+1pc{
Y\times Z\ar[r]^-{!\times\idmap}\ar@ /^5ex/[rr]^-{0\times\idmap}\ar[d]_{!} & 
   T(0)\times Z\ar[r]^-{T(!)\times\idmap}\ar[d]^{\st} & 
   T(X)\times Z\ar[d]^{\st} \\
T(0)\ar[r]_-{\cong}\ar@ /_5ex/[rr]_-{T(!)} & 
   T(0\times Z)\ar[r]_-{T(!\times\idmap)} & T(X\times Z)
}$$

\noindent Now we can prove that the diagram in
Lemma~\ref{bcproplem}.\ref{comStbcprop} commutes:
$$\begin{array}{rcl}
\lefteqn{\st\times\st \after \mathit{dbl} \after \bc\times\idmap} \\
& = &
\st\times\st \after \tuple{\pi_{1}\times\idmap}{\pi_{2}\times\idmap}
   \after \bc\times\idmap \\
& = &
\tuple{\st \after (\mu \after T(p_{1}))\times\idmap}
     {\st \after (\mu \after T(p_{2})) \times\idmap} \\
& = &
\tuple{\mu \after T(\st) \after \st \after T(p_{1})\times\idmap}
     {\mu \after T(\st) \after \st \after T(p_{2}) \times\idmap} \\
& = &
\tuple{\mu \after T(\st) \after T(p_{1}\times\idmap) \after \st}
     {\mu \after T(\st) \after T(p_{2} \times\idmap) \after \st} \\
& = &
\tuple{\mu \after T(p_{1} \after \mathit{dist}^{-1})}
   {\mu \after T(p_{2} \after \mathit{dist}^{-1})} \after \st\\
& = &
\bc \after T(\mathit{dist}^{-1}) \after \st.
\end{array}$$
}

The definition of the map $\bc$ also makes sense for arbitrary
set-indexed (co)products (see~\cite{Jacobs10a}), but here we only
consider finite ones. Such generalised $\bc$-maps also satisfy
(suitable generalisations of) the properties in Lemma~\ref{bcproplem}
above.

We will study monads for which the canonical map $\bc$ is an
isomorphism. Such monads will be called `additive monads'.

\begin{definition}
  A monad $T$ on a category $\cat{C}$ with finite products $(\times,
  1)$ and finite coproducts $(+,0)$ will be called \emph{additive} if
  $T(0)\cong 1$ and if the canonical map $\bc\colon T(X+Y) \to T(X)
  \times T(Y)$ from~\eqref{bcDef} is an isomorphism.
\end{definition}

We write $\cat{AMnd}(\cat{C})$ for the category of additive monads on
$\cat{C}$ with monad morphism between them, and similarly
$\cat{ACMnd}(\cat{C})$ for the category of additive and commutative
monads on $\cat{C}$.

A basic result is that additive monads $T$ induce a commutative monoid
structure on objects $T(X)$. This result is sometimes taken as
definition of additivity of monads
(\textit{cf.}~\cite{GoncharovSM09}).

\begin{lemma}
\label{AdditiveMonadMonoidLem}
\label{Mnd2MonLem2}
Let $T$ be an additive monad on a category $\cat{C}$ and $X$ an object
of $\cat{C}$. There is an addition $+$ on $T(X)$ given by
$$
\xymatrix{
+ \stackrel{\textrm{def}}{=} \Big(T(X) \times T(X) \ar[r]^-{\bc^{-1}} &
   T(X+X) \ar[r]^-{T(\nabla)} & T(X)\Big),
}
$$

\noindent where $\nabla = \cotuple{\idmap}{\idmap}$. Then:
{\renewcommand{\theenumi}{(\roman{enumi})}
\begin{enumerate}
\item this $+$ is commutative and associative,
\item and has unit $0_{1,X}: 1 \to T(X)$;
\item this monoid structure is preserved by maps $T(f)$ as well as by
  multiplication $\mu$;
\item the mapping $(T,X) \mapsto (T(X), +, 0_{1,X})$ yields a
  functor $\Evx \colon \cat{AMnd}(\cat{C}) \times \cat{C} \to
  \cat{CMon}(\cat{C})$.
\end{enumerate}}
\end{lemma}

\begin{proof}
The first three statements follow by the properties of $\bc$ from
Lemma~\ref{bcproplem}. For instance, $0$ is a (right) unit for $+$ as
demonstrated in the following diagram.
$$\xymatrix{
T(X) \ar[rr]^-{\rho^{-1}}_-{\cong}\ar[drr]^(0.6){\cong}_-{T(\rho^{-1})} & & 
   T(X)\times T(0) \ar[rr]^-{\idmap \times T(!)} \ar[d]^-{\bc^{-1}} & & 
   T(X)\times T(X) \ar[d]^-{\bc^{-1}}\ar@ /^10ex/[dd]^{+}
\\
& & T(X+0) \ar[rr]^-{T(\idmap +\, !)} \ar[drr]_-{T(\rho)}^-{\cong} & & 
   T(X+X) \ar[d]^-{T(\nabla)}
\\
&&&&T(X)
}
$$
Regarding (iv) we define, for a pair of morphisms $\sigma: T \to S$ in $\cat{AMnd}(\cat{C})$ and $f\colon X \to Y$ in $\cat{C}$, 
$$\Evx((\sigma, f)) = \sigma \after T(f) \colon T(X) \to S(Y),$$ which
is equal to $S(f) \after \sigma$ by naturality of
$\sigma$. Preservation of the unit by $\Evx((\sigma, f))$ follows from
Lemma \ref{zeroproplem}. The following diagram demonstrates
that addition is preserved.
$$
\xymatrix{
T(X) \times T(X) \ar[rr]^-{T(f)\times T(f)} \ar[d]_-{\bc^{-1}} && T(Y)\times T(Y) \ar[r]^{\sigma \times \sigma} \ar[d]^-{\bc^{-1}} & S(Y)\times S(Y) \ar[dd]^-{\bc^{-1}}
\\
T(X+X) \ar[rr]^-{T(f+f)} \ar[rrd]_-{\sigma} \ar[dd]_-{T(\nabla)} && T(Y+Y) \ar[rd]^-{\sigma}
\\
&&S(X+X) \ar[r]^-{S(f+f)} \ar[d]_-{S(\nabla)} & S(Y+Y) \ar[d]^-{S(\nabla)}
\\
T(X) \ar[rr]^-{\sigma} && S(X) \ar[r]^-{S(f)} & S(Y)
}
$$ 

\noindent where we use point~\ref{natbcprop} of Lemma \ref{bcproplem}
and the naturality of $\sigma$. It is easily checked that this mapping
defines a functor.\QED
\end{proof}

\auxproof{
Commutativity holds since $+ \after \tuple{\pi_2}{\pi_1} = +$, see:
$$
\xymatrix{
T(X)\times T(Y) \ar[r]^-{\bc^{-1}} \ar[d]_-{\tuple{\pi_2}{\pi_1}} & 
   T(X+X) \ar[d]_-{T(\cotuple{\kappa_2}{\kappa_1})} \ar[rd]^-{T(\nabla)}
\\
T(X)\times T(X) \ar[r]_-{\bc^{-1}} & T(X+X) \ar[r]_-{T(\nabla)} & T(X)
}
$$

\noindent Associativity:
$$
\xymatrix{
(T(X)\times T(X))\times T(X) \ar[r]^-{\alpha} \ar[d]_-{\bc^{-1} \times\idmap} & 
   T(X) \times (T(X)\times T(X)) \ar[r]^-{\idmap \times \bc^{-1}} & 
   T(X) \times T(X+X) \ar[d]^-{\bc^{-1}} \ar[r]^-{\idmap \times T(\nabla)} & 
   T(X) \times T(X) \ar[d]^-{\bc^{-1}}
\\
T(X+X) \times T(X) \ar[d]_-{T(\nabla) \times \idmap} \ar[r]^-{\bc^{-1}} & 
   T((X+X)+X) \ar[r]^-{\alpha'} \ar[d]^-{T(\nabla +\idmap)} & 
   T(X+(X+X)) \ar[r]^-{T(\idmap+\nabla)} & T(X+X) \ar[d]^-{T(\nabla)}
\\
T(X)\times T(X) \ar[r]^-{\bc^{-1}} & T(X+X) \ar[rr]^-{T(\nabla)} && T(X)
}
$$

\noindent Preservation by maps $T(f)$:
$$\begin{array}{rcl}
+ \after (T(f) \times T(f))
& = &
T(\nabla) \after \bc^{-1} \after (T(f)\times T(f)) \\
& = &
T(\nabla) \after T(f+f) \after \bc^{-1} \\
& = &
T(f) \after T(\nabla) \after \bc^{-1} \\
& = &
T(f) \after +
\end{array}$$

\noindent Similarly for $\mu$,
$$\begin{array}{rcl}
+ \after (\mu\times\mu) 
& = &
T(\nabla) \after \bc^{-1} \after (\mu\times\mu) \\
& = &
T(\nabla) \after \mu \after T(\cotuple{T(\kappa_{1})}{T(\kappa_{2})})
   \after \bc^{-1} \\
& = &
\mu \after T^{2}(\nabla) \after 
   T(\cotuple{T(\kappa_{1})}{T(\kappa_{2})}) \after \bc^{-1} \\
& = &
\mu \after 
   T(T(\nabla) \after \cotuple{T(\kappa_{1})}{T(\kappa_{2})}) \after \bc^{-1} \\
& = &
\mu \after 
   T(\cotuple{\idmap}{\idmap}) \after \bc^{-1} \\
& = &
\mu \after +.
\end{array}$$
}

\auxproof{
Direct proof of the fact that T(X) with the given addition is a commutative monoid:

To show that $0_{1,X}$ indeed serves as a unit, consider:
$$
\xymatrix{
T(X) \ar[r]^-{\rho^{-1}} & T(X)\times 1 \ar[r]^-{id \times !^{-1}} & T(X)\times T(0) \ar[rr]^-{id \times T(!)} \ar[d]_-{\bc^{-1}}&& T(X)\times T(X) \ar[d]^-{\bc^{-1}}
\\
&&T(X+0) \ar[rr]^-{T(id + !)} \ar[drr]_-{T(\cotuple{id}{!})} && T(X+X) \ar[d]^-{T(\nabla)}
\\
&&&&T(X)
}
$$
Note that
$$
\begin{array}{rcl}
\rho \after \id\times ! \after \bc &=& \pi_1 \after \tuple{\mu \after T(p_1)}{!}\\
	&=&
\mu \after T(\cotuple{\eta}{0})\\
	&=&
\mu \after T(\eta) \after T(\cotuple{id}{!})\\
	&=&
T(\cotuple{id}{!}),
\end{array}
$$
which proves the right identity law. The left identity law is shown similarly. 

Associativity:
$$
\xymatrix{
(T(X)\times T(X))\times T(X) \ar[r]^-{\alpha} \ar[d]_-{\bc^{-1} \times id} & T(X) \times (T(X)\times T(X)) \ar[r]^-{\idmap \times \bc^{-1}} & T(X) \times T(X+X) \ar[d]^-{\bc^{-1}} \ar[r]^-{id \times T(\nabla)} & T(X) \times T(X) \ar[d]^-{\bc^{-1}}
\\
T(X+X) \times T(X) \ar[d]_-{T(\nabla) \times id} \ar[r]^-{\bc^{-1}} & T((X+X)+X) \ar[r]^-{\alpha'} \ar[d]^-{T(\nabla +id)} & T(X+(X+X)) \ar[r]^-{T(id+\nabla)} & T(X+X) \ar[d]^-{T(\nabla)}
\\
T(X)\times T(X) \ar[r]^-{\bc^{-1}} & T(X+X) \ar[rr]^-{T(\nabla)} && T(X)
}
$$
Commutativity:
$$
\xymatrix{
T(X)\times T(Y) \ar[r]^-{\bc^{-1}} \ar[d]_-{\tuple{\pi_2}{\pi_1}} & T(X+X) \ar[d]_-{T(\cotuple{\kappa_2}{\kappa_1})} \ar[rd]^-{T(\nabla)}
\\
T(X)\times T(X) \ar[r]^-{\bc^{-1}} & T(X+X) \ar[r]^-{T(\nabla)} & T(X)
}
$$
}

By Lemma \ref{KleisliStructLem}, for a monad $T$ on a category
$\cat{C}$ with finite coproducts, the Kleisli construction yields a
category $\Kl(T)$ with finite coproducts. Below we will prove that,
under the assumption that $\cat{C}$ also has products, these
coproducts form biproducts in $\Kl(T)$ if and only if $T$ is
additive. Again, as in Lemma \ref{zerolem}, a related result holds for
the category $\Alg(T)$.

\begin{definition}
\label{biprodcatdef}
A \emph{category with biproducts} is a category $\cat{C}$ with a zero
object $0 \in \cat{C}$, such that, for any pair of objects $A_1, A_2
\in \cat{C}$, there is an object $A_1 \oplus A_2 \in \cat{C}$ that is
both a product with projections $\pi_i: A_1 \oplus A_2 \to A_i$ and a
coproduct with coprojections $\kappa_i: A_i \to A_1 \oplus A_2$, such
that
$$\begin{array}{rcl}
\pi_j \after \kappa_i & = & \left\{
\begin{array}{ll}
	\idmap_{A_i} & \text{if }\, i = j \\
	0_{A_i, A_j} & \text{if }\, i \ne j.
\end{array} \right.
\end{array}$$
\end{definition}

\begin{theorem}
\label{AMnd2BCat}
For a monad $T$ on a category $\cat{C}$ with finite products $(\times,
1)$ and coproducts $(+,0)$, the following are equivalent.
\begin{enumerate}
\renewcommand{\theenumi}{(\roman{enumi})}
\item $T$ is additive;
\item the coproducts in $\cat{C}$ form biproducts in 
the Kleisli category $\Kl(T)$;
\item the products in $\cat{C}$ yield biproducts in 
the category of Eilenberg-Moore algebras $\Alg(T)$.
\end{enumerate}
\end{theorem}

Here we shall only use this result for Kleisli categories, but we
include the result for algebras for completeness.

\begin{proof}
  First we assume that $T$ is additive and show that $(+,0)$ is a
  product in $\Kl(T)$. As projections we take the maps $p_i$
  from~\eqref{kleisliprojdef}. For Kleisli maps $f\colon Z\rightarrow
  T(X)$ and $g\colon Z\rightarrow T(Y)$ there is a tuple via the map
  \bc, as in
$$\xymatrix{
\tuple{f}{g}_{\Kl} \stackrel{\textrm{def}}{=} \Big(Z \ar[r]^-{\tuple{f}{g}} & 
   T(X) \times T(Y) \ar[r]^-{\bc^{-1}} & T(X + Y)\Big).
}$$

\noindent One obtaines $p_{1} \klafter \tuple{f}{g}_{\Kl} = \mu \after
T(p_{1}) \after \bc^{-1} \after \tuple{f}{g} = \pi_{1} \after \bc
\after \bc^{-1} \after \tuple{f}{g} = \pi_{1} \after \tuple{f}{g} =
f$. Remaining details are left to the reader.

Conversely, assuming that the coproduct $(+,0)$ in $\cat{C}$ forms a biproduct in $\Kl(T)$, we have to show that the bicartesian map $\bc \colon T(X+Y) \to T(X) \times T(Y)$ is an isomorphism.
As $+$ is a biproduct, there exist projection maps $q_i \colon X_1+X_2 \to X_i$ in $\Kl(T)$ satisfying
$$\begin{array}{rcl}
q_j \klafter \kappa_i & = & \left\{
\begin{array}{ll}
	\idmap_{X_i} & \text{if }\, i = j \\
	0_{X_i, X_j} & \text{if }\, i \ne j.
\end{array} \right.
\end{array}$$ 

\noindent From these conditions it follows that $q_i = p_i$, where
$p_i$ is the map defined in~\eqref{kleisliprojdef}.  The ordinary
projection maps $\pi_i \colon T(X_1) \times T(X_2) \to T(X_i)$ are
maps $T(X_1) \times T(X_2) \to X_i$ in $\Kl(T)$. Hence, as + is a
product, there exists a unique map $h \colon T(X_1) \times T(X_2) \to
X_1 + X_2$ in $\Kl(T)$, \textit{i.e.} $h \colon T(X_1) \times T(X_2)
\to T(X_1 + X_2)$ in $\cat{C}$, such that $p_1 \klafter h = \pi_1$ and
$p_2 \klafter h = \pi_2$. It is readily checked that this map $h$ is
the inverse of $\bc$.

To prove the equivalence of $\textit{(i)}$ and $\textit{(iii)}$, first
assume that the monad $T$ is additive. In the category $\Alg(T)$ of
algebras there is the standard product
$$\xymatrix@C-0.5pc{
\Big(T(X)\ar[r]^-{\alpha} & X\Big)\times\Big(T(Y)\ar[r]^-{\beta} & Y\Big)
\stackrel{\textrm{def}}{=}
\Big(T(X\times Y)\ar[rrr]^-{\tuple{\alpha\after T(\pi_{1})}
   {\beta\after T(\pi_{2})}} & & & X\times Y\Big).
}$$

\noindent In order to show that $\times$ also forms a coproduct in
$\Alg(T)$, we first show that for an arbitrary algebra $\gamma\colon
T(Z)\rightarrow Z$ the object $Z$ carries a commutative monoid
structure. We do so by adapting the structure $(+,0)$ on $T(Z)$ from
Lemma~\ref{AdditiveMonadMonoidLem} to $(+_{Z}, 0_Z)$ on $Z$ via
$$\begin{array}{rcl}
+_{Z} 
& \stackrel{\textrm{def}}{=} &
\xymatrix{
   \Big(Z\times Z\ar[r]^-{\eta\times\eta} & 
   T(Z)\times T(Z)\ar[r]^-{+} & T(Z)\ar[r]^-{\gamma} & Z\Big)
} \\[-.3pc]
0_{Z} 
& \stackrel{\textrm{def}}{=} &
\xymatrix{\Big(1 \ar[r]^-{0} & 
   T(Z)\ar[r]^-{\gamma} & Z\Big)
}
\end{array}$$

\noindent This monoid structure is preserved by homomorphisms of
algebras.  Now, we can form coprojections $k_{1} = \tuple{\idmap}{0_{Y}
  \after\;!} \colon X\rightarrow X\times Y$, and a cotuple of algebra
homomorphisms $\smash{(TX\stackrel{\alpha}{\rightarrow}X)
  \stackrel{f}{\longrightarrow} (TZ\stackrel{\gamma}{\rightarrow}Z)}$
and $\smash{(TY\stackrel{\beta}{\rightarrow}X)
  \stackrel{g}{\longrightarrow} (TZ\stackrel{\gamma}{\rightarrow}Z)}$
given by
$$\xymatrix@C-.2pc{
\cotuple{f}{g}_{\Alg} \stackrel{\textrm{def}}{=} 
   \Big(X\times Y\ar[r]^-{f\times g} & Z\times Z\ar[r]^-{+_Z} & Z\Big).
}$$

\noindent Again, remaining details are left to the reader. 

Finally, to show that \textit{(iii)} implies \textit{(i)}, consider
the algebra morphisms:
$$\xymatrix@C-1pc{
\Big(T^2(X_i) \ar[r]^-{\mu} & T(X_i)\Big)\ar[rr]^-{T(\kappa_i)} & &
   \Big(T^2(X_1 + X_2)\ar[r]^-{\mu} &T(X_1 +X_2)\Big).
}$$ 

\noindent The free functor $\cat{C}\rightarrow \Alg(T)$ preserves
coproducts, so these $T(\kappa_{i})$ form a coproduct diagram in
$\Alg(T)$. As $\times$ is a coproduct in $\Alg(T)$, by assumption, the
cotuple $\cotuple{T(\kappa_1)}{T(\kappa_2)} \colon T(X_1) \times
T(X_2) \to T(X_1+X_2)$ in $\Alg(T)$ is an isomorphism. The
coprojections $\ell_{i}\colon T(X_{i}) \rightarrow T(X_{1})\times
T(X_{2})$ satisfy $\ell_{1} =
\tuple{\pi_{1}\after\ell_{1}}{\pi_{2}\after\ell_{2}} =
\tuple{\id}{0}$, and similarly, $\ell_{2} = \tuple{0}{\id}$.  Now we
compute:
$$\begin{array}{rcl}
\bc \after \cotuple{T(\kappa_1)}{T(\kappa_2)} \after \ell_{1}
& = &
\tuple{\mu\after T(p_{1})}{\mu\after T(p_{2})} \after T(\kappa_{1}) \\
& = &
\tuple{\mu \after T(p_{1} \after \kappa_{1})}
   {\mu \after T(p_{2} \after \kappa_{1})} \\
& = &
\tuple{\mu \after T(\eta)}{\mu \after T(0)} \\
& = &
\tuple{\id}{0} \\
& = &
\ell_{1}.
\end{array}$$

\noindent Similarly, $\bc \after \cotuple{T(\kappa_1)}{T(\kappa_2)}
\after \ell_{2} = \ell_{2}$, so that $\bc \after
\cotuple{T(\kappa_1)}{T(\kappa_2)} = \id$, making $\bc$ an
isomorphism. \QED
\end{proof}

\auxproof{
\textbf{Proof (i) implies (ii) and (i) implies (iii):}

Unicity of the product map:\\
Suppose $h \colon Z \to X + Y$ s.t. $p_i \klafter h = f_i$. To prove: $h = \tuple{f}{g}_{\Kl}$, that is,
$$
	\bc^{-1} \after \tuple{f_1}{f_2} = h.
$$
This is equivalent to
$$
	\tuple{p_1 \klafter h}{p_2 \klafter h} = \bc \after h.
$$ 
This is shown as follows:
$$
\begin{array}{rcl}
	\bc \after h &=& \tuple{\mu \after T(p_1)}{\mu \after T(p_2)} \after h \\
	&=&
	\tuple{\mu \after T(p_1) \after h}{\mu \after T(p_2) \after h} \\
	&=&
	\tuple{p_1 \klafter h}{p_2 \klafter h}
\end{array}
$$
Properties of the composition of projections with coprojections
$k_{i} = \eta \after \kappa_{i}$
$$
\begin{array}{rcl}
	p_1 \klafter k_1 &=& \mu \after T(p_1) \after \eta \after 	\kappa_1 \\
&=&
	\mu \after \eta \after p_1 \after \kappa_1\\
&=&
	\cotuple{\eta}{0} \after {\kappa_1}\\
&=&
	\eta \,\,(=id_{\Kl})	
\end{array}
$$

$$
\begin{array}{rcl}
	p_1 \klafter k_2 &=& \mu \after T(p_1) \after \eta \after \kappa_2 \\
&=&	
	\cotuple{\eta}{0} \after \kappa_2\\
&=&
	0
\end{array}
$$

In the category of algebras we first prove that $(+_{\gamma}, 0_{\gamma})$ is a
commutative monoid. We first check that they are homomorphisms of
algebras. For $0_{\gamma}$ this hold because $T(0)$ is final and
the map $0\colon T(0)\rightarrow T(Z)$ is $T(!_{Z})$, so that:
$$\xymatrix{
T^{2}(0)\ar[d]_{\mu_{0} = \;!}\ar[r]^-{T(0)} & 
   T^{2}(Z)\ar[d]_{\mu_Z}\ar[r]^-{T(\gamma)} & T(Z)\ar[d]^{\gamma} \\
T(0)\ar[r]^-{0} & T(Z)\ar[r]^-{\gamma} & Z
}$$

\noindent Similarly, the addition $+_\gamma$ is a homomorphism of
algebras $\gamma\times\gamma\rightarrow\gamma$, as shown in:
$$\begin{array}{rcl}
\lefteqn{+_{\gamma} \after (\gamma\times\gamma)} \\
& = &
\gamma \after + \after \eta\times\eta \after
    \tuple{\gamma\after T(\pi_{1})}{\gamma\after T(\pi_{2})} \\
& = &
\gamma \after + \after \eta\times\eta \after \gamma\times\gamma \after
    \tuple{T(\pi_{1})}{T(\pi_{2})} \\
& = &
\gamma \after + \after T(\gamma)\times T(\gamma) \after \eta\times\eta \after 
    \tuple{T(\pi_{1})}{T(\pi_{2})} \\
& = &
\gamma \after T(\gamma) \after + \after \eta\times\eta \after
    \tuple{T(\pi_{1})}{T(\pi_{2})} \\
& = &
\gamma \after \mu \after + \after \eta\times\eta \after
    \tuple{T(\pi_{1})}{T(\pi_{2})} \\
& = &
\gamma \after + \after \mu\times\mu \after \eta\times\eta \after
    \tuple{T(\pi_{1})}{T(\pi_{2})} \\
& = &
\gamma \after + \after \mu\times\mu \after T(\eta)\times T(\eta) \after
    \tuple{T(\pi_{1})}{T(\pi_{2})} \\
& = &
\gamma \after \mu \after + \after 
    \tuple{T(\eta \after \pi_{1})}{T(\eta \after \pi_{2})} \\
& = &
\gamma \after \mu \after + \after 
    \tuple{T(\pi_{1} \after \eta\times\eta)}{T(\pi_{2} \after \eta\times\eta)} \\
& = &
\gamma \after \mu \after  T(\nabla) \after \bc^{-1} \after 
    \tuple{T(\pi_{1})}{T(\pi_{2})} \after T(\eta\times\eta) \\
& = &
\gamma \after \mu \after  T^{2}(\nabla) \after 
   T(\cotuple{T(\kappa_{1})}{T(\kappa_{2})} \after \bc^{-1} \after 
    \tuple{T(\pi_{1})}{T(\pi_{2})} \after T(\eta\times\eta) \\
& = &
\gamma \after T(\nabla) \after \mu \after 
      T(\cotuple{T(\kappa_{1})}{T(\kappa_{2})} \after \bc^{-1} \after 
    \tuple{T(\pi_{1})}{T(\pi_{2})} \after T(\eta\times\eta) \\
& = &
\gamma \after T(\nabla) \after \bc^{-1} \after \mu\times\mu \after
    \tuple{T(\pi_{1})}{T(\pi_{2})} \after T(\eta\times\eta) \\
& & \qquad\mbox{by Lemma~\ref{bcproplem}} \\
& = &
\gamma \after T(\nabla) \after \mu \after T(\bc^{-1}) \after 
   T(\eta\times\eta) \\
& & \qquad\mbox{again by Lemma~\ref{bcproplem}} \\
& = &
\gamma \after \mu \after T^{2}(\nabla) \after T(\bc^{-1}) \after 
   T(\eta\times\eta) \\
& = &
\gamma \after T(\gamma) \after T(T(\nabla) \after \bc^{-1} \after 
   \eta\times\eta) \\
& = &
\gamma \after T(\gamma \after + \after \eta\times\eta) \\
& = &
\gamma \after T(+_{\gamma}).
\end{array}$$

We continue with the unit law $+_{\gamma} \after
(\idmap\times 0_{\gamma}) = \pi_{1} \colon Z\times T(0) \rightarrow Z$,
where we omit the isomorphism $T(0)\congrightarrow 1$ for convenience
(or simply take $T(0)$ as final object).
$$\begin{array}{rcl}
+_{\gamma} \after (\idmap\times 0_{\gamma})
& = &
\gamma \after + \after \eta\times\eta \after 
   \idmap\times \gamma \after \idmap\times T(!_{Z}) \\
& = &
\gamma \after + \after \idmap\times T(\gamma) \after
    \eta\times\eta \after \idmap\times T(!_{Z}) \\
& = &
\gamma \after + \after T(\gamma)\times T(\gamma) \after
    T(\eta)\times\eta \after \eta\times T(!_{Z}) \\
& & \qquad \mbox{since } \gamma \after \eta = \idmap \\
& = &
\gamma \after + \after T(\gamma)\times T(\gamma) \after
    \eta\times\eta \after \eta\times T(!_{Z}) \\
& & \qquad \mbox{since } T(\eta) \after \eta = \eta \after \eta \\
& = &
\gamma \after T(\gamma) \after + \after 
    \eta\times\eta \after \eta\times T(!_{Z}) \\
& = &
\gamma \after \mu \after + \after 
    \eta\times\eta \after \eta\times T(!_{Z}) \\
& = &
\gamma \after + \after \mu\times\mu \after 
    \eta\times\eta \after \eta\times T(!_{Z}) \\
& = &
\gamma \after + \after \idmap{}\times T(!_{Z}) \after \eta\times\idmap \\
& = &
\gamma \after \pi_{1} \after \eta \times \idmap \\
& & \qquad \mbox{by the unit law for $+$} \\
& = &
\gamma \after \eta \after \pi_{1} \\
& = &
\pi_{1}
\end{array}$$

\noindent Commutativity is easy:
$$\begin{array}{rcl}
+_{\gamma} \after \tuple{\pi_{2}}{\pi_{1}} 
& = &
\gamma \after + \after \eta\times\eta \after \tuple{\pi_{2}}{\pi_{1}} \\
& = &
\gamma \after + \after \tuple{\eta\after\pi_{2}}{\eta\after\pi_{1}} \\
& = &
\gamma \after + \after 
   \tuple{\pi_{2} \after \eta\times\eta}{\pi_{1}\after\eta\times\eta} \\
& = &
\gamma \after + \after \tuple{\pi_{2}}{\pi_{1}} \after \eta\times\eta \\
& = &
\gamma \after + \after \eta\times\eta 
   \qquad\mbox{by commutativity of $+$ on $T(Z)$} \\
& = &
+_{\gamma}.
\end{array}$$

\noindent Associativity also holds via associativity of $+$ on $T(Z)$,
using $\alpha =
\tuple{\tuple{\pi_1}{\pi_{1}\after\pi_{2}}}{\pi_{2}\after\pi_{2}}
\colon Z\times (Z\times Z)\rightarrow (Z\times Z)\times Z$. It
requires a bit more work.
$$\begin{array}{rcl}
\lefteqn{+_{\gamma} \after (+_{\gamma}\times\idmap) \after \alpha} \\
& = &
\gamma \after + \after \eta\times\eta \after 
   ((\gamma \after + \after \eta\times\eta)\times\idmap) \after \alpha \\
& = &
\gamma \after + \after ((\eta \after \gamma \after +)\times\idmap)
   \after ((\eta\times\eta)\times\eta) \after \alpha \\
& = &
\gamma \after + \after (T(\gamma) \after \eta)\times (T(\gamma) \after T(\eta))
   \after (+\times\idmap) \after \alpha \after (\eta\times(\eta\times\eta)) \\
& = &
\gamma \after T(\gamma) \after + \after (\eta\times T(\eta))
   \after (+\times\idmap) \after \alpha \after (\eta\times(\eta\times\eta)) \\
& = &
\gamma \after \mu \after + \after (\eta\times T(\eta))
   \after (+\times\idmap) \after \alpha \after (\eta\times(\eta\times\eta)) \\
& = &
\gamma \after + \after (\mu\times\mu) \after (\eta\times T(\eta))
   \after (+\times\idmap) \after \alpha \after (\eta\times(\eta\times\eta)) \\
& = &
\gamma \after + \after (+\times\idmap) \after \alpha \after 
   (\eta\times(\eta\times\eta)) \\
& = &
\gamma \after + \after (\idmap\times +) \after (\eta\times(\eta\times\eta)) 
   \qquad\mbox{by associativity of +} \\
& = &
\gamma \after + \after (\mu\times\mu) \after (T(\eta)\times \eta)
    \after (\idmap\times +) \after (\eta\times(\eta\times\eta)) \\
& = &
\gamma \after \mu \after + \after (T(\eta)\times \eta)
    \after (\idmap\times +) \after (\eta\times(\eta\times\eta)) \\
& = &
\gamma \after T(\gamma) \after + \after (T(\eta)\times \eta)
    \after (\idmap\times +) \after (\eta\times(\eta\times\eta)) \\
& = &
\gamma \after + \after (T(\gamma) \after T(\eta))\times (T(\gamma) \after \eta)
    \after (\idmap\times +) \after (\eta\times(\eta\times\eta)) \\
& = &
\gamma \after + \after (\idmap\times(\eta \after \gamma)) \after 
   (\idmap\times +) \after (\eta\times(\eta\times\eta)) \\
& = &
\gamma \after + \after (\eta\times\eta) \after 
   (\idmap\times(\gamma\after + \after \eta\times\eta)) \\
& = &
+_{\gamma} \after (\idmap\times +_{\gamma}).
\end{array}$$

Preservation by algebra homomorphism
$\smash{(TZ\stackrel{\gamma}{\rightarrow}Z)
  \stackrel{h}{\longrightarrow} (TW\stackrel{\delta}{\rightarrow}W)}$
follows from:
$$\begin{array}{rcl}
+_{\delta} \after (h\times h)
& = &
\delta \after + \after \eta\times\eta \after h\times h \\
& = &
\delta \after + \after T(h)\times T(h) \after \eta\times\eta  \\
& = &
\delta \after T(h) \after + \after \eta\times\eta  \\
& = &
h \after \gamma \after + \after \eta\times\eta  \\
& = &
h \after +_{\gamma} \\
h \after 0_{\gamma} 
& = &
h \after \gamma \after 0 \\
& = &
\delta \after T(h) \after 0 \\
& = &
\delta \after 0 \\
& = &
0_{\delta}.
\end{array}$$

Now we can prove the cotuple laws:
$$\begin{array}{rcl}
\cotuple{f}{g}_{\Alg} \after k_{1}
& = &
+_{\gamma} \after f\times g \after \tuple{\idmap}{0_{\beta} \after \;!} \\
& = & 
+_{\gamma} \after f\times g \after \idmap\times 0_{\beta} \after 
   \tuple{\idmap}{!} \\
& = & 
+_{\gamma} \after \idmap\times 0_{\gamma} \after f\times \idmap \after
   \tuple{\idmap}{!} \\
& = & 
\pi_{1} \after f\times \idmap \after \tuple{\idmap}{!} \\
& = & 
f \after \pi_{1} \after \tuple{\idmap}{!} \\
& = &
f.
\end{array}$$

\noindent What we forgot to check is that $\cotuple{f}{g}_{\Alg}$
is a homomorphism of algebra $\alpha\times\beta \rightarrow \gamma$.
This holds since:
$$\begin{array}{rcl}
\cotuple{f}{g}_{\Alg} \after (\alpha\times\beta)
& = &
+_{\gamma} \after f\times g \after 
   \tuple{\alpha \after T(\pi_{1})}{\beta \after T(\pi_{2})} \\
& = &
+_{\gamma} \after f\times g \after \alpha\times\beta \after
   \tuple{T(\pi_{1})}{T(\pi_{2})} \\
& = &
+_{\gamma} \after \gamma\times \gamma \after T(f)\times T(g) \after
   \tuple{T(\pi_{1})}{T(\pi_{2})} \\
& = &
+_{\gamma} \after \gamma\times \gamma \after 
   \tuple{T(f\after \pi_{1})}{T(g\after \pi_{2})} \\
& = &
+_{\gamma} \after \gamma\times \gamma \after 
   \tuple{T(\pi_{1}) \after T(f\times g)}{T(\pi_{2}) \after T(f\times g)} \\
& = &
+_{\gamma} \after \gamma\times \gamma \after 
   \tuple{T(\pi_{1})}{T(\pi_{2})} \after T(f\times g) \\
& = &
+_{\gamma} \after 
   \tuple{\gamma \after T(\pi_{1})}{\gamma \after T(\pi_{2})} \after 
   T(f\times g) \\
& = &
\gamma \after T(+_{\gamma}) \after T(f\times g) \\
& & \qquad\mbox{since $+_\gamma$ is a homomorphism of algebras} \\
& = &
\gamma \after T(\cotuple{f}{g}_{\Alg})
\end{array}$$

\noindent If $h\colon X\times Y\rightarrow Z$ is also an algebra
homomorphism satisfying $h \after k_{1} = f$ and $h\after k_{2} = g$,
then:
$$\begin{array}{rcl}
\lefteqn{\cotuple{f}{g}_{\Alg}} \\
& = &
+_{\gamma} \after (f\times g) \\
& = &
\gamma \after + \after \eta\times\eta \after 
   (h\after k_{1})\times (h\after k_{2}) \\
& = &
\gamma \after + \after T(h\after k_{1})\times T(h\after k_{2}) \after 
   \eta\times\eta \\
& = &
\gamma \after T(h) \after + \after T(k_{1})\times T(k_{2}) \after 
   \eta\times\eta \\
& = &
h \after \tuple{\alpha\after T(\pi_{1})}{\beta \after T(\pi_{2})} \after 
   + \after \\
& & \qquad
   T(\idmap\times(\beta\after T(!)) \after \rho^{-1})\times 
      T((\alpha\after T(!))\times \idmap\after \lambda^{-1}) 
   \after \eta\times\eta \\
& = &
h \after \alpha\times\beta \after \tuple{T(\pi_{1})}{T(\pi_{2})} \after 
   + \after \\
& & \qquad
   T((\alpha\after\eta)\times(\beta\after T(!)) \after \rho^{-1})\times 
      T((\alpha\after T(!))\times (\beta\after\eta)\after \lambda^{-1}) 
   \after \eta\times\eta \\
& = &
h \after \alpha\times\beta \after \tuple{T(\pi_{1})}{T(\pi_{2})} \after 
   + \after T(\alpha\times\beta)\times T(\alpha\times\beta) \after \\
& & \qquad
   T(\eta\times T(!)\after \rho^{-1})\times 
       T(T(!)\times \eta \after \lambda^{-1}) \after \eta\times\eta \\
& = &
h \after \alpha\times\beta \after \tuple{T(\pi_{1})}{T(\pi_{2})} \after 
   T(\alpha\times\beta) \after + \after \\
& & \qquad
   T(\eta\times T(!)\after \rho^{-1})\times 
       T(T(!)\times \eta \after \lambda^{-1}) \after \eta\times\eta \\
& = &
h \after \alpha\times\beta \after T(\alpha)\times T(\beta) \after 
   \tuple{T(\pi_{1})}{T(\pi_{2})} \after T(\nabla) \after \bc^{-1} \after \\
& & \qquad
   T(\eta\times T(!)\after \rho^{-1})\times 
       T(T(!)\times \eta \after \lambda^{-1}) \after \eta\times\eta \\
& = &
h \after \alpha\times\beta \after \mu\times \mu \after 
   \tuple{T(\pi_{1})}{T(\pi_{2})} \after T(\nabla) \after \\
& & \qquad
   T\Big(\big(\eta\times T(!)\after \rho^{-1}\big) + 
       \big(T(!)\times \eta \after \lambda^{-1}\big)\Big)
   \after \bc^{-1} \after \eta\times\eta \\
& \stackrel{(*)}{=} &
h \after \alpha\times\beta \after \mu\times \mu \after 
   \tuple{T(p_{1})}{T(p_{2})} \after \bc^{-1} \after \eta\times\eta \\
& = &
h \after \alpha\times\beta \after  
   \tuple{\mu \after T(p_{1})}{\mu \after T(p_{2})} \after 
   \bc^{-1} \after \eta\times\eta \\
& = &
h \after \alpha\times\beta \after \bc \after 
   \bc^{-1} \after \eta\times\eta \\
& = &
h \after \alpha\times\beta \after \eta\times\eta \\
& = &
h,
\end{array}$$

\noindent where the marked equation $\stackrel{(*)}{=}$ holds because:
$$\begin{array}{rcl}
\lefteqn{\pi_{1} \after \nabla \after 
  \Big(\big(\eta\times T(!)\after \rho^{-1}\big) + 
       \big(T(!)\times \eta \after \lambda^{-1}\big)\Big)} \\
& = &
\pi_{1} \after 
  \big[\eta\times T(!)\after \rho^{-1}, \;
       T(!)\times \eta \after \lambda^{-1}\big] \\
& = &
\big[\pi_{1} \after \eta\times T(!)\after \rho^{-1}, \;
       \pi_{1} \after T(!)\times \eta \after \lambda^{-1}\big] \\
& = &
\big[\eta \after \pi_{1} \after \rho^{-1}, \;
       T(!) \after \pi_{1} \after \lambda^{-1}\big] \\
& = &
\cotuple{\eta}{T(!) \after \;!} \qquad \mbox{since } \\
& & \qquad
\xymatrix{\big(T(X)\ar[r]^-{\rho^{-1}}_-{\cong} & 
   T(X)\times T(0)\ar[r]^-{\pi_{1}}_-{\cong} & T(X)\big) = \idmap_{X} } \\
& & \qquad
\xymatrix{\big(T(Y)\ar[r]^-{\lambda^{-1}}_-{\cong} & 
   T(0)\times T(Y)\ar[r]^-{\pi_{1}}_-{\cong} & T(0)\big) = \; !_{T(Y)} } \\
& = &
\cotuple{\eta}{0} \\
& = &
p_{1}.
\end{array}$$

\textbf{proof of (ii) impies (i) and (iii) implies (i)}

\renewcommand{\theenumi}{(\roman{enumi})}
\begin{enumerate}
\item Kleisli category.\\
First we show that $q_i = p_i \colon X_1+X_2 \to T(X_i)$.
$$
\begin{array}{rcll}
q_1 \after \kappa_1^{\Sets} &=& \mu \after \eta \after q_1 \after \kappa_1^{\Sets}\\
&=&
\mu \after Tq_1 \after \eta \after \kappa_1^{\Sets}\\
&=&
q_1 \klafter \kappa_1^{\Kl}\\
&=&
\id^{\Kl} = \eta
\end{array}
$$
And similarly $q_1 \after \kappa_2^{\Sets} = q_1 \klafter \kappa_2^{\Kl} = 0$, hence $q_1 = \cotuple{\eta}{0} = p_1$.

Let $h$ be the map defined in the proof above. Then $\bc \after h = \id$, as
$$
\begin{array}{rcll}
\bc \after h &=& \tuple{\mu \after T(p_1)}{\mu \after T(p_2)} \after h\\
&=&
\tuple{\mu \after T(p_1)\after h}{\mu \after T(p_2)\after h}\\
&=&
\tuple{p_1 \klafter h}{p_2 \klafter h} \\
&=& 
\tuple{\pi_1}{\pi_2} = \id
\end{array}
$$
Also $h \after \bc = \idmap \colon T(X_1+X_2) \to T(X_1+X_2)$. To prove this, note that $h \after \bc$ is a map $T(X_1+X_2) \to X_1 + X_2$ in $\Kl(T)$. Composing with the projections in $\Kl(T)$ yields:
$$
p_1 \klafter (h \after \bc) = \mu \after Tp_1 \after h \after \bc = (p_1 \klafter h) \after \bc = \pi_1 \after \bc
$$ 
and similarly $p_2 \klafter (h \after \bc) = \pi_2 \after \bc$. Hence
$$
h \after \bc = \tuple{\pi_1 \after \bc}{\pi_2 \after \bc}_{\Kl} = \tuple{\mu \after Tp_1}{\mu \after Tp_2} = \id,
$$
as $p_i \klafter \idmap = \mu \after Tp_i$.

So $h$ is the inverse of $\bc$.

\item $\Alg(T)$\\

As, by assumption $\times$ is a biproduct in $\Alg(T)$, there exists
coprojections $l_i \colon (T(A) \to A) \to (T(A\times B) \to A\times
B)$. By definition of the biproduct, $\pi_1 \after l_1 = \id$ and
$\pi_2 \after l_1 = 0$. The map $0 \colon A \to B$ in $\Alg(T)$ is the
unique map $A \xrightarrow{!} 1 \xrightarrow{!} B$. Considering the
diagram
$$
\xymatrix{
T(1) \ar[r]^{\cong} \ar[d]_{!} & T^2(0) \ar[d]^{\mu} \ar[r]^{T^2(!)} & T^2(B) \ar[r]^{T(\beta)} \ar[d]^{\mu} & T(B) \ar[d]^{\beta}\\
1 \ar[r]^{\cong} & T(0) \ar[r]^-{T(!)} & T(B) \ar[r]^-{\beta} &B
}
$$
we see that $\beta \after T(!) \colon 1 \to B$ is an algebra morphism, and hence it is the unique map $1 \to B$. So 
$$
l_1 = \tuple{\id}{\beta \after T(!) \after !}
$$
Write $g = \cotuple{T(\kappa_1)}{T(\kappa_2)} \colon T(X_1) \times T(X_2) \to T(X_1+X_2)$. We prove that $g$ is the inverse of $\bc$. 

First we show that $\bc$ is a morphism in $\Alg(T)$. We have to prove
that the following diagram commutes:
$$
\xymatrix{
T^2(X_1 + X_2) \ar[r]^-{T\bc} \ar[d]_{\mu} & T(T(X_1) \times T(X_2)) \ar[d]^-{\tuple{\mu \after T\pi_1}{\mu \after T\pi_2}}\\
T(X_1+X_2) \ar[r]^-{\bc} & T(X_1) \times T(X_2)
}
$$
$$
\begin{array}{rcll}
\tuple{\mu \after T\pi_1}{\mu \after T\pi_2} \after T\bc &=& \tuple{\mu \after T\mu \after T^2p_1}{\mu \after T\mu \after T^2p_2}\\
&=&
\tuple{\mu \after \mu \after T^2p_1}{\mu \after \mu \after T^2p_2}\\
&=&
\tuple{\mu \after Tp_1 \after \mu}{\mu \after Tp_2 \after \mu}\\
&=&
\tuple{\mu \after Tp_1}{\mu \after Tp_2} \after \mu = \bc \after \mu
\end{array}
$$

Now we consider $\bc \after g \colon T(X_1) \times T(X_2) \to T(X_1) \times T(X_2)$. The coprojection $l_1 \colon T(X_1) \to T(X_1) \times T(X_2)$ is given by 
$$
l_1 = \tuple{\id}{\mu \after T(!) \after !} = \tuple{id}{T(!)\after !}
$$
Note that:
$$
\begin{array}{rcll}
\bc \after g \after l_1 &=& \bc \after T\kappa_1\\
&=&
\tuple{\mu \after T(\cotuple{\eta}{0})}{\mu \after T(\cotuple{0}{\eta})} \after T\kappa_1 &\text{where}\,0=X_1 \xrightarrow{!} T(0) \xrightarrow{T(!)} T(X_2)\\
&=&
\tuple{\mu \after T\eta}{\mu \after T(0)}\\
&=&
\tuple{\id}{\mu \after T^2(!) \after T(!)}\\
&=&
\tuple{\id}{T(!) \after \mu \after T(!)}\\
&=&
\tuple{\id}{T(!) \after !} = l_1 &\text{$\mu \after T(!)$ maps into the final object}\\
\end{array}
$$
Similarly $\bc \after g \after l_2 = l_2$. Hence, by the property of the coproduct, $\bc \after g = \id$.

To prove that $g \after \bc = \idmap \colon T(X_1+X_2) \to T(X_1) + T(X_2)$, first note that the free construction
$$
F \colon \cat{C} \leftrightarrows \Alg(T) \colon U,
$$
where $F(X) = T^2(X) \xrightarrow{\mu} X$, preserves coproducts. 
Hence $T^2(X_1 + X_2) \xrightarrow{\mu} T(X_1+X_2)$ is the coproduct of $T^2(X_1) \xrightarrow{\mu} T(X_1)$ and $T^2(X_2) \xrightarrow{\mu} T(X_2)$ with coprojection $T\kappa_1$ and $T\kappa_2$. 
$$
\begin{array}{rcll}
g \after \bc \after T\kappa_1 &=& g \after \tuple{\mu \after \eta}{\mu \after T(0)}\\
&=&
g \after \tuple{\id}{\mu \after T(0)}\\
&=&
g \after l_1 = T\kappa_1
\end{array}
$$
Similarly $g \after \bc \after T\kappa_2 = T\kappa_2$. Hence, by the property of the coproduct, $g\after\bc =\id$. So $g$ is the inverse of $\bc$.
\end{enumerate}
}

It is well-known (see for instance~\cite{KellyL80,AbramskyC04}) that a
category with finite biproducts $(\oplus, 0)$ is enriched over
commutative monoids: each homset carries a commutative monoid
structure $(+,0)$, and this structure is preserved by pre- and
post-composition. The addition operation $+$ on homsets is obtained
as
\begin{equation}
\label{HomsetPlus}
\xymatrix{
f+g \stackrel{\textrm{def}}{=} \Big(X\ar[r]^-{\tuple{\idmap}{\idmap}} &
   X\oplus X\ar[r]^-{f\oplus g} &
   Y\oplus Y\ar[r]^-{\cotuple{\idmap}{\idmap}} & Y\Big).
}
\end{equation}

\noindent The zero map is neutral element for this addition.  One can
also describe a monoid structure on each object $X$ as
\begin{equation}
\label{ObjMon}
\xymatrix{
X\oplus X\ar[r]^-{\cotuple{\idmap}{\idmap}} & X &
   0.\ar[l]_-{0}
}
\end{equation}

\noindent We have just seen that the Kleisli category of an additive
monad has biproducts, using the addition operation from
Lemma~\ref{AdditiveMonadMonoidLem}. When we apply the sum
description~\eqref{ObjMon} to such a Kleisli category its biproducts,
we obtain precisely the original addition from
Lemma~\ref{AdditiveMonadMonoidLem}, since the codiagonal $\nabla =
\cotuple{\idmap}{\idmap}$ in the Kleisli category is given $T(\nabla)
\after \bc^{-1}$.

\subsection{Additive commutative monads}

In the remainder of this section we focus on the category
$\ACMnd(\cat{C})$ of monads that are both additive and commutative on
a distributive category \cat{C}. As usual, we simply write $\ACMnd$ for
$\ACMnd(\Sets)$. For $T \in \cat{ACMnd}(\cat{C})$, the Kleisli
category $\Kl(T)$ is both symmetric monoidal---with $(\times,1)$ as
monoidal structure, see Lemma~\ref{KleisliStructLem}---and has
biproducts $(+,0)$. Moreover, it is not hard to see that this monoidal
structure distributes over the biproducts via the canonical map
$(Z\times X)+(Z\times Y)\rightarrow Z\times (X+Y)$ that can be lifted
from \cat{C} to $\Kl(T)$.

We shall write $\SMBLaw \hookrightarrow \SMLaw$ for the category of
symmetric monoidal Lawvere theories in which $(+,0)$ form not only
coproducts but biproducts. Notice that a projection $\pi_{1}\colon
n+m\rightarrow n$ is necessarily of the form $\pi_{1} =
\cotuple{\id}{0}$, where $0 \colon m\rightarrow n$ is the zero map
$m\rightarrow 0 \rightarrow n$. The tensor $\otimes$ distributes over
$(+,0)$ in \SMBLaw, as it already does so in \SMLaw. Morphisms in
\SMBLaw are functors that strictly preserve all the structure.

The following result extends Corollary~\ref{Mnd2FCCatCor}.

\begin{lemma}
The (finitary) Kleisli construction on a monad yields a functor
$\Kl_{\NNO} \colon \ACMnd \rightarrow \SMBLaw$.
\end{lemma}

\begin{proof}
It follows from Theorem~\ref{AMnd2BCat} that $(+,0)$ form biproducts
in $\Kl_{\NNO}(T)$, for $T$ an additive commutative monad (on
\Sets). This structure is preserved by functors $\Kl_{\NNO}(\sigma)$,
for $\sigma\colon T\rightarrow S$ in \ACMnd. \QED

\auxproof{
$$
\begin{array}[b]{rcl}
\Kl_{\NNO}(\sigma)(k_{1})
& = &
\sigma \after \eta \after \kappa_{i} \\
& = &
\eta \after \kappa_{i} \\
& = &
k_{i} \\
\Kl_{\NNO}(\sigma)(p_{1})
& = &
\sigma \after \cotuple{\eta}{0} \\
& = &
\cotuple{\sigma\after\eta}{\sigma\after 0} \\
& = &
\cotuple{\eta}{0} \\
& = &
p_{1} \\
\Kl_{\NNO}(\sigma)(\cotuple{f}{g}_{\Kl})
& = &
\sigma \after \cotuple{f}{g} \\
& = &
\cotuple{\sigma\after f}{\sigma \after g} \\
& = &
\cotuple{\Kl_{\NNO}(\sigma)(f)}{\Kl_{\NNO}(\sigma)(g)}_{\Kl} \\
\Kl_{\NNO}(\sigma)\tuple{f}{g}_{\Kl}) 
& = &
\sigma \after \bc^{-1} \after \tuple{f}{g} \\
& = &
\bc^{-1} \after (\sigma\times\sigma) \after \tuple{f}{g} \\
& = &
\bc^{-1} \after \tuple{\sigma\after f}{\sigma \after g} \\
& = &
\tuple{\Kl_{\NNO}(\sigma)(f)}{\Kl_{\NNO}(\sigma)(g)}_{\Kl}.
\end{array}
$$
}
\end{proof}

We have already seen in Lemma \ref{LMCommLem} that the functor $\LM
\colon \Law\rightarrow\Mnd$ defined in Lemma \ref{GLaw2MndLem}
restricts to a functor between symmetric monoidal Lawvere theories and
commutative monads. We now show that it also restricts to a functor
between symmetric monoidal Lawvere theories with biproducts and
commutative additive monads. Again, this restriction is left adjoint
to the finitary Kleisli construction.

\begin{lemma}
\label{LMCSRngLem}
The functor $\LM\colon\SMLaw\rightarrow\CMnd$ from Lemma
\ref{LMCommLem} restricts to $\SMBLaw \rightarrow \ACMnd$. Further,
this restriction is left adjoint to the finitary Kleisli construction
$\Kl_{\NNO}\colon \ACMnd\rightarrow \SMBLaw$.
\end{lemma}

\begin{proof}
First note that $T_{\cat{L}}(0)$ is final: 
$$
\begin{array}{rclclcl}
T_{\cat{L}}(0) &=& \textstyle{\coprod_i}\,\cat{L}(1,i) \times 0^i &\cong& \cat{L}(1,0) \times 0^0 &\cong& 1,
\end{array}
$$ 

\noindent where the last isomorphism follows from the fact that
$(+,0)$ is a biproduct in $\cat{L}$ and hence $0$ is terminal.  The
resulting zero map $0_{X,Y} \colon X \to T(Y)$ is given by
$$
\begin{array}{rcl}
x &\mapsto& [\kappa_0(!\colon 1 \to 0, ! \colon 0 \to Y)].
\end{array}
$$

To prove that the bicartesian map $\bc \colon T_{\cat{L}}(X+Y) \to
T_{\cat{L}}(X) \times T_{\cat{L}}(Y)$ is an isomorphism, we introduce
some notation. For $[\kappa_i(g,v)] \in T_{\cat{L}}(X+Y)$, where $g
\colon 1 \to i$ and $v \colon i \to X+Y$, we form the pullbacks (in
\Sets)
$$\xymatrix@R1.5pc{
i_{X}\ar[d]_{v_X}\ar[r]\pullback & i\ar[d]^-{v} & 
   i_{Y}\ar[d]^{v_Y}\ar[l]\pullback[dl] \\
X\ar[r]^-{\kappa_{1}} & X+Y & Y\ar[l]_-{\kappa_{2}}
}$$

\noindent By universality of coproducts we can write $i = i_{X} + i_{Y}$
and $v = v_{X}+v_{Y} \colon i_{X}+i_{Y} \rightarrow X+Y$. Then we
can also write $g = \tuple{g_X}{g_Y} \colon 1\rightarrow i_{X}+i_{Y}$.
Hence, for $[k_i(g,v)] \in T_{\cat{L}}(X+Y)$,
\begin{equation}
\label{LawMndbcEqn}
\begin{array}{rcl}
\bc([\kappa_i(g,v)]) 
& = & 
\big([\kappa_{i_X}(g_{X}, v_{X})], \, [\kappa_{i_Y}(g_{Y}, v_{Y})]\big).
\end{array}
\end{equation}



\noindent It then easily follows that the map $T_{\cat{L}}(X) \times
T_{\cat{L}}(Y) \to T_{\cat{L}}(X+Y)$ defined by
$$
\begin{array}{rcll}
([\kappa_i(g,v)], [\kappa_j(h,w)])& \mapsto & [\kappa_{i+j}(\langle g, h\rangle,v+w)]
\end{array}
$$
is the inverse of $\bc$.

Checking that the unit of the adjunction
$\LM\colon\SMLaw\leftrightarrows\CMnd\colon \Kl_{\NNO}$ preserves the
product structure is left to the reader. This proves that also the
restricted functors form an adjunction.\QED
\end{proof}

\auxproof{
To see that the map $\bc = \tuple{\mu \after T_{\cat{L}}p_1}{\mu \after T_{\cat{L}}p_2}$ is indeed given as above, note that for $[\kappa_i(g,v)] \in T_{\cat{L}}(X+Y)$
$$
(\mu \after T([\eta,0]))([\kappa_i(g,v)]) = \mu([\kappa_i(g, [\eta,0] \after v)]) 
$$
and
$$
([\eta,0] \after v)(a) = \left\{
\begin{array}{ll}
	\kappa_1([\id,a]) & \text{if}\, v(a) \in X \\
	\kappa_0([!,!]) & \text{otherwise}.
\end{array} \right.
$$

Write $f$ for the map $T(X) \times T(Y) \to T(X+Y)$, $([\kappa_i(g,v)], [\kappa_j(h,w)]) \mapsto [\kappa_{i+j}(\langle g, h\rangle,v+w)]$. Claim: $f$ is the inverse of $\bc$.
$$
\begin{array}{rcll}
(f \after \bc)([\kappa_i(g,v)]) &=& [\kappa_{i_X+i_Y}\big(\tuple{\langle \pi_{x_0}, \ldots, \pi_{x_{i_X-1}}\rangle \after g}{\langle \pi_{y_0}, \ldots, \pi_{y_{i_Y-1}}\rangle \after g},\\
&&
v \after [\kappa_{x_0}, \ldots \kappa_{x_{i_X-1}}] + v \after [\kappa_{y_0}, \ldots \kappa_{y_{i_Y-1}}]\big)]\\
&=&
[\kappa_i\big(\langle \pi_{x_0}, \ldots, \pi_{x_{i_X-1}}, \pi_{y_0}, \ldots, \pi_{y_{i_Y-1}}\rangle \after g,\\
&&
v \after[\kappa_{x_0}, \ldots \kappa_{x_{i_X-1}},\kappa_{y_0}, \ldots \kappa_{y_{i_Y-1}}]\big)]\\
&=&
[\kappa_i(g,v \after[\kappa_{x_0}, \ldots \kappa_{x_{i_X-1}},\kappa_{y_0}, \ldots \kappa_{y_{i_Y-1}}] \after \langle \pi_{x_0}, \ldots, \pi_{x_{i_X-1}}, \pi_{y_0}, \ldots, \pi_{y_{i_Y-1}}\rangle)]\\
&=&
[\kappa_i(g,v)]   
\end{array}
$$  

$$
\begin{array}{rcll}
(\bc\after f)([\kappa_i(g,v)], [\kappa_j(h,w)])
&=&
\bc([\kappa_{i+j}(\langle g, h\rangle,v+w)])\\
&=&
\big([\kappa_i(\langle \pi_0, \ldots \pi_{i-1} \rangle \after \langle g, h\rangle, (v+w) \after [\kappa_0, \ldots \kappa_{i-1}])],\\
&&
[\kappa_j(\langle \pi_i, \ldots \pi_{i+j} \rangle \after \langle g, h\rangle, (v+w) \after [\kappa_i, \ldots \kappa_{i+j}])]\big)\\
&=&
([\kappa_i(g,v)], [\kappa_j(h,w)]) 
\end{array}
$$
Hence, $\bc$ is an isomorphism.

To show that $\eta = \overline{\id_{T_{\cat{L}}}} \colon \cat{L} \to (\Kl_{\NNO}\after\LM)(\cat{L})$ preserves the product structure consider $n+m \xrightarrow{\pi_1^{\cat{L}}} n$ in $\cat{L}$. We will show that $\overline{\id}(\pi_1^{\cat{L}}) \colon n+m \to T_{\cat{L}}(n) = \pi_1^{\Kl_{\NNO}(T_{\cat{L}})} = [\eta, 0]$ (where this last $\eta$ is the unit of the monad $T_{\cat{L}}$). For $i \in n$,
$$
\begin{array}{rcll}
\overline{\id}(\pi_1^{\cat{L}})(i) &=& [\kappa_n(\pi_1^{\cat{L}}\after\kappa_i, \id_n)]\\
&=&
[\kappa_1(\id, \pi_1 \after \kappa_i)] &\text{eq. rel}\\
&=&
[\kappa_1(\id, i)]\\
&=&
\eta(i)
\end{array}
$$
For $i \in m$,
$$
\begin{array}{rcll}
\overline{\id}(\pi_1^{\cat{L}})(i) &=& [\kappa_n(\pi_1^{\cat{L}}\after\kappa_i, \id_n)]\\
&=&
[\kappa_n(\pi_1^{\cat{L}}\after\kappa_2 \after \kappa_{i-n}, \id_n)] &(\text{where}\, \kappa_{i-n} \colon 1 \to m)\\
&=&
[\kappa_n(0_{m,n} \after \kappa_{i-n}, \id_n)] \\
&=& 
[\kappa_n(! \after ! \after \kappa_{i-n}, \id_n)]\\
&=&
[\kappa_0(! \after \kappa_{i-n}, \id_n \after !)]\\
&=&
[\kappa_0(!,!)]\\
&=&
0_{m,n}(i)
\end{array}
$$
}

In the next two sections we will see how additive commutative monads
and symmetric monoidal Lawvere theories with biproducts relate to
commutative semirings.

\section{Semirings and monads}\label{SemiringMonadSec}

This section starts with some clarification about semirings and
modules. Then it shows how semirings give rise to certain
``multiset'' monads, which are both commutative and additive.  It is
shown that the ``evaluate at unit 1''-functor yields a map in the
reverse direction, giving rise to an adjunction, as before.

A commutative semiring in \Sets consists of a set $S$ together with
two commutative monoid structures, one additive $(+,0)$ and one
multiplicative $(\cdot, 1)$, where the latter distributes over the
former: $s\cdot 0 = 0$ and $s\cdot (t+r) = s\cdot t + s\cdot r$. For
more information on semirings, see~\cite{Golan99}. Here we only
consider commutative ones. Typical examples are the natural
numbers $\NNO$, or the non-negative rationals $\mathbb{Q}_{\geq 0}$, or
the reals $\mathbb{R}_{\geq 0}$.

One way to describe semirings categorically is by considering the
additive monoid $(S,+,0)$ as an object of the category \Cat{CMon} of
commutative monoids, carrying a multiplicative monoid structure
$\smash{I \stackrel{1}{\rightarrow} S \stackrel{\cdot}{\leftarrow}
  S\otimes S}$ in this category \Cat{CMon}. The tensor guarantees that
multiplication is a bihomomorphism, and thus distributes over
additions.

In the present context of categories with finite products we do not
need to use these tensors and can give a direct categorical
formulation of such semirings, as a pair of monoids $\smash{1
  \stackrel{0}{\rightarrow} S \stackrel{+}{\leftarrow} S\times S}$ and
$\smash{1 \stackrel{1}{\rightarrow} S \stackrel{\cdot}{\leftarrow}
  S\times S}$ making the following distributivity diagrams commute.
$$\xymatrix@R-.5pc{
S\times 1\ar[r]^-{\idmap\times 0}\ar[d]_-{!} & 
   S\times S\ar[d]^{\cdot}
&
(S\times S)\times S\ar[r]^-{\mathit{dbl}}\ar[d]_{+\times\idmap} & 
   (S\times S)\times (S\times S)\ar[r]^-{\cdot\times\cdot} & 
   S\times S\ar[d]^{+} \\
1 \ar[r]^-{0}& S
&
S\times S\ar[rr]^-{\cdot} & & S
}$$

\noindent where $\mathit{dbl} =
\tuple{\pi_{1}\times\idmap}{\pi_{2}\times\idmap}$ is the doubling map
that was also used in Lemma~\ref{bcproplem}. With the obvious notion
of homomorphism between semirings this yields a category
$\CSRng(\Cat{C})$ of (commutative) semirings in a category \Cat{C}
with finite products.

Associated with a semiring $S$ there is a notion of module over
$S$. It consists of a commutative monoid $(M,0,+)$ together with a
(multiplicative) action $\star\colon S\times M\rightarrow M$ that is
an additive bihomomorphism, that is, the action preserves the additive
structure in each argument separately. We recall that the properties
of an action are given categorically by $\star \after (\cdot \times
\idmap) = \star \after (\idmap \times \star) \after \alpha^{-1} \colon
(S\times S) \times M \to M$ and $\star \after (1\times\idmap) = \pi_2
\colon 1 \times M \to M$. The fact that $\star$ is an additive
bihomomorphism is expressed by
$$\xymatrix@R-.5pc@C-1pc{
S\times (M\times M)\ar[r]^-{\mathit{dbl'}}\ar[dd]_{\idmap\times +} & 
   (S\times M)\times (S\times M)\ar[d]^-{\star\times\star} &
   (S\times S)\times M\ar[l]_-{\mathit{dbl}}\ar[dd]^{+\times\idmap} \\
&  M\times M\ar[d]^{+} \\
S\times M\ar[r]^-{*} & M & S\times M\ar[l]_-{*}
}$$

\noindent where $\mathit{dbl}'$ is the obvious duplicator of $S$.
Preservation of zeros is simply $\star \after (0\times \idmap) = 0
\after \pi_{1} \colon 1\times M\rightarrow M$ and $\star \after
(\idmap\times 0) = 0 \after \pi_{2} \colon S\times 1\rightarrow M$.

\auxproof{
$$\xymatrix{
1\times M\ar[r]^-{0\times\idmap}\ar[d]_{\pi_1} &
   S\times M\ar[d]^{\star} &
   S\times 1\ar[l]_-{\idmap\times 0}\ar[d]^{\pi_{2}} \\
1\ar[r]^-{0} & M & 1\ar[l]_-{0}
}$$
}

We shall assemble such semirings and modules in one category
$\Mod(\Cat{C})$ with triples $(S, M, \star)$ as objects, where
$\star\colon S\times M\rightarrow M$ is an action as above.  A
morphism $(S_{1},M_{1},\star_{1})\rightarrow (S_{2},M_{2},\star_{2})$
consists of a pair of morphisms $f\colon S_{1}\rightarrow S_{2}$
and $g\colon M_{1} \rightarrow M_{2}$ in \Cat{C} such that $f$
is a map of semirings, $f$ is a map of monoids, and the actions
interact appropriately: $\star_{2} \after (f\times g) = g
\after \star_{1}$.

\subsection{From semirings to monads}

To construct an adjunction between semirings and additive commutative
monads we start by defining, for each commutative semiring $S$, the
so-called multiset monad on $S$ and show that this monad is both
commutative and additive.

\begin{definition}\label{MultisetDef}
  For a semiring $S$, define a ``multiset'' functor
  $M_{S}\colon\Sets\rightarrow\Sets$ on a set $X$ by
$$\begin{array}{rcl}
M_{S}(X)
& = &
\set{\varphi\colon X\rightarrow S}{\support(\varphi)\mbox{ is finite}},
\end{array}$$

\noindent where $\support(\varphi) = \setin{x}{X}{\varphi(x) \neq 0}$
is called the support of $\varphi$. For a function $f\colon
X\rightarrow Y$ one defines $M_{S}(f) \colon M_{S}(X) \rightarrow
M_{S}(Y)$ by:
$$\begin{array}{rcl}
M_{S}(f)(\varphi)(y)
& = &
\sum_{x\in f^{-1}(y)}\varphi(x).
\end{array}$$
\end{definition}

\noindent Such a multiset $\varphi\in M_{S}(X)$ may be written as
formal sum $s_{1}x_{1}+\cdots+s_{k}x_{k}$, where $\support(\varphi) =
\{x_{1}, \ldots, x_{k}\}$ and $s_{i} = \varphi(x_{i})\in S$ describes
the ``multiplicity'' of the element $x_{i}$. In this notation one can
write the application of $M_S$ on a map $f$ as
$M_{S}(f)(\sum_{i}s_{i}x_{i}) = \sum_{i}s_{i}f(x_{i})$. Functoriality
is then obvious.

\begin{lemma}
\label{CSRng2CAMndLem}
For each semiring $S$, the multiset functor
$M_S$ forms a commutative and additive monad, with unit and multiplication:
$$\xymatrix@R-1.8pc{
X\ar[r]^-{\eta} & M_{S}(X)
& &
M_{S}(M_{S}(X))\ar[r]^-{\mu} & M_{S}(X) \\
x \ar@{|->}[r] & 1x
& &
\sum_{i}s_{i}\varphi_{i} \ar@{|->}[r] &
  \lamin{x}{X}{\sum_{i}s_i\varphi_{i}(x)}.
}$$
\end{lemma}

\begin{proof}
The verification that $M_S$ with these $\eta$ and $\mu$ indeed forms
a monad is left to the reader. We mention that for commutativity and
additivity the relevant maps are given by:
$$\xymatrix@C-1pc@R-1.8pc{
M_{S}(X)\times M_{S}(Y)\ar[r]^-{\dst} & M_{S}(X\times Y)
&
M_{S}(X+Y)\ar[r]^-{\bc} & M_{S}(X)\times M_{S}(Y) \\
(\varphi,\psi)\ar@{|->}[r] & \lam{(x,y)}{\varphi(x)\cdot\psi(y)}
&
\chi\ar@{|->}[r] & (\chi\after\kappa_{1}, \chi\after\kappa_{2}).
}$$

\noindent Clearly, $\bc$ is an isomorphism, making $M_S$ additive. \QED
\end{proof}

\begin{lemma}\label{SRng2CAMndProp}
The assignment $S \mapsto M_S$ yields a functor $\mathcal{M} \colon
\CSRng \to \ACMnd$.
\end{lemma}

\begin{proof}
  Every semiring homomorphism $f \colon S \to R$, gives rise to a
  monad morphism $\mathcal{M}(f) \colon M_S \to M_R$ with components
  defined by $\mathcal{M}(f)_X (\sum_{i}s_{i}x_{i}) =
  \sum_{i}f(s_{i})x_{i}$. It is left to the reader to check that
  $\mathcal{M}(f)$ is indeed a monad morphism. \QED
\end{proof}

\auxproof{
$\mathcal{M}(f)$ is a monad morphism:
\begin{itemize}
\item Naturality: let $g \colon X \to Y$ in $\Sets$,
$$
\begin{array}{rcl}
\mathcal{M}(f)_Y \after M_S(g)(\sum_{i}s_{i}x_{i}) &=& (\sum_{i}f(s_{i})g(x_{i}))\\
&=&
M_R(g) \after \mathcal{M}(f)_X(\sum_{i}s_{i}x_{i})
\end{array}
$$
\item Commutativity with $\eta$:
$$
\begin{array}{rcl}
\mathcal{M}(f) \after \eta^S (x) &=& \mathcal{M}(f)(1^S x) \\
&=&
f(1)x = 1^R x \\
&=& 
\eta^R(x)
\end{array}
$$
\item Commutativity with $\mu$:
$$
\begin{array}{rcl}
\lefteqn{\textstyle\mu \after \mathcal{M}(f)_{M_R(X)} \after 
   M_S(\mathcal{M}(f)_X)(\sum_{i}s_{i}(\sum_{j}t_{ij}x_{j}))} \\
&=& 
\mu \after \mathcal{M}(f)_{M_T(X)}(\sum_{i}s_{i}(\sum_{j}f(t_{ij})x_{j}))\\
&=&
\mu(\sum_{i}f(s_{i})(\sum_{j}f(t_{ij})x_{j}))\\
&=&
\sum_{j}(\sum_i^R f(s_{i})f(t_{ij}))x_{j}\\
&=&
\sum_{j}f(\sum_i^R s_{i}t_{ij})x_{j}\\
&=&
\mathcal{M}(f)_X(\sum_i^T s_{i}t_{ij}x_{j})\\
&=&
\mathcal{M}(f)_X \after \mu (\sum_{i}s_{i}(\sum_{j}f(t_{ij})x_{j}))
\end{array}
$$ 
\end{itemize}
Functoriality of $\mathcal{M}$:\\
For $S \xrightarrow{f} T \xrightarrow{g} R$,
$$
\mathcal{M}(g)_X \after \mathcal{M}(f)_X (\sum_i s_i x_i) = \sum_i g(f(s_i)) x_i = \mathcal{M}(g \after f)(\sum_i s_i x_i)
$$
For $S \xrightarrow{id} S$,
$$
\mathcal{M}(id)(\sum_i s_i x_i) = \sum_i id(s_i) x_i = \sum_i s_i x_i
$$
}

For a semiring $S$, the category $\Alg(M_{S})$ of algebras of the
multiset monad $M_S$ is (equivalent to) the category
$\Mod_{S}(\Cat{C}) \hookrightarrow \Mod(\Cat{C})$ of modules over
$S$. This is not used here, but just mentioned as background
information.


\subsection{From monads to semirings}
A commutative additive monad $T$ on a category $\cat{C}$ gives rise to
two commutative monoid structures on $T(1)$, namely the multiplication
defined in Lemma \ref{CMnd2CMonLem} and the addition defined in Lemma
\ref{Mnd2MonLem2} (considered for $X=1$). In case the category
$\cat{C}$ is distributive these two operations turn $T(1)$ into a semiring.

\begin{lemma}
\label{CAMnd2CSRngLem}
Each commutative additive monad $T$ on a distributive category
$\cat{C}$ with terminal object $1$ gives rise to a semiring $\Evc(T)
= T(1)$ in $\cat{C}$. The mapping $T \mapsto \Evc(T)$ yields a functor
$\ACMnd(\cat{C}) \to \CSRng(\cat{C})$.
\end{lemma}

\begin{proof}
  For a commutative additive monad $T$ on $\cat{C}$, addition on
  $T(1)$ is given by $T(\nabla) \after \bc^{-1} \colon T(1) \times
  T(1) \to T(1)$ with unit $0_{1,1} \colon 1 \to T(1)$ as in
  Lemma~\ref{Mnd2MonLem2}, the multiplication is given by $\mu \after
  T(\lambda) \after \st \colon T(1) \times T(1) \to T(1)$ with unit
  $\eta_1 \colon 1 \to T(1)$ as in Lemma \ref{CMnd2CMonLem}.

It was shown in the lemmas just mentioned that both addition and
multiplication define a commutative monoid structure on $T(1)$. The
following diagram proves distributivity of multiplication over
addition.
$$
\xymatrix{
(T(1) \times T(1))\times T(1) \ar[r]^-{\bc^{-1} \times id} \ar[d]_-{dbl} & T(1+1)\times T(1) \ar[r]^-{T(\nabla)\times id} \ar[d]_-{\st} & T(1) \times T(1) \ar[d]^-{\st}
\\
(T(1) \times T(1))\times (T(1)\times T(1)) \ar[d]_-{\st\times \st} & T((1+1) \times T(1)) \ar[r]^-{T(\nabla \times id)} \ar[d]_-{\cong} & T(1\times T(1)) \ar[ddd]^-{T(\lambda)}
\\
T(1\times T(1)) \times T(1\times T(X)) \ar[r]^-{\bc^{-1}} \ar[d]_-{T(\lambda) \times T(\lambda)} & T(1\times T(1) + 1 \times T(1)) \ar[d]_-{T(\lambda + \lambda}
\\
T^2(1) \times T^2(1) \ar[dd]_-{\mu\times\mu} \ar[r]^-{\bc^{-1}}& T(T(1)+T(1)) \ar[rd]^-{T(\nabla)} \ar[d]_-{T(\cotuple{T(\kappa_1)}{T(\kappa_2)})}
\\
&T^2(1+1) \ar[r]^-{T^2(\nabla)} \ar[d]_-{\mu} & T^2(1) \ar[d]^-{\mu}
\\
T(1)\times T(1) \ar[r]^-{\bc^{-1}} & T(1+1) \ar[r]^-{T(\nabla)} & T(1)
}
$$

\noindent Here we rely on Lemma~\ref{bcproplem} for the commutativity
of the upper and lower square on the left.

In a distributive category $0 \cong 0 \times X$, for every object
$X$. In particular $T(0 \times T(1)) \cong T(0) \cong 1$ is final.
This is used to obtain commutativity of the upper-left square of the
following diagram proving $0 \cdot s = 0$:
$$
\xymatrix{
T(1) \ar[r]^-{\cong} \ar[d]^-{!} & T(0) \times T(1) \ar[r]^-{T(!) \times id} \ar[d]^-{\st} & T(1) \times T(1) \ar[d]^-{\st}
\\
T(0) \ar@/_{6ex}/[ddrr]_{T(!)} \ar[r]^-{\cong} \ar[rd]^{T(!)} & T(0 \times T(1)) \ar[d]^-{T(\lambda)} \ar[r]^-{T(!\times id)} & T(1 \times T(1)) \ar[d]^-{T(\lambda)}
\\
&T^2(1) \ar[r]^{id} & T^2(1) \ar[d]^-{\mu}
\\
&&T(1)
}
$$ For a monad morphism $\sigma \colon T \to S$, we define
$\Evc(\sigma) = \sigma_1 \colon T(1) \to S(1)$. By Lemma
\ref{Mnd2MonLem}, $\sigma_1$ commutes with the multiplicative
structure. As $\sigma_1 = T(id) \after \sigma_1 = \Evx((\sigma,id))$,
it follows from Lemma \ref{Mnd2MonLem2} that $\sigma_1$ also commutes
with the additive structure and is therefore a
$\CSRng$-homomorphism.\QED
\end{proof}

\subsection{Adjunction between monads and semirings}
The functors defined in the Lemmas \ref{SRng2CAMndProp} and
\ref{CAMnd2CSRngLem}, considered on $\cat{C} = \Sets$, form an
adjunction $\mathcal{M} \colon \CSRng \leftrightarrows
\ACMnd \colon \Evc$. To prove this adjunction we first show that
each pair $(T,X)$, where $T$ is a commutative additive monad on a
category $\cat{C}$ and $X$ an object of $\cat{C}$, gives rise to a
module on $\cat{C}$ as defined at the beginning of this section.

\begin{lemma}
\label{Mnd2ModLem}
Each pair $(T,X)$, where $T$ is a commutative additive monad on a
category $\cat{C}$ and $X$ is an object of $\cat{C}$, gives rise to a
module $\Mod(T,X) = (T(1), T(X), \star)$. Here $T(1)$ is the
commutative semiring defined in Lemma \ref{CAMnd2CSRngLem} and $T(X)$
is the commutative monoid defined in Lemma \ref{Mnd2MonLem2}. The
action map is given by $\star = T(\lambda) \after dst \colon T(1)
\times T(X) \to T(X)$. The mapping $(T,X) \mapsto \Mod(T,X)$ yields a
functor $\ACMnd(\cat{C}) \times \cat{C} \to \Mod(\cat{C})$.
\end{lemma}

\begin{proof}
Checking that $\star$ defines an appropriate action requires some work
but is essentially straightforward, using the properties from Lemma
\ref{bcproplem}. For a pair of maps $\sigma \colon T \to S$ in
$\ACMnd(\cat{C})$ and $g \colon X \to Y$ in $\cat{C}$, we define a map
$\Mod(\sigma,g)$ by
$$
\xymatrix{(T(1), T(X), \star^T) \ar[rr]^-{(\sigma_1, \sigma_Y \after T(g))} && (S(1), S(Y), \star^S).}
$$ 

\noindent Note that, by naturality of $\sigma$, one has $\sigma_Y
\after T(g) = S(g) \after \sigma_X$. It easily follows that this
defines a $\Mod(\cat{C})$-map and that the assignment is
functorial.\QED
\end{proof}

\auxproof{First note that $\star$ may also be described as follows.
	$$
	\begin{array}{rcl}
	\star &=& T(\pi_2) \after \dst\\
	&=&
	T(\pi_2) \after \mu \after T(\st') \after \st\\
	&=&
	\mu \after T^2(\pi_2) \after T(\st') \after \st\\
	&=&
	\mu \after T(\pi_2) \after \st
	\end{array} 
	$$
	
\begin{itemize}
\item $\star$ defines an action:
\begin{enumerate}
	\item $ s \star 0 = 0$:
	$$
	\xymatrix{
	T(1) \ar[r]^-{\tuple{id}{!^{-1}}} \ar[dd]_-{!} & T(1)\times T(0) \ar[r]^-{T(id) \times T(!)} \ar[d]_-{dst} &T(1) \times T(X) \ar[d]^-{\dst}
	\\
	&T(1\times 0) \ar[r]^-{T(id\times !)} \ar[d]_-{T(\pi_2)} & T(1\times X) \ar[d]^-{T(\pi_2)}
	\\
	1 \ar[r]^-{\cong} & T(0) \ar[r]^-{T(!)} & T(X)
	}
	$$
	The left square commutes as $T(0)$ is terminal. The upper right square commutes by naturality of $\dst$.
	
	\item $0 \star x = 0$:
	$$
	\xymatrix{
	T(X) \ar[d]_-{!} \ar[r]^-{\cong} & T(0) \times T(X) \ar[d]^-{\st} \ar[r]^-{T(!) \times \id} & T(1) \times T(X) \ar[d]^-{\st}
	\\
	T(0) \ar@/_{6ex}/[ddrr]_{T(!)} \ar[r]^-{T(\cong)} \ar[rd]_-{T(!)} & T(0 \times T(X)) \ar[d]^-{T(\pi_2)} \ar[r]^-{T(!\times\id)} & T(1\times T(X)) \ar[d]^-{T(\pi_2)}\\
	&T^2(X) \ar[r]^-{id} & T^2(X) \ar[d]^-{\mu}\\
	&&T(X)}
	$$
	The upper left square commutes as the category $\cat{C}$ is assumed to be distributive and therefore $T(0 \times T(X)) \cong T(0) \cong 1$ is terminal. 
	
	\item $1 \star x = x$:
	$$
	\begin{array}{rcll}
	T(\pi_2) \after \dst \after (\eta \times \id) 
	&=&
	T(\pi_2) \after \mu \after T(\st) \after \st' \after (\eta \times \id)\\
	&=&
	T(\pi_2) \after \mu \after T(\st) \after T(\eta \times \id) \after \st' &\text{naturality of $\st'$}\\
	&=&
	T(\pi_2) \after \mu \after T(\eta) \after \st'\\
	&=&
	T(\pi_2) \after \st'\\
	&=&
	\pi_2 \colon 1 \times T(1) \to T(1)
	\end{array}
	$$
	
	\item $s \star (x+y) = s\star x + s\star y$
	$$
	\xymatrix{
	T(1) \times (T(X)\times T(X)) \ar[r]^-{\idmap \times \bc^{-1}} \ar[d]^-{dbl} & T(1) \times T(X+X) \ar[d]^-{\st'} \ar[r]^-{\idmap \times T(\nabla)} & T(1) \times T(X) \ar[dd]^-{\st'}
	\\
	(T(1)\times T(X))\times(T(1)\times T(X)) \ar[d]_-{\st'\times\st'} & T(T(1) \times (X+X)) \ar[dr]^-{T(id \times \nabla)} \ar[d]^-{\cong}
	\\
	T(T(1) \times X)\times T(T(1)\times X) \ar[r]^-{\bc^{-1}} \ar[d]_-{T(\st)\times T(\st)} & T(T(1) \times X + T(1) \times X) \ar[r]^-{T(\nabla)} \ar[d]^-{T(\st + \st)}& 
	T(T(1)\times X) \ar[dd]^-{T(\st)}
	\\
	T^2(1\times X) \times T^2(1\times X) \ar[r]^-{\bc^{-1}} \ar[dd]_-{\mu\times\mu} & T(T(1\times X)+T(1\times X)) \ar[rd]^-{T(\nabla)} \ar[d]_-{T(\cotuple{T(\kappa_1)}{T(\kappa_2)})} 
	\\
	&T^2(1 \times X + 1 \times X) \ar[r]^-{T^2(\nabla)} \ar[d]^-{\mu} &	T^2(1\times X) \ar[d]^-{\mu}
	\\
	T(1 \times X)\times T(1\times X) \ar[d]_-{T(\pi_2)\times T(\pi_2)} \ar[r]^-{\bc^{-1}} & T(1\times X + 1\times X) \ar[d]^-{T(\pi_1+\pi_2)} \ar[r]^-{T(\nabla)} &T(1\times X) \ar[d]^-{T(\pi_2)}
	\\
	T(X)\times T(X) \ar[r]^-{\bc^{-1}} & T(X+X) \ar[r]^-{T(\nabla)} &T(X)}
	$$
	\item $(s+t)\star x = (s\star x) + (t\star x)$
	$$
	\xymatrix{
	(T(1)\times T(1))\times T(X) \ar[d]_{dbl'} \ar[r]^{\bc^{-1}\times \id} & T(1+1)\times T(X) \ar[d]^-{\st} \ar[r]^-{T(\nabla) \times \id} & T(1) \times T(X) \ar[d]^-{\st}
	\\
	(T(1)\times T(X))\times(T(1)\times T(X)) \ar[d]_-{\st\times\st} & T((1+1)\times T(X)) \ar[d]^-{T(\cong)} \ar[r]^-{T(\nabla \times \id)} & T(1 \times T(X)) \ar[dd]^-{T(\pi_2)}
	\\
	T(1\times T(X)) \times T(1\times T(X)) \ar[d]_-{T(\pi_2)\times T(\pi_2)} \ar[r]^-{\bc^{-1}} &T(1\times T(X) + 1\times T(X)) \ar[d]^-{T(\pi_2 + \pi_2)}
	\\
	T^2(X) \times T^2(X) \ar[r]^-{\bc^{-1}} \ar[dd]_-{\mu \times \mu} & T(T(X)+T(X)) \ar[r]^-{T(\nabla)} \ar[d]_-{T(\cotuple{T(\kappa_1)}{T(\kappa_2)})} & T^2(X) \ar[dd]^-{\mu}
	\\
	&T^2(X+X)\ar[d]^-{\mu} \ar[ru]_-{T^2(\nabla)}
	\\
	T(X)\times T(X) \ar[r]^-{bc^{-1}} & T(X+X) \ar[r]^-{T(\nabla)} &T(X)
	}
	$$
	
	\item $(s\cdot t) \star x = s\star(t \star x)$
	$$
	\begin{sideways}
	\xymatrix{
	(T(1)\times T(1)) \times T(X) \ar[rr]^-{\st\times\id} \ar[d]^-{\alpha^{-1}} && T(1\times T(1)) \times T(X) \ar[d]^-{\st} \ar[r]^-{T(\pi_2) \times \id} & T^2(1) \times T(X) \ar[r]^-{\mu \times \id} \ar[d]^-{\st} & T(1) \times T(X) \ar[dd]^-{\st}
	\\
	T(1)\times (T(1)\times T(X)) \ar[d]^-{\id\times \st} \ar[r]^-{\st} & T(1\times(T(1)\times T(X))) \ar[r]^-{T(\alpha)} \ar[d]^-{T(\idmap \times\st)} \ar@/_{4ex}/[rr]_{T(\pi_2)} & T((1\times T(1)) \times T(X)) \ar[r]^-{T(\pi_2\times \id)} & T(T(1)\times T(X))\ar[d]^-{T(\st)}
	\\
	T(1)\times T(1\times T(X)) \ar[d]^-{\idmap \times T(\pi_2)} \ar[r]^-{\st} & T(1\times T(1\times T(X))) \ar[d]^-{T(\idmap \times T(\pi_2))} && T^2(1 \times T(X)) \ar[r]^-{\mu} \ar[d]^-{T^2(\pi_2)} & T(1\times T(X)) \ar[d]^-{T(\pi_2)}
	\\
	T(1) \times T^2(X) \ar[r]^-{\st} \ar[d]^-{\idmap \times\mu} & T(1\times T^2(X)) \ar[d]^-{T(\idmap \times \mu)} && T^3(X) \ar[r]^-{\mu} \ar[d]^-{T(\mu)} & T^2(X) \ar[d]^-{\mu}
	\\
	T(1)\times T(X) \ar[r]^-{\st} & T(1\times T(X)) \ar[rr]^-{T(\pi_2)} && T^2(X) \ar[r]^-{\mu} & T(X)}
	\end{sideways}
	$$ 
\end{enumerate}

\item $\Mod(\sigma,g) \colon \Mod(T,X) \to \Mod(S,Y)$ is a map of modules.

In Lemma \ref{CAMnd2CSRngLem} we have shown already that $\sigma_1$ is
a $\SRng$-morphism. The map $\sigma_Y \after T(g)$ is a
$\cat{Mon}(\cat{C})$-map as it preserves $0$:
	$$
	\xymatrix{
	1 \ar[r]^-{\cong} \ar[rd]_-{\cong} & T(0) \ar[r]^-{T(!)} \ar[d]_-{\sigma_0} & T(X) \ar[d]^-{\sigma_X} \ar[r]^-{T(g)} & T(Y) \ar[d]^-{\sigma_Y}
	\\
	&S(0) \ar[r]^-{S(!)} \ar@/_{4ex}/[rr]_{S(!)}& S(X) \ar[r]^-{S(g)} & S(Y)
	}
	$$
and the addition:
	$$
	\begin{array}{rcl}
	\sigma_Y \after T(g) \after T(\nabla) \after \bc^{-1} 
	&=&
	S(g) \after \sigma_X \after T(\nabla) \after \bc^{-1}\\
	&=&
	S(g) \after S(\nabla) \after \sigma_{X+X} \after \bc^{-1}\\
	&=&
	S(g) \after S(\nabla) \after \bc^{-1} \after (\sigma_X \times \sigma_X) \\
	&=&
	S(\nabla) \after S(g+g) \after \bc^{-1} \after (\sigma_X \times \sigma_X)\\
	&=& 
	S(\nabla) \after \bc^{-1} \after (S(g) \times S(g)) \after (\sigma_X \times \sigma_X)\\
	&=&
	S(\nabla) \after \bc^{-1} \after (S(g) \after \sigma_X)\times (S(g) \after \sigma_X)\\
	&=&
	S(\nabla) \after \bc^{-1} \after (\sigma_Y \after T(g))\times (\sigma_Y \after T(g))
	\end{array}
	$$
Furthermore the pair preserves the action:
	$$
	\xymatrix{
	T(1)\times T(X) \ar[r]^{\sigma_1 \times \sigma_X} \ar[d]^-{\dst} \ar@/_{8ex}/[dd]_{\star} & S(1) \times S(X) \ar[d]_-{\dst} \ar[r]^-{\idmap \times S(g)} & S(1) \times S(Y) \ar[d]_-{\dst} \ar@/^{8ex}/[dd]^{\star}
	\\
	T(1\times X) \ar[r]^-{\sigma_{1\times X}} \ar[d]^-{T(\pi_2)} & S(1\times X) \ar[r]^-{S(\idmap \times g)} \ar[d]_-{S(\pi_2)} & S(1\times Y) \ar[d]_-{S(\pi_2)}
	\\
	T(X) \ar[r]^-{\sigma_X} & S(X) \ar[r]^-{S(g)} & S(Y)
	}
	$$
\item Functoriality

$\Mod$ preserves the identity map, as
	$$
	\begin{array}{rcl}
	\Mod(\id_{(T,X)}) &=& \Mod((\id_T,\id_X))\\
	&=&
	((\id_T)_1, (\id_T)_Y \after T(\id_X))\\
	&=&
	(\id_{T(1)}, \id_{T(X)})
	\end{array}
	$$
And it preserves the composition as, for $(T,X) \xrightarrow{(\tau, f)} (S, Y) \xrightarrow{(\sigma, g)} (R,Z)$,
	$$
	\begin{array}{rcl}
	\Mod(\sigma, g) \after Mod(\tau, f) 
	&=&
	(\sigma_1, \sigma_Z \after S(g)) \after (\tau_1, \tau_Y \after T(f))\\
	&=&
	(\sigma_1 \after \tau_1, \sigma_Z \after S(g) \after \tau_Y \after T(f))\\
	&=&
	((\sigma \after \tau)_1, \sigma_Z \after \tau_Z \after T(g) \after T(f))\\
	&=&
	((\sigma \after \tau)_1, (\sigma \after \tau)_Z \after T(g \after f))\\
	&=&
	\Mod(\sigma \after \tau, g \after f)
	\end{array}
	$$
	
\end{itemize}
}

\begin{lemma}
\label{AdjCSR2ACMLem}
The pair of functors $\mathcal{M} \colon \CSRng \leftrightarrows
\ACMnd \colon \Evc$ forms an adjunction, $\mathcal{M} \dashv
\Evc$.
\end{lemma}

\begin{proof}
For a semiring $S$ and a commutative additive monad $T$ on \Sets there are
(natural) bijective correspondences:
$$\begin{bijectivecorrespondence}
  \correspondence[\qquad in \Cat{CAMnd}]
   {\xymatrix{\llap{$M_{S} = \;$}\mathcal{M}(S)\ar[r]^-{\sigma} & T}}
  \correspondence[\qquad in \CSRng]
   {\xymatrix{S\ar[r]_-{f} & \Evc(T)\rlap{$\;=T(1)$}}}
\end{bijectivecorrespondence}$$
Given $\sigma \colon M_S \to T$, one defines a semiring map $\overline{\sigma} \colon S \to T(1)$ by
$$
\xymatrix{
\overline{\sigma} = \Big( S \ar[rr]^-{\lambda s. (\lambda x. s)} && M_S(1) \ar[r]^-{\sigma_1} & T(1) \Big).}
$$
Conversely, given a semiring map $f \colon S \to T(1)$, one gets a monad map $\overline{f} \colon M_S \to T$ with components:
$$\xymatrix{
M_S(X) \ar[r]^-{\overline{f}_X} & T(X) 
  \quad\mbox{given by}\quad
   \textstyle{\sum_i} s_ix_i \ar@{|->}[r] &
   \textstyle{\sum_i} f(s_i) \star \eta(x_i),
}$$

\noindent where the sum on the right hand side is the addition in
$T(X)$ defined in Lemma \ref{Mnd2MonLem2} and $\star$ is the action of
$T(1)$ on $T(X)$ defined in Lemma \ref{Mnd2ModLem}.

Showing that $\overline{f}$ is indeed a monad morphism requires some
work. In doing so one may rely on the properties of the action and on
Lemma \ref{AdditiveMonadMonoidLem}. The details are left to the
reader. \QED
\end{proof}
\auxproof{
The adjunction is proved as follows:
\begin{enumerate}
	\item Given $\sigma$, $\overline{\sigma}$ is a $\CSRng$-morphism (notes p. 88):
	\begin{itemize}
		\item Preservation of 0:
		$$
		\begin{array}{rcll}
		\overline{\sigma}(0_S) &=& \sigma_1(0\cdot *) &* \, \text{is the unique element of the terminal set}\\
		&=&
		\sigma_1(M_S(!)(\emptyset)) &\text{empty function is the unique element of $M_S(0)$}\\
		&&&M_S(!) \colon M_S(0) \to M_S(1), \emptyset \mapsto 0\cdot *\\
		&=&
		T(!) \after \sigma_0(\emptyset)\\
		&=&
		0^T &M_S(0) \cong 1 \cong T(0) 
		\end{array}
		$$
		\item Preservation of the addition:
		$$
		\begin{array}{rcll}
		\overline{\sigma}(a+b) &=& \sigma_1((a+b)\cdot *)\\
		&=&
		\sigma_1 \after M_S(\nabla) \after (\bc^{M_S})^{-1}(\tuple{a\cdot *}{b\cdot*}) &\text{use definition of $\bc$ in $M_S$}\\
		&=&
		T(\nabla) \after \sigma \after (\bc^{M_S})^{-1}(\tuple{a\cdot *}{b\cdot*})\\
		&=&
		T(\nabla)\after (\bc^{T})^{-1}\after \sigma \times \sigma(\tuple{a\cdot *}{b\cdot*})\\ 
		&=&
		T(\nabla)\after (\bc^{T})^{-1}(\tuple{\sigma(a\cdot *)}{\sigma(b\cdot*)})\\
		&=&
		\overline{\sigma}(a) + \overline{\sigma}(b)
		\end{array}
		$$
		\item Preservation of 1:
		$$
		\begin{array}{rcl}
		\overline{\sigma}(1) &=& \sigma_1(1\cdot*)\\
		&=&
		\sigma_1 \after \eta^{M_S}(*)\\
		&=&
		\eta^T(*) = 1
		\end{array}
		$$ 
		\item Preservation of the multiplication
		$$
		\begin{array}{rcll}
		\overline{\sigma}(a\cdot b) &=& \sigma_1((a\cdot b)\cdot *)\\
		&=&
		\sigma_1 \after \mu \after M_S(\pi_2) \after \st^{M_S}(\tuple{a\cdot*}{b\cdot*})\\
		&=&
		\mu \after T(\pi_2)\after \st^T \after \sigma \times \sigma (\tuple{a\cdot*}{b\cdot*}) &\text{see diagram}\\
		&=&
		\overline{\sigma}(a)\cdot\overline{\sigma}(b)
		\end{array}
		$$
		$$
		\xymatrix{
		M_S(1)\times M_S(1) \ar[d]_{\st} \ar[r]^-{\sigma \times id} & T(1) \times M_S(1) \ar[d]_-{\st} \ar[r]^-{id \times \sigma} &T(1)\times T(1) \ar[d]^-{\st}
		\\
		M_S(1 \times M_S(1)) \ar[d]_-{M_S(\pi_2)} \ar[r]^-{\sigma} & T(1\times M_S(1)) \ar[d]^-{T(\pi_2)} \ar[r]^-{T(id \times \sigma)} &T(1\times T(1)) \ar[d]^-{T(\pi_2)}
		\\
		M_S^2(1) \ar[r]^-{\sigma} \ar[d]_-{\mu} &TM_S(1) \ar[r]^-{T\sigma} & T^2(1) \ar[d]^-{\mu}
		\\
		M_S(1) \ar[rr]^{\sigma} && T(1)} 
		$$ 
	 \end{itemize}
	 \item Given $f \colon S \to T(1)$, $\overline{f} \colon M_S \to T$ is a $\ACMnd$-morphism.
	\begin{itemize}
	 	\item Naturality:
		$$
		\begin{array}{rcll}
			T(g) \after \overline{f}_X 
			&=&
			T(g)(\sum^{T(X)}f(s_i) \star \eta(x_i))\\
			&=&
			\sum^{T(Y)}T(g)\big(f(s_i) \star \eta(x_i)\big) &\text{T(g) preserves +, Lemma \ref{AdditiveMonadMonoidLem}}\\
			&=& 
			\sum^{T(Y)} f(s_i) \star(T(g)(\eta(x_i)) &\text{same argument as above}\\
			&=&
			\sum^{T(Y)} f(s_i) \star \eta(g(x_i)) &\text{naturality of $\eta$}\\
			&=&
			\overline{f}_Y(\sum^{T(Y)} s_i \star g(x_i))\\
			&=&
			\overline{f}_Y \after M_S(g)(\sum s_ix_i)
		\end{array}
		$$
		\item $\overline{f}$ commutes with $\eta$:
		$$
		\begin{array}{rcll}
			\overline{f}\after\eta^{M_S}(x) 
			&=&
			\overline{f}(1\cdot x)\\
			&=&
			f(1) \star\eta^T(x)\\
			&=&
			1 \star \eta^T(x) &\text{$f$ preserves 1}\\
			&=&
			\eta^T(x)
		\end{array}
		$$
		\item $\overline{f}$ commutes with $\mu$:\\
		Let $\alpha \in M_S^2(X)$, say $\alpha = \sum_i s_i\sum_j t_{ij} x_{j}$
		$$
		\begin{array}{rcll}
		\lefteqn{\mu^T \after \overline{f}_{T(X)} \after M_S(\overline{f}_X)(\alpha)}\\
		&=&
		\mu^T \after \overline{f}_{T(X)}\big(\sum_i s_i \sum_j^{T(X)} f(t_{ij})\star\eta(x_{j})\big)\\
		&=&
		\mu^T\big(\sum_i^{T^2(X)} f(s_i) \cdot \eta(\sum_j^{T(X)} f(t_{ij})\star\eta(x_{j}))\big)\\
		&=&
		\sum_i^{T(X)} \mu \big(f(s_i) \cdot \eta(\sum_j^{T(X)} f(t_{ij})\star\eta(x_{j}))\big)
		&\text{$\mu$ commutes with $+$, Lemma \ref{AdditiveMonadMonoidLem}}\\
		&=& \sum_i^{T(X)} f(s_i) \star \sum_j^{T(X)} f(t_{ij})\star\eta(x_{j})&
		(1)\\
		&=& \sum_i^{T(X)}\sum_j^{T(X)} f(s_it_{ij})\star \eta(x_{j})&
		\text{properties action}\\
		&&&\text{$f$ pres. mult.}\\
		&=&
		\sum_j^{T(X)}\sum_i^{T(X)} f(s_it_{ij})\star \eta(x_{j})&
		\text{com. of $+$ in $T(X)$}\\
		&=&
		\sum_j^{T(X)}(\sum_i^{T(X)} f(s_it_{ij}))\star \eta(x_{j})&
		\text{property action}\\
		&=&
		\sum_j^{T(X)}(f(\sum_i^S s_it_{ij}))\star\eta(x_{j})&
		\text{$f$ preserves $+$}\\
		&=&
		\overline{f}_X(\sum_j(\sum_i^S s_it_{ij})x_{j})\\
		&=&
		\overline{f}_X \after \mu^{M_S}(\alpha)
		\end{array}
		$$
		(1) relies on the fact that for $a \in T(1), b \in T(X)$, $a\star b = \mu(a \star \eta(b))$:
		$$
		\begin{array}{rcl}
		\mu(a \star \eta(b)) &=& \mu \after T(\pi_2) \after \dst 
		\after \idmap \times \eta (a,b)\\
		&=&
		\mu \after T(\pi_2) \after \st (a,b) \\
		&=&
		T(\pi_2)\after\dst(a,b)\\
		&=&
		a\star b
		\end{array}
		$$
	\end{itemize}
	\item $\overline{\overline{f}} =  f$:
	$$
	\begin{array}{rcl}
		\overline{\overline{f}}(s) &=& \overline{f}_1(s*)\\
		&=&
		f(s)\star\eta(*)\\
		&=&
		f(s) \star 1^{T(1)} = f(s)
	\end{array}
	$$	
	\item $\overline{\overline{\sigma}} = \sigma$:
	$$
	\begin{array}{rcll}
		\overline{\overline{\sigma}}_X(\sum s_ix_i)
		&=&
		\sum^{T(X)} \overline{\sigma}(s_i) \star^T \eta^T(x_i)\\
		&=&
		\sum^{T(X)} \sigma_1(s_i*) \star^T (\sigma_X \after \eta^{M_S}) (x_i)\\
		&=&
		\sum^{T(X)} \sigma_1(s_i*) \star^T \sigma_X(1x_i)\\
		&=&
		\sum^{T(X)} \sigma_X((s_i*) \star^{M_S} (1x_i)) &\text{$\sigma$ preserves $\star$}\\
		&=&
		\sum^{T(X)} \sigma_X(s_ix_i) &\text{def. $\dst$ for $M_S$}\\
		&=&
		\sigma_X(\sum s_ix_i) &\sigma_X \text{preserves +, Lemma \ref{AdditiveMonadMonoidLem}}
	\end{array}
	$$
	\item Naturality: consider
	$$
	\begin{bijectivecorrespondence}
  \correspondence[in \Cat{CAMnd}] {\xymatrix{M_S\ar[r]^-{\mathcal{M}(f)} & M_R \ar[r]^-{\sigma} & T \ar[r]^-{\tau} & V}}
  \correspondence[in \CSRng]{\xymatrix{S\ar[r]_-{f} & R \ar[r]_-{\overline{\sigma}} & T(1) \ar[r]_-{\tau_1} & V(1)}}
\end{bijectivecorrespondence}
	$$
	$$
	\begin{array}{rcll}
	\overline{\tau \after \sigma \after \mathcal{M}(f)}(s)
	&=&
	(\tau \after \sigma \after \mathcal{M}(f))_1(s*)\\
	&=&
	(\tau_1 \after \sigma_1)(f(s)*)\\
	&=&
	(\tau_1 \after \overline{\sigma} \after f) (s)
	\end{array}
	$$	
			
\end{enumerate}
}

Notice that the counit of the above adjunction $\Ev\mathcal{M}(S) =
M_S(1) \to S$ is an isomorphism. Hence this adjunction is in fact a
reflection.

\section{Semirings and Lawvere theories}\label{Semiringcatsec}

In this section we will extend the adjunction between commutative
monoids and symmetric monoidal Lawvere theories depicted in
Figure~\ref{ComMonoidTriangleFig} to an adjunction between commutative
semirings and symmetrical monoidal Lawvere theories with biproducts,
\textit{i.e.}~between the categories $\CSRng$ and $\SMBLaw$.

\subsection{From semirings to Lawvere theories}\label{CSRng2CatSbSec}

Composing the multiset functor $\mathcal{M} \colon \CSRng \to \ACMnd$
from the previous section with the finitary Kleisli construction
$\Kl_{\NNO}$ yields a functor from \CSRng to \SMBLaw. This functor may
be described in an alternative (isomorphic) way by assigning to every
semiring $S$ the Lawvere theory of matrices over $S$, which is defined
as follows.

\begin{definition}\label{MatrCatDef}
  For a semiring $S$, the Lawvere theory $\Mat(S)$ of \emph{matrices
    over S} has, for $n, m \in \mathbb{N}$ morphisms (in $\Sets$) $n
  \times m \to S$, \textit{i.e.}~$n\times m$ matrices over $S$, as
  morphisms $n \to m$.  The identity $\id_n \colon n \to n$ is given
  by the identity matrix:
$$\begin{array}{rcl}
\id_n(i,j) & = & \left\{
\begin{array}{ll}
	1 & \text{if }\, i = j \\
	0 & \text{if }\, i \ne j.
\end{array} \right.
\end{array}$$

\noindent The composition of $g \colon n \to m$ and $h \colon m \to p$
is given by matrix multiplication:
$$
(h \after g)(i,k) = \textstyle \sum_j g(i,j) \cdot h(j,k).
$$
The coprojections $\kappa_1 \colon n \to n+m$ and $\kappa_2 \colon m
\to n+m$ are given by
$$
\kappa_1(i,j) = \left\{
\begin{array}{ll}
	1 & \text{if }\, i = j \\
	0 & \text{otherwise}.
\end{array} \right.
\quad\quad\quad
\kappa_2(i,j) = \left\{
\begin{array}{ll}
	1 & \text{if }\, j \ge n \,\text{and}\, j-n=i \\
	0 & \text{otherwise}.
\end{array} \right.
$$
\end{definition}

\begin{lemma}
\label{KleisliMatLem}
The assignment $S \mapsto \Mat(S)$ yields a functor $\CSRng \to
\Law$. The two functors $\Mat \Ev$ and $\Kl_{\NNO} \colon \ACMnd
\to \cat{Law}$ are naturally isomorphic.
\end{lemma}

\begin{proof}
  A map of semirings $f \colon S \to R$ gives rise to a functor
  $\Mat(f) \colon \Mat(S) \to \Mat(R)$ which is the identity on
  objects and which acts on morphisms by post-composition: $h \colon n
  \times m \to S$ in $\Mat(S)$ is mapped to $f \after h \colon n
  \times m \to T$ in $\Mat(T)$. It is easily checked that $\Mat(f)$ is
  a morphism of Lawvere theories and that the assigment is functorial.

  To prove the second claim we define two natural
  transformations. First we define $\xi \colon \Mat \Ev \to
  \Kl_{\NNO}$ with components $\xi_T \colon \Mat(T(1)) \to
  \Kl_{\NNO}(T)$ that are the identity on objects and send a morphism
  $h \colon n \times m \to T(1)$ in $\Mat(T(1))$ to the morphism
  $\xi_T(h)$ in $\Kl_{\NNO}(T)$ given by
$$
\xymatrix@C2.5pc{
\xi_T(h){=}
\Big(n \ar[rr]^-{\langle h(\_,j) \rangle_{j \in m}} && T(1)^m \ar[r]^-{\bc_m^{-1}}& T(m)\Big),
}
$$
where $\bc_m^{-1}$ is the inverse of the generalised bicartesian map 
$$
\xymatrix{
\bc_m = \Big(T(m) \, = \, T(\coprod_m 1) \ar[r]& T(1)^m\Big).
}
$$

\noindent And secondly, in the reverse direction, we define $\theta
\colon \Kl_{\NNO} \to \Mat \Ev$ with components $\theta_T \colon
\Kl_{\NNO}(T) \to \Mat(T(1))$ that are the identity on objects and
send a morphism $g \colon n \to T(m)$ in $\Kl_{\NNO}(T)$ to the
morphism $\theta_T(g) \colon n \times m \to T(1)$ in $\Mat(T(1))$
given by
\begin{equation}
\label{Kl2MatEqn}
\begin{array}{rcl}
\theta_T(g)(i,j) & = & (\pi_j \after \bc_m \after g) (i).
\end{array}
\end{equation}

\noindent It requires some work, but is relatively straightforward to
check that the components $\xi_T$ and $\theta_T$ are $\Law$-maps. To
prove preservation of the composition by $\xi_T$ and $\theta_T$ one
uses the definition of addition and multiplication in $T(1)$ and
(generalisations of) the properties of the map $\bc$ listed in Lemma
\ref{bcproplem}. A short computation shows that the functors are each
other's inverses.  The naturality of both $\xi$ and $\theta$ follows
from (a generalisation of) point \ref{natbcprop} of Lemma
\ref{bcproplem}.\QED
\end{proof}

\auxproof{
$\Mat(f) \colon \Mat(S) \to \Mat(T)$ is a functor as it preserves the identity:
$$
\Mat(f)(\id_n^S) = f \after \id_n^S = \id_n^R,
$$
($f$ is a semiring homomorphism and therefore preserves $0$ and $1$) and composition:
$$
\begin{array}{rcll}
\Mat(f)(h \after g)(i,j) &=& f(\sum_k h(k,j)\cdot g(i,k))\\
&=&
\sum_k f(h(k,j))\cdot f(g(i,k)) &\text{$f$ preserves addition and multiplication}\\
&=&
\Mat(f)(h) \after \Mat(f)(g).
\end{array}
$$
$\Mat(f)$ strictly preserves finite coproducts as
$$
\begin{array}{rcll}
\Mat(f)(\kappa_i^{\Mat{S}}) &=& f \after \kappa_i^{\Mat{S}}\\
&=&
\kappa_i^{\Mat{R}},
\end{array}
$$
where the last equality follows from the fact that the matrices representing the coprojections consists of only zeros and ones and being a semiring homomorphism $f(0) = 0$ and $f(1) =1$.

Furthermore $\Mat$ itself is a functor $\CSRng \to \Law$, as $\Mat(\id_S) = \id_{\Mat(S)}$ and 
$$
\begin{array}{rcl}
	\Mat(f_2 \after f_1)(h) &=& (f_2 \after f_1) \after h = f_2 \after (f_1 \after h)\\
	&=& \Mat(f_2) \after Mat(f_1)(h)
\end{array}
$$

Now we show that the two functors $\Mat\Ev$ and $\Kl_{\NNO}$ are naturally isomorphic.
\begin{itemize}
\item $\xi_T: \Mat\Ev(T) = \Mat(T(1)) \to \Kl_{\NNO}(T)$ is a $\Law$-map.\\
	\begin{enumerate}
		\item $\xi_T$ preserves the identity:
			$$
			\begin{array}{rcll}
				\xi_T(\id_n^{\Mat(T(1))}) &=&
				\bc^{-1} \after \langle \id_n^{\Mat(T(1))} \after \lambda i. \tuple{i}{j}\rangle_{j \in n}\\
				&=&
				\bc^{-1} \after \langle p_j \rangle_{j \in n} &\text{where $p_j$ is in \eqref{kleisliprojdef}}\\
				&=&
				\eta_n = \id_n^{\Kl_{\NNO}(T)} &\text{Lemma \ref{bcproplem}\ref{comEtaMubcprop}}
			\end{array}
			$$
		\item $\xi_T$ preserves the composition:\\
		Let $g \colon n\times m \to T(1)$ and $h \colon m \times p \to T(1)$.
			$$
			\begin{array}{rcll}
			\xi(h \after g)(\_) &=& \bc^{-1} \after \langle (h \after g)(\_,k)\rangle_{k \in p}\\
			&=& 
			\bc^{-1} \after \langle \sum_{j\in m} h(j,k) \cdot g(\_,j)\rangle_{k\in p} &\text{definition of composition in $\Mat(T(1))$}	
			\end{array}
			$$
			$$
			\begin{array}{rcll}
			(\xi(h) \after \xi(g))(\_) &=& \mu \after T(\xi(h)) \after \xi(g)\\
			&=&
			\mu \after T(\bc^{-1}) \after T(\langle h(\_,k)\rangle_{k\in p}) \after \bc^{-1} \after \langle g(\_,j)\rangle_{j \in m}\\
			&=&
			\bc^{-1} \after \mu^p \after \langle T(\pi_k)\rangle_{k\in p} \after T(\langle h(\_,k)\rangle_{k\in p}) \after \bc^{-1} \after \langle g(\_,j)\rangle_{j \in m}&\text{Lemma \ref{bcproplem}\ref{comEtaMubcprop}}\\
			&=& 
			\bc^{-1} \after \langle \mu \after T(h(\_,k))\rangle_{k \in p} \after \bc^{-1} \after \langle g(\_,j)\rangle_{j \in m}
			\end{array}
			$$
			So it suffices to show that, for $k \in p$,
			$$
			\sum_{j\in m} h(j,k) \cdot g(\_,j) = \mu \after T(h(\_,k)) \after \bc^{-1} \after \langle g(\_,j)\rangle_{j \in m}
			$$
			using the definition of $+$ and $\cdot$ in $T(1)$ this is demonstrated as follows ($j \in m$ everywhere):
			$$
			\xymatrix@C1pc{
	n \ar[r]^-{\langle g(_,j)\rangle_{j}} & T(1)^m \ar[r]^-{\cong} \ar[rdd]_{\id} & (T(1) \times 1)^m \ar[d]_{\st^m} \ar[rr]^-{\prod_{j} \idmap \times \kappa_j} && (T(1) \times m)^m \ar[rr]^-{\prod_{j} \idmap \times h(\_,k)} && (T(1) \times T(1))^m \ar[d]^-{\st^m} \ar[rdd]^-{\cdot}
	\\
	&&T(1 \times 1)^m \ar[rr]^-{\prod_{j} T(\idmap \times \kappa_j)} \ar[d]^-{T(\pi_2)^m}&& T(1 \times m)^m \ar[rr]^-{\prod_{j} T(\idmap \times h(\_,k))} && T(1 \times T(1))^m \ar[d]^-{T(\pi_2)^m}
	\\
	&&T(1)^m \ar[rr]^-{\prod_{j} T(\kappa_j)} \ar[d]^-{\bc^{-1}} && T(m)^m \ar[rr]^-{\prod_{j} T(h(\_,k))} && T^2(1)^m \ar[r]^-{\mu^m} \ar[d]^-{\bc^{-1}}& T(1)^m \ar[d]^-{\bc^{-1}}
	\\
	&T(\coprod_m 1) \ar[r]^-{\cong} & T(m) \ar[rrrrd]_-{T(h(\_,k))} \ar[rr]^-{T(\coprod_j \kappa_j)} &&T(\coprod_j m) \ar[rr]^-{T(\coprod_j h(\_,k)}&&T(\coprod_j T(1)) \ar[d]^-{T(\nabla)}&T(m) \ar[d]^-{T(\nabla)}\\
	&&&&&&T^2(1) \ar[r]^-{\mu} &T(1) 
			}
			$$
	Commutation of the lower right square follows by point (iii) of Lemma \ref{Mnd2MonLem2}.
	
	\item $\xi_T$ preserves coproducts:\\
	Let $\kappa_1^{\Mat(T(1))} \colon n \to n+m$ be the coprojection in $\Mat(T(1))$. To prove:
	$$
	\xi_T(\kappa_1^{\Mat(T(1))}) = \bc^{-1} \after \langle \kappa_1^{\Mat(T(1))}(\_,j)\rangle_{j\in n+m} = \eta \after \kappa_1^{\Sets} = \kappa_1^{\Kl_{\NNO}(T)}
	$$
	This is shown as follows:
	$$
	\begin{array}{rcll}
	\bc \after \eta \after \kappa_1^{\Sets} &=& \langle p_j \rangle_{j \in n+m} \after \kappa_1^{\Sets} &\text{Lemma \ref{bcproplem}(iii)}\\
	&=&  
	\langle \kappa_1^{\Mat(T(1))}(\_,j)\rangle_{j\in n+m},
	\end{array}
	$$
	where the last equality follows from the fact that $0^{T(1)} = 0_{1,1}(*)$ and $1^{T(1)} = \eta_1(*)$
	\end{enumerate}

\item $\theta_T \colon \Kl_{\NNO}(T) \to \Mat(T(1))$ is a $\Law$-map.
\begin{enumerate}
	\item $\theta_T$ preserves the identity:
		$$
		\begin{array}{rcll}
		\theta_T(\id_n^{\Kl_{\NNO}(T)}(\_,j) &=&
		\pi_j \after \bc \after \eta_n\\ 
		&=&
		\pi_j \after \langle p_j \rangle_{j \in n} &\text{Lemma \ref{bcproplem}\ref{comEtaMubcprop}}\\
		&=&
		p_j = \id_n^{\Mat(T(1))}(\_,j)
		\end{array}
		$$
	\item $\theta_T$ preserves the composition:\\
	Let $g\colon n \to T(m)$ and $h\colon m \to T(p)$. Then $\theta_T(h \after g) \colon n \times p \to T(1)$. We fix the second coordinate and consider it as a map $n \to T(1)$.
	$$
	\begin{array}{rcll}
	\lefteqn{(\theta(h)\after\theta(g))(\_,k)}\\
	&=&
	\sum_j \theta(h)(j,k) \cdot \theta(g)(\_,j)\\
	&=&
	T(\nabla) \after \bc^{-1} \after \mu^m \after T(\pi_2)^m \after \st^m \after \prod_j (\idmap \times \theta(h)(\_,k)) \after \prod_j (\idmap \times \kappa_j) \after \\ &&(\rho^{-1})^m \after \bc \after g\\
	&=&
	T(\nabla) \after \bc^{-1} \after \mu^m \after T(\pi_2)^m \after \prod_j T(\idmap \times \theta(h)(\_,k)) \after \prod_j T(\idmap \times \kappa_j) \after \st^m \after \\ && (\rho^{-1})^m \after \bc \after g\\
	&=&
	T(\nabla) \after \bc^{-1} \after \mu^m \after \prod_j T(\theta(h)(\_,k)) \after \prod_j T(\kappa_j) \after T(\pi_2)^m \after \st^m \after \\ && (\rho^{-1})^m \after \bc \after g\\
	&=&
	T(\nabla) \after \bc^{-1} \after \mu^m \after \prod_j T(\theta(h)(\_,k)) \after \prod_j T(\kappa_j) \after \bc \after g\\
	&=&
	\mu \after T(\nabla) \after \bc^{-1} \after \prod_j T(\theta(h)(\_,k)) \after \prod_j T(\kappa_j) \after \bc \after g\\
	&=&
	\mu \after T(\nabla) \after T(\coprod_j \theta(h)(\_,k)) \after T(\coprod_j \kappa_j) \after \bc^{-1} \after \bc \after g\\
	&=&
	\mu \after T(\theta(h)(\_,k)) \after g(\_)\\
	&=&
	\mu \after T(\pi_k) \after T(\bc) \after T(h) \after g\\
	&=&
	\mu \after \pi_k \after \langle T(\pi_k)\rangle_{k\in p} \after T(\bc) \after T(h) \after g\\
	&=&
	\pi_k \after \mu^p \after \langle T(\pi_k)\rangle_{k\in p} \after T(\bc) \after T(h) \after g\\
	&=&
	\pi_k \after \bc \after \mu \after T(h) \after g\\
	&=&
	\theta(h \after g)(\_,k)  
	\end{array}
	$$
	
	\item $\theta_T$ preserves the coproduct structure\\
	Consider $\kappa_1^{\Kl_{\NNO}(T)} \colon n \to T(n+m) = \eta \after \kappa_1^{\Sets}$. Then
	$$
	\theta_T(\kappa_1^{\Kl_{\NNO}(T)}) \colon n \times (n+m) \to T(1)
	$$
	and
	$$
	\begin{array}{rcll}
	\theta_T(\kappa_1^{\Kl_{\NNO}(T)})(\_,j) &=& \pi_j \after \bc_m \after \eta \after \kappa_1^{\Sets}\\
	&=&
	\pi_j \after \langle p_j \rangle_{j \in n+m} \after \kappa_1^{\Sets} &\text{(Lemma \ref{bcproplem}(iii))}\\
	&=&
	p_j \after \kappa_1^{\Sets}\\
	&=&
	\kappa_1^{\Mat(T(1))}(\_,j)
	\end{array}
	$$
\end{enumerate}	
\item Naturality of $\xi \colon \Mat\Ev \to \Kl_{\NNO}$.\\
Let $\sigma \colon T \to V$ and $h \colon n \times m \to T(1)$ in $(\Mat\Ev)(T)$.
$$
\begin{array}{rcl}
	(\Kl_{\NNO}(\sigma) \after \xi_T)(h) 
	&=&
	\sigma_m \after \bc^{-1} \after \langle h \after \lambda i.\tuple{i}{j} \rangle_{j \in m}\\
	&=&
	\bc^{-1} \after \sigma_1^m \after \langle h \after \lambda i.\tuple{i}{j} \rangle_{j \in m}\\
	&=&
	\bc^{-1} \after \langle (\sigma_1 \after h) \after \lambda i.\tuple{i}{j} \rangle_{j \in m}\\
	&=&
	(\xi_V \after (\Mat\Ev)(\sigma))(h)
\end{array}
$$
\item Naturality of $\theta \colon \Kl_{\NNO} \to \Mat\Ev$.\\
Let $\sigma \colon T \to V$ and $g \colon n \to T(m)$ in $\Kl_{\NNO}(T)$.
$$
\begin{array}{rcl}
((\Mat\Ev)(\sigma)\after \theta_T)(g)(\_,j) 
	&=&
	\sigma_1 \after \pi_j \after \bc \after g\\
	&=&
	\pi_j \after \sigma_1^m \after \bc \after g\\
	&=&
	\pi_j \after \bc \after \sigma_m \after g \\
	&=&
	(\theta_V \after \Kl_{\NNO}(\sigma))(g)(\_,j) 
\end{array}
$$

\item $\xi_T \after \theta_T = id$.\\
Let $g \colon n \to T(m)$ in $\Kl_{\NNO}(T)$.
$$
\begin{array}{rcl}
(\xi_T \after \theta_T)(g) &=& \bc^{-1} \after \langle \theta_T(g) \after \lambda i.\tuple{i}{j}\rangle_{j \in m}\\
&=&
	\bc^{-1} \after \langle \pi_j \after \bc \after g \rangle_{j \in m}\\
&=&
	\bc^{-1} \after \bc \after g\\
&=&
	g
\end{array}
$$
\item $\theta_T \after \xi_T = id$.\\
Let $h \colon n\times m \to T(1)$ in $\Mat(T(1))$.
$$
\begin{array}{rcl}
(\theta_T \after \xi_T) (i,j) &=& \pi_j \after \bc \after \xi_T(h)(i)\\
&=&
	\pi_j \after \bc \after \bc^{-1} \after \langle h \after \lambda x.\tuple{x}{k}\rangle_{k \in m}(i)\\
	&=&
	\pi_j\big(\langle h(i,k)\rangle_{k\in m}\big)\\
	&=&
	h(i,j)
\end{array}  
$$
\end{itemize}
}

The pair of functors $\mathcal{M} \colon \CSRng \leftrightarrows
\ACMnd \colon \Ev$ forms a reflection, $\Ev\mathcal{M} \cong \id$
(Lemma~\ref{AdjCSR2ACMLem}). Combining this with the previous
proposition, it follows that also the functors $\Mat,
\Kl_{\NNO}\mathcal{M} \colon \CSRng \to \Law$ are naturally
isomorphic. Hence, the functor $\Mat \colon \CSRng \to \Law$ may be
viewed as a functor from commutative semirings to symmetric monoidal
Lawvere theories with biproducts. For a commutative semiring $S$ the
projection maps $\pi_1 \colon n+m \to n$ and $\pi_2 \colon n + m \to
m$ in $\Mat(S)$ are defined in a similar way as the coprojection maps
from Definition \ref{MatrCatDef}. For a pair of maps $g \colon m \to
p$, $h \colon n \to q$, the tensor product $g \otimes h \colon (m
\times n) \to (p \times q)$ is the map $g \otimes h \colon (m \times
n) \times (p \times q) \to S$ defined as
$$\begin{array}{rcl}
(g \otimes h)((i_0,i_1),(j_0,j_1))
& = &
g(i_0,j_0) \cdot h(i_1,j_1),
\end{array}$$ 
where $\cdot$ is the multiplication from $S$.

\subsection{From Lawvere theories to semirings}\label{Cat2CSRngSec}

In Section~\ref{ComMonoidSubsec}, just after Lemma~\ref{CMon2CMndLem},
we have already seen that the homset $\cat{L}(1,1)$ of a Lawvere
theory $\cat{L} \in \SMLaw$ is a commutative monoid, with
multiplication given by composition of endomaps on $1$. In case $\cat{L}$ also has biproducts we have, by
\eqref{HomsetPlus}, an addition on this homset, which is preserved by
composition. Combining those two monoid structures yields a semiring
structure on $\cat{L}(1,1)$. This is standard, see
\textit{e.g.}~\cite{AbramskyC04,KellyL80,Heunen10a}. The assignment of
the semiring $\cat{L}(1,1)$ to a Lawvere theory $\cat{L} \in \SMBLaw$
is functorial and we denote this functor, as in
Section~\ref{ComMonoidSubsec}, by $\mathcal{H} \colon \SMBLaw \to
\CSRng$.

\subsection{Adjunction between semirings and Lawvere theories}

Our main result is the adjunction on the right in the triangle of
adjunctions for semirings, see Figure~\ref{CSRngTriangleFig}.


\begin{lemma}
\label{CSRng2CatAdjLem}
The pair of functors $\Mat \colon \CSRng \rightleftarrows \SMBLaw
\colon \mathcal{H}$, forms an adjunction $\Mat \dashv \mathcal{H}$.
\end{lemma}

\begin{proof}
For $S \in \CSRng$ and $\cat{L} \in \SMBLaw$ there are
(natural) bijective correspondences:
$$\begin{bijectivecorrespondence}
  \correspondence[in \SMBLaw]{\xymatrix{\Mat(S)\ar[r]^-{F} & \cat{L}}}
  \correspondence[in \CSRng]{\xymatrix{S\ar[r]_-{f} & \mathcal{H}(\cat{L})}}
\end{bijectivecorrespondence}$$
\
\noindent Given $F$ one defines a semiring map $\overline{F}
\colon S\rightarrow \mathcal{H}(\cat{L}) = \cat{L}(1,1)$ by 
$$
s \mapsto F(1 \times 1 \xrightarrow{\lam{x}{s}} S).
$$ 

\noindent Note that $1 \times 1 \xrightarrow{\lam{x}{s}} S$ is an
  endomap on $1$ in $\Mat(S)$ which is mapped by $F$ to an element of
  $\cat{L}(1,1)$.

Conversely, given $f$ one defines a \SMBLaw-map $\overline{f} \colon
\Mat(S) \rightarrow \cat{L}$ which sends a morphism $h \colon n \to m$
in $\Mat(S)$, \textit{i.e.} $h \colon n \times m \to S$ in $\Sets$, to
the following morphism $n\rightarrow m$ in $\cat{L}$, forming an
$n$-cotuple of $m$-tuples
$$\xymatrix@C+1pc{
\overline{f}(h) = \Big(n
   \ar[rr]^-{\big[\big\langle f(h(i,j))\big\rangle_{j < m}\big]_{i < n}} 
   &&m\Big)}. 
$$
It is readily checked that $\overline{F} \colon S \to \cat{L}(1,1)$ is
a map of semirings. To show that $\overline{f} \colon \Mat(S) \to
\cat{L}$ is a functor one has to use the definition of the semiring
structure on $\cat{L}(1,1)$ and the properties of the biproduct on
$\cat{L}$. One easily verifies that $\overline{f}$ preserves the
biproduct. To show that it also preserves the monoidal structure one
has to use that, for $s, t \in \cat{L}(1,1)$, $s \otimes t = t\after s
\,(=s \after t)$. \QED

\end{proof}
\auxproof{
\begin{enumerate}
	 \item $\overline{F}$ is a semiring homomorphism
	 	\begin{itemize}
	 		\item $\overline{F}$ preserves $\cdot$
	 		$$			
	 			\begin{array}{rcll}
	 			\overline{F}(s\cdot t) &=& F(1 \times 1 \xrightarrow{{\lambda x. (s\cdot t)}} S)\\
	 			&=&
	 			F((1 \times 1 \xrightarrow{\lam{x}{s}} S) \after (1 \times 1 \xrightarrow{{\lambda x. t}} S))\\
	 			&=&
	 			F(1 \times 1 \xrightarrow{\lam{x}{s}} S) \after F(1 \times 1 \xrightarrow{{\lambda x. t}} S)\\
	 			&=&
	 			 			F(1 \times 1 \xrightarrow{{\lambda x. t}} S) \cdot^{\mathcal{H}(\cat{C})} F(1 \times 1 \xrightarrow{\lam{x}{s}} S)\\
	 			&=&
	 			\overline{F}(t) \cdot \overline{F}(s)\\
	 			&=& \overline{F}(s) \cdot \overline{F}(t)
	 			&\text{commutativity}
	 			\end{array}
	 		$$
	 		\item $\overline{F}$ preserves $+$
	 		$$
	 			\begin{array}{rcll}
	 			\overline{F}(s + t) &=& F(1 \times 1 \xrightarrow{{\lambda x. (s+t)}} S)\\
	 			&=& F(1 \xrightarrow{\cotuple{s}{t}} 1 \oplus 1 \xrightarrow{\nabla} 1)\\
	 			&=& I \xrightarrow{\cotuple{F(s)}{F(t)}} I \oplus I \xrightarrow{\nabla} I &\text{$F$ preserves all structure}\\
	 			&=& F(s) + F(t) 	 			
	 			\end{array}
	 		$$
	 		\item $\overline{F}$ preserves 0 
	 		$$
	 			\begin{array}{rcl}
	 			\overline{F}(0) &=& F(1 \times 1 \xrightarrow{{\lambda x. 0}} S)\\
	 			&=& F(1 \xrightarrow{!} 0 \xrightarrow{!} 1)\\
	 			&=& F(1)\xrightarrow{F(!)} F(0) \xrightarrow{F(!)} F(1)\\
	 			&=& 1 \xrightarrow{!} 0 \xrightarrow{!} 1 = 0_{\mathcal{H}(\cat{L})} 
	 			\end{array}
	 		$$
	 		\item $\overline{F}$ preserves 1 
	 		$$
	 			\begin{array}{rcl}
	 			\overline{F}(1) &=& F(1 \times 1 \xrightarrow{{\lambda x. 1}} S)\\
	 			&=& F(id_1)\\
	 			&=& id_1\\
	 			&=& 1_{\mathcal{H}(\cat{L})} 
	 			\end{array}
	 		$$
	 	\end{itemize}
	 	\item $\overline{f}$ is a \SMBLaw-morphism
	 	\begin{itemize}
	 		\item $\overline{f}$ is a functor\\
	 			Let $n$ be an object of $\Mat(S)$,
	 				$$\xymatrix@C+1pc{
					\overline{f}(id_n) = \Big(n 
   				\ar[rr]^-{\big[\big\langle f(id_n(i,j))\big\rangle_{j\in m}\big]_{i \in n}} 
   				&&n\Big)}. 
					$$
				As $f$ is semiring homomorphism, it preserves $0$ and $1$. Hence
				$$
				f(\id_n(i,j)) = \left\{
				\begin{array}{ll}
					id_1 & \text{if }\, i = j \\
					0_{1,1} & \text{if }\, i \ne j.
				\end{array} \right.
				$$
				It follows that $\overline{f}(\id_n) = \id_{\bigoplus_n I}$ (using the properties of the biproduct).\\
				\\
	 			Let $g \colon n \to m$ and $h \colon m \to p$ in $\Mat{S}$.
	 			$$
	 			\begin{array}{rcl}
	 			\overline{f}(h \after g) &=& \big[\big\langle f(h \after g(i,k))\big\rangle_{k\in p}\big]_{i \in n}\\
	 			&=&
	 			\big[\big\langle f(\sum_j h(j,k) \cdot g(i,j))\big\rangle_{k\in p}\big]_{i \in n}\\
	 			&=&
	 			\big[\big\langle f(\sum_j g(i,j) \cdot h(j,k))\big\rangle_{k\in p}\big]_{i \in n}\\
	 			&=&
	 			\big[\big\langle \sum_j f(g(i,j)) \cdot f(h(j,k))\big\rangle_{k\in p}\big]_{i \in n}\\
	 			&=&
	 			\big[\big\langle \nabla \after \oplus_{j \in m} ((f(h(j,k)) \after f(g(i,j)))) \after \Delta \big\rangle_{k\in p}\big]_{i \in n}\\
	 			&=&
	 			\big[\big\langle \nabla \after (\oplus_{j \in m} f(h(j,k))) \after (\oplus_{j\in m} f(g(i,j))) \after \Delta \big\rangle_{k\in p}\big]_{i \in n}\\
	 			&=&
	 			\big[\big\langle [f(h(j,k))]_{j\in m} \after \big\langle f(g(i,j)) \big\rangle_{j\in m} \big\rangle_{k\in p}\big]_{i \in n}\\
	 			&=&
	 			\big[\big\langle [f(h(j,k))]_{j\in m} \big\rangle_{k\in p} \after \big\langle f(g(i,j)) \big\rangle_{j\in m} \big]_{i \in n}\\
	 			&=&
	 			\big\langle [f(h(j,k))]_{j\in m} \big\rangle_{k\in p} \after \big[\langle f(g(i,j)) \rangle_{j\in m} \big]_{i \in n}\\
	 			&=&
	 			\big[\big\langle f(h(j,k))\big\rangle_{k\in p}\big]_{j\in m} \after \big[\langle f(g(i,j)) \rangle_{j\in m} \big]_{i \in n}\\
	 			&=&
	 			\overline{f}(h) \after \overline{f}(g)
	 			\end{array}
	 			$$
	 			
	 		\item $\overline{f}$ preserves the biproduct structure\\
	 		We have to show that the canonical maps: 
	 		$$
	 		n +m = \overline{f}(n \oplus m) \xrightarrow{\tuple{\overline{f}(\pi_1)}{\overline{f}(\pi_2)}} \overline{f}(n) \oplus \overline{f}(m) = n+m
	 		$$
	 		and
	 		$$
	 		n+m = \overline{f}(n) \oplus \overline{f}(m) \xrightarrow{\cotuple{\overline{f}(\kappa_1)}{\overline{f}(\kappa_2)}} \overline{f}(n \oplus m) = n+m
	 		$$
	 		are mutually inverse.
	 		$$
	 		\begin{array}{rcl}
	 			\tuple{\overline{f}(\pi_1)}{\overline{f}(\pi_2)} &=& \tuple{\big[\langle f(\pi_1(i,x))\rangle_{i \in n}\big]_{x \in n+m}}{\big[\langle f(\pi_2(j,x))\rangle_{j \in m}\big]_{x \in n+m}}\\
	 			&=&
	 			\tuple{\big\langle [f(\pi_1(i,x))]_{x \in n+m}\big\rangle_{i \in n}}{\big\langle [f(\pi_2(j,x))]_{x \in n+m}\big\rangle_{j \in m}}\\
	 			&=&
	 			\id_{n+m}
	 		\end{array} 
	 		$$
	 		And similarly $\cotuple{\overline{f}(\kappa_1)}{\overline{f}(\kappa_2)} = \id_{\oplus_{n+m}I}$ 
	 		\item $\overline{f}$ preserves the monoidal structure\\
	 		On objects:
	 		$$
	 		\begin{array}{rcll}
	 		\overline{f}(n)\otimes \overline{f}(m) &=& n \times m\\
	 		&=&
	 		\overline{f}(n\times m)
	 		\end{array}
	 		$$	 
						
		Furthermore we have to show that for $h_1 \colon m \to p$ and $h_2 \colon n \to q$ in $\Mat(S)$, $\overline{f}(h_1 \otimes h_2) = \overline{f}(h_1) \otimes \overline{h_2}$. By way of example we consider $h_1 \colon 2 \to 1$ and $h_2 \colon 2 \to 1$,
		Consider the canonical identities:
		$$
		\beta_1 = 1 \otimes 1 + 1 \otimes 1 + 1 \otimes 1 + 1 \otimes 1 \xrightarrow{[\kappa_1 \after (\idmap \otimes \kappa_1), \kappa_1 \after (\idmap \otimes \kappa_2),\kappa_2 \after (\idmap \otimes \kappa_1), \kappa_2 \after (\idmap \otimes \kappa)]} 1 \otimes (1+1) + 1 \otimes (1+1)
		$$
		and
		$$
		\beta_2 = 1 \otimes (1+1) + 1 \otimes (1+1) \xrightarrow{[\kappa_1 \otimes \id, \kappa_2 \otimes \id]} (1+1) \otimes (1+1)
		$$
		We have to show:
		$$
		(\overline{f}(h_1) \otimes \overline{f}(h_2)) \after \beta_2 \after \beta_1 = \overline{f}(h_1 \otimes h_2)
		$$
		We just show
		$$ 
		(\overline{f}(h_1) \otimes \overline{f}(h_2)) \after \beta_2 \after \beta_1 \after \kappa_1 = \overline{f}(h_1 \otimes h_2) \after \kappa_1,
	 	$$
	 	showing equality under composition with the other three coprojections is done in a similar way.
	 	$$
	 	\begin{array}{rcll}
	 	\lefteqn{(\overline{f}(h_1) \otimes \overline{f}(h_2)) \after \beta_2 \after \beta_1 \after \kappa_1} \\
	 	&=& 
	 	 (\overline{f}(h_1) \otimes \overline{f}(h_2)) \after \beta_2 \after \kappa_1 \after (\idmap \otimes \kappa_1)\\
	 	&=&
	 	 (\overline{f}(h_1) \otimes \overline{f}(h_2)) \after (\kappa_1 \otimes \id) \after (\idmap \otimes \kappa_1)\\
	 	 &=&
	 	 (\overline{f}(h_1) \after \kappa_1) \otimes (\overline{f}(h_2) \after \kappa_1)\\
	 	 &=&
	 	 f(h_1(0,0)) \otimes f(h_2(0,0)) &\text{(def. $\overline{f}$)}\\
	 	 &=&
	 	 f(h_2(0,0)) \after f(h_1(0,0)) &(s \otimes t = t \after s)\\
	 	 &=&
	 	 f(h_2(0,0) \cdot h_1(0,0)) &(\text{$f$ pres. mult.})\\
	 	 &=&
	 	 \overline{f}(h_1 \otimes h_2) \after \kappa_1 &(\text{def. $\otimes$ in $\Mat(S)$}) 
	 	\end{array}
	 	$$
	 	\end{itemize}
	 	\item $\overline{\overline{F}} = F$\\
		Let us denote $\underline{s} = 1 \times 1 \xrightarrow{\lam{x}{s}} S$.\\
	 	Clearly $\overline{\overline{F}} = F$ on objects. Now let $h \colon n \to m$ in $\Mat(S)$, \textit{i.e} $h \colon n\times m \to S$ in \cat{Sets}, by definition
	 	$$
		\xymatrix@C+1pc{
\overline{\overline{F}}(h) = \Big(n 
   \ar[rr]^-{\big[\big\langle \overline{F}(h(i,j))\big\rangle_{j\in m}\big]_{i \in n}} 
   &&m\Big)}.
	 	$$
	 	$$
	 		\begin{array}{rcll}
	 		\big[\big\langle \overline{F}(h(i,j))\big\rangle_{j\in m}\big]_{i \in n} &=& \big[\big\langle F(\underline{h(i,j)})\big\rangle_{j\in m}\big]_{i \in n} &\text{definition $\overline{F}$}\\
	 		&=&
	 		F(\big[\big\langle \underline{h(i,j)}\big\rangle_{j\in m}\big]_{i \in n}) &\text{$F$ preserves the biproduct structure}\\
	 		&=&
	 		F(h) &\text{definition of (co)tuples in $\Mat(S)$}
	 		\end{array}	
	 	$$
	 	
	 	\item $\overline{\overline{f}} = f$\\
	 	$$
	 		\begin{array}{rcll}
	 		\overline{\overline{f}}(s) &=& \overline{f}(\underline{s})\\
	 		&=&
	 			I \xrightarrow{f(\underline{s}(*,*))} I &\text{definition $\overline{f}$}\\
	 		&=&
	 			f(s)	
	 		\end{array}
	 	$$ 
	 	\item For naturality consider
		$$\begin{bijectivecorrespondence}
  \correspondence[in \SMBLaw]{\xymatrix{\Mat(S)\ar[r]^-{\Mat(g)} & \Mat(R) \ar[r]^-{F} & \cat{L} \ar[r]^-{G} & \cat{K}}}
  \correspondence[in \CSRng]{\xymatrix{S\ar[r]_-{g} & R \ar[r]_-{\overline{F}} & \mathcal{H}(\cat{L}) \ar[r]_-{\mathcal{H}(G)} & \mathcal{H}(\cat{K})}}
\end{bijectivecorrespondence}$$
Then
  $$
  \begin{array}{rcll}
  	\overline{G \after F \after \Mat(g)}(s) &=& G \after f \after \Mat(g)(\underline{s})\\
  	&=&
  	G \after F (\underline{g(s)})\\
  	&=&
  	G \after \overline{F}(g(s)) &\text{definition $\overline{F}$}\\
  	&=&
  	\mathcal{H}(G) \after \overline{F} \after g(s) &\text{definition $\mathcal{H}(G)$}
  \end{array}
  $$
\end{enumerate}
}

The results of Section~\ref{SemiringMonadSec} and~\ref{Semiringcatsec}
are summarized in Figure \ref{CSRngTriangleFig}.

\begin{figure}
$$\xymatrix@R-.5pc@C+.5pc{
& & \CSRng\ar@/_2ex/ [ddll]_{\cal M}
   \ar@/_2ex/ [ddrr]_(0.3){\Mat \cong \Kl_{\NNO}\mathcal{M}\;\;} \\
& \dashv & & \dashv & \\
\ACMnd\ar @/_2ex/[rrrr]_{\Kl_\NNO}
   \ar@/_2ex/ [uurr]_(0.6){\;{\Ev} \cong \mathcal{H}\Kl_{\NNO}} & & \bot & &  
   \SMBLaw\ar@/_2ex/ [uull]_{\mathcal{H}}\ar @/_2ex/[llll]_{\LM}
}$$
\caption{Triangle of adjunctions starting from commutative semirings,
  with commutative additive monads, and symmetric
  monoidal Lawvere theories with biproducts.}
\label{CSRngTriangleFig}
\end{figure}
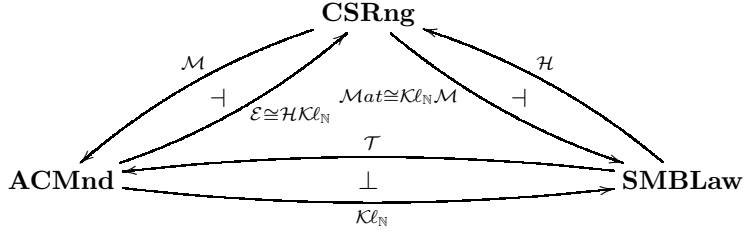

\section{Semirings with involutions}\label{InvolutionSec}

In this final section we enrich our approach with involutions.
Actually, such involutions could have been introduced for monoids
already. We have not done so for practical reasons: involutions on
semirings give the most powerful results, combining daggers on
categories with both symmetric monoidal and biproduct structure.

An involutive semiring (in \Sets) is a semiring $(S, +, 0, \cdot, 1)$
together with a unary operation $*$ that preserves the addition and
multiplication, \textit{i.e.}~$(s+t)^* = s^*+t^*$ and $0^* = 0$, and
$(s\cdot t)^* = s^* \cdot t^*$ and $1^*=1$, and is involutive,
\textit{i.e.}~$(s^*)^* = s$. The complex numbers with conjugation form
an example. We denote the category of involutive semirings, with
homomorphisms that preserve all structure, by $\ICSRng$.


The adjunction $\mathcal{M} \colon \CSRng \leftrightarrows \cat{ACMnd}
\colon \Evc$ considered in Lemma~\ref{AdjCSR2ACMLem} may be restricted
to an adjunction between involutive semirings and so-called involutive
commutative additive monads (on \Sets), which are commutative additive
monads $T$ together with a monad morphism $\zeta\colon T \to T$
satisfying $\zeta \after \zeta = id$. We call $\zeta$ an involution on
$T$, just as in the semiring setting. A morphism between such monads
$(T,\zeta)$ and $(T',\zeta')$, is a monad morphism $\sigma \colon T
\to T'$ preserving the involution, \textit{i.e.}~satisfying $\sigma
\after \zeta = \zeta' \after \sigma$. We denote the category of
involutive commutative additive monads by $\IACMnd$.

\begin{lemma}
  The functors $\mathcal{M} \colon \CSRng \leftrightarrows \ACMnd
  \colon \Evc$ from Lemma \ref{SRng2CAMndProp} and Lemma
  \ref{CAMnd2CSRngLem} restrict to a pair of functors $\mathcal{M}
  \colon \ICSRng \leftrightarrows \IACMnd \colon \Evc$. The restricted
  functors form an adjunction $\mathcal{M} \vdash \Evc$.
\end{lemma}

\begin{proof}
  Given a semiring $S$ with involution $*$, we may define an
  involution $\zeta$ on the multiset monad $\mathcal{M}(S) = M_S$ with
  components
$$
\zeta_X \colon M_S(X) \to M_S(X), \qquad 
   \textstyle\sum_i s_ix_i \mapsto \sum_i s_i^*x_i.
$$ 
Conversely, for an involutive monad $(T, \zeta)$, the map $\zeta_1$ gives an involution on the semiring $\Ev(T) = T(1)$.

A simple computation shows that the unit and the counit of the
adjunction $\mathcal{M} \colon \CSRng \leftrightarrows \ACMnd \colon
\Evc$ from Lemma~\ref{AdjCSR2ACMLem} preserve the involution (on
semirings and on monads respectively). Hence the restricted functors
again form an adjunction. \QED
\end{proof}

\auxproof{
\begin{itemize}
\item $\mathcal{M} \colon \CSRng \leftrightarrows \ACMnd$ restricts to a functor $\ICSRng \leftrightarrows \IACMnd$.\\
Using Lemma \ref{SRng2CAMndProp}, it is only left to show that $\zeta \colon M_S \to M_S$ is a monad morphism s.t. $\zeta \after \zeta = \id$ and that for a morphism of involutive semirings $g \colon S \to R$, $\mathcal{M}(g)$ preserves the involution.
\begin{enumerate}
	\item Naturality\\
	Let $f \colon X \to Y$, then
	$$
	\begin{array}{rcl}
(\zeta_Y \after M_S(f))(\sum_i s_i x_i) &=& \zeta_Y(\sum_i s_if(x_i))\\
	&=&
	\sum_i s_i^*f(x_i)\\
	&=&
	(M_S(f) \after \zeta_X)(\sum_i s_i x_i)
	\end{array}
	$$ 
	\item Commutativity with $\eta$:
	$$
	\zeta \after \eta (x) = \zeta(1x) = 1^*x = 1x = \eta(x)
	$$
	\item Commutativity with $\mu$:
	$$
	\begin{array}{rcll}
	\lefteqn{	(\mu \after M_S\zeta_X\after\zeta_{M_S(X)})(\sum_i s_i \sum_j t_{ij}x_{j})}\\
	&=&
	(\mu \after M_S\zeta_X)(\sum_i s_i^* \sum_j t_{ij}x_{j})\\
	&=&
	\mu(\sum_i s_i^* \sum_j t_{ij}^*x_{j})\\
	&=&
	\sum_j(\sum_i^S s_i^*t_{ij}^*)x_j &\text{definition $\mu$}\\
	&=&
	\sum_j(\sum_i^S s_it_{ij})^*x_j &\text{$*$ commutes with $\cdot$ and $+$}\\
	&=&
	(\zeta_X \after \mu)(\sum_i s_i \sum_j t_{ij}x_j)
	\end{array}
	$$
	\item $\zeta \after \zeta = \id$
	$$
	(\zeta \after \zeta)(\sum_i s_i x_i) = \zeta(\sum_i s_i^*x_i) = \sum_i s_i^{**}x_i = \sum_i s_i x_i.
	$$
	\item Let $g \colon S \to R$. To prove: $\mathcal{M}(g) \after \zeta^S = \zeta^R \after \mathcal{M}(g)$.\\
	$$
	\begin{array}{rcl}
	(\mathcal{M}(g) \after \zeta^S)(\sum_i s_i x_i) &=& \mathcal{M}(g) (\sum_i s_i^* x_i)\\
	&=&
	\sum_i g(s_i^*) x_i\\
	&=&
	\sum_i (g(s_i))^* x_i\\
	&=&
	(\zeta^R \after \mathcal{M}(g))(\sum_i s_ix_i)
	\end{array}
	$$
\end{enumerate}

\item $\Evc\colon \ACMnd \to \CSRng$ restricts to a functor $\ICSRng \to \IACMnd$:\\
Using Lemma \ref{CAMnd2CSRngLem} it is left to show that $\zeta_1$ is an involution on $T(1)$, and that, for a morphism $\sigma \colon T \to T'$, $\Ev(\sigma) = \sigma_1$ preserves the involution. This last point is trivial as, by definition of morphisms in $\IACMnd$, $\sigma \after \zeta = \zeta \after \sigma$. As for the first point:
\begin{enumerate}
\item $(a+b)^* = a^* + b^*$\\
Addition on $T(1)$ is given by $T(\nabla) \after \bc^{-1} \colon T(1) \times T(1) \to T(1)$.
$$
\begin{array}{rcll}
\zeta_1 \after T(\nabla) \after \bc^{-1} &=& T(\nabla) \after \zeta_{1+1} \after \bc^{-1} & \text{Naturality of $\zeta$}\\
&=&
T(\nabla) \after \bc^{-1} \after (\zeta_1 \times \zeta_1) & \text{Lemma \ref{bcproplem}(i)}
\end{array}
$$
\item $(a\cdot b)^* = a^*b^*$\\
Multiplication on $T(1)$ is given by $T(\pi_2) \after \dst \colon T(1) \times T(1) \to T(1)$.\
$$
\begin{array}{rcll}
\zeta_1 \after T(\pi_2) \after \dst &=& T(\pi_2) \after \zeta_{1\times 1} \after \dst & \text{Naturality of $\zeta$}\\
&=&
T(\pi_2) \after \dst \after (\zeta \times \zeta) & \text{Naturality of $\dst$}  
\end{array}
$$
\item $a^{**} = a$
$$
a^{**} = \zeta_1(\zeta_1(a)) = (\zeta_1 \after \zeta_1)(a) = \id(a) = a.
$$
\end{enumerate}

\item The restricted functors again form an adjunction.\\
  Consider the bijective correspondances defined in
  Lemma~\ref{AdjCSR2ACMLem}.
\begin{itemize}
\item $\overline{\sigma} \colon S \to T(1)$ preserves the involution.
$$
\begin{array}{rcll}
\overline{\sigma}(s^*) &=& \sigma_1(s^**)\\
&=&
\sigma_1(\zeta^{M_S}(s*))\\
&=&
\zeta^T(\sigma_1(s*))\\
&=&
(\overline{\sigma}(s))^* 
\end{array}
$$
\item $\overline{f} \colon M_S \to T$ preserves the involution.
$$
\begin{array}{rcll}
(\overline{f} \after \zeta^{M_S})(\sum_i s_ix_i) 
&=&
\overline{f}(\sum_i s_i^*x_i)\\
&=&
\sum^{T(X)}_if(s_i^*)\star\eta^T(x_i) &\text{$\star$ is the action of $T(1)$ on $T(X)$}\\
&=&
\sum^{T(X)}_if(s_i)^*\star\eta^T(x_i)\\
&=&
\sum^{T(X)}_i\zeta_1^T(f(s_i))\star\zeta^T(\eta^T(x_i)) &\text{$\zeta_1$ is the involution on $T(1)$}\\
&&&\text{and $\zeta \after \eta = \eta$}\\
&=&
\sum^{T(X)}_i\zeta^T((f(s_i))\star\eta^T(x_i))&\text{By Lemma \ref{Mnd2ModLem}, $\Mod(\zeta, \id) = (\zeta_1,\zeta)$}\\
&&&\text{is a map of modules}\\
&=&
\zeta^T(\sum^{T(X)}_ig(s_i)\star\eta^T(x_i))&\text{(1)}\\
&=&
(\zeta^T\after\overline{f})(\sum_i s_ix_i)
\end{array}
$$
(1) follows from the fact that
$$
T(\nabla) \after \bc^{-1} \after (\zeta \times \zeta) = \zeta \after T(\nabla) \after \bc^{-1},
$$
using Lemma \ref{bcproplem}(i) and naturality of $\zeta$.
\end{itemize}
\end{itemize}
}


The adjunction $\Mat \colon \CSRng \rightleftarrows \SMBLaw \colon
\mathcal{H}$ from Lemma~\ref{CSRng2CatAdjLem} may also be restricted
to involutive semirings. To do so, we have to consider dagger
categories.  A dagger category is a category $\cat{C}$ with a functor
$\dagger \colon \cat{C}\op \to \cat{C}$ that is the identity on
objects and satisfies, for all morphisms $f \colon X \to Y$,
$(f^{\dagger})^{\dagger} = f$. The functor $\dagger$ is called a
dagger on $\cat{C}$.  Combining this dagger with the categorical
structure we studied in Section \ref{Semiringcatsec} yields a
so-called dagger symmetric monoidal category with dagger biproducts,
that is, a category $\cat{C}$ with a symmetric monoidal structure
$(\otimes, I)$, a biproduct structure $(\oplus,0)$ and a dagger
$\dagger$, such that, for all morphisms $f$ and $g$, $(f \otimes
g)^{\dagger} = f^{\dagger} \otimes g^{\dagger}$, all the coherence
isomorphisms $\alpha$, $\rho$ and $\gamma$ are dagger isomorphisms
and, with respect to the biproduct structure, $\kappa_i =
\pi_i^{\dagger}$, where a dagger isomorphism is an isomorphism $f$
satisfying $f^{-1} = f^{\dagger}$. Further details may be found
in~\cite{AbramskyC04,AbramskyC09,Heunen10a}.

We will denote the category of dagger symmetric monoidal Lawvere
theories with dagger biproducts such that the monoidal structure
distributes over the biproduct structure by \DSMBLaw. Morphisms in
\DSMBLaw are maps in \SMBLaw that (strictly) commute with the daggers.

\begin{lemma}
  The functors $\Mat \colon \CSRng \rightleftarrows \SMBLaw \colon
  \mathcal{H}$ defined in Section \ref{Semiringcatsec} restrict to a
  pair of functors $\Mat \colon \ICSRng \rightleftarrows \DSMBLaw
  \colon \mathcal{H}$. The restricted functors form an adjunction,
  $\Mat \dashv \mathcal{H}$.
\end{lemma}

\begin{proof}
For an involutive semiring $S$, we may define a dagger on the Lawvere
theory $\Mat(S)$ by assigning to a morphism $f\colon n \to m$ in
$\Mat(S)$ the morphism $f^{\dagger} \colon m \to n$ given by
\begin{equation}
\label{MatDagEqn}
\begin{array}{rcl}
f^{\dagger}(i,j)
& = &
f(j,i)^{*}.
\end{array}
\end{equation}

\noindent Some short and straightforward computations show that the
functor $\dagger$ is indeed a dagger on $\Mat(S)$, which interacts
appropriately with the monoidal and biproduct structure. 

\auxproof{
\begin{enumerate}
\item $\dagger$ is a funtor\\
Clearly $\id^{\dagger} = \id$, as $1^* = 1$ and $0^* = 0$. Preservation of the composition is shown as follows:
$$
\begin{array}{rcll}
(f \after g)^{\dagger}(i,j) &=& ((g \after f)(j,i))^*\\
&=&
(\sum_l g(j,l)\cdot f(l,i))^*\\
&=&
\sum_l g(j,l)^*\cdot f(l,i)^*\\
&=&
\sum_l g(l,j)^{\dagger}\cdot f(i,l)^{\dagger}\\
&=&
\sum_l f(i,l)^{\dagger}\cdot g(l,j)^{\dagger}\\
&=&
(g^{\dagger}\after f^{\dagger})(i,j)
\end{array}
$$
\item $f^{\dagger\dagger} = f$
$$
f^{\dagger\dagger}(i,j) = (f^{\dagger}(j,i))^* = f(i,j)^{**} = f(i,j).
$$
\item $\pi_i^{\dagger} = \kappa_i$
\item $(h\otimes g)^{\dagger} = h^{\dagger}\otimes g^{\dagger}$\\
Let $h \colon n \times p \to S$ and $g \colon m \times q \to S$, then $(h\otimes g)^{\dagger} \colon (p \times q) \times (n \times m) \to S$.
$$
\begin{array}{rcll}
(h\otimes g)^{\dagger}((x_0, x_1),(y_0,y_1)) &=& ((h \otimes g)((y_0, y_1),(x_0, x_1)))^*\\
&=&
(h(y_0,x_0) \cdot g(y_1,x_1))^*\\
&=&
(h(y_0,x_0))^* \cdot (g(y_1,x_1))^*\\
&=&
h^{\dagger}(x_0, y_0) \cdot g^{\dagger}(x_1,y_1)\\
&=&
(h^{\dagger} \otimes g^{\dagger})((x_0, x_1),(y_0,y_1))
\end{array}
$$
\item $\dagger$ interacts appropriately with the other symmetric
  monoidal structure.\\ As the matrices representing the morphisms
  $\alpha \colon (n \otimes m)\otimes k \to n \otimes (m\otimes k)$,
  $\rho \colon n \otimes 1 \to n$, $\lambda \colon 1 \times n \to n$
  and $\gamma \colon n \otimes m \to n \otimes m$ only consists of
  zeros and ones, one easily sees that $\alpha^{\dagger} =
  \alpha^{-1}$ etc.

\item $Mat(f) \colon \Mat(S) \to \Mat(R)$ is a morphism in
  $\DSMBLaw$, for all semiring morphisms $f \colon S \to R$.\\ We
  have already seen that $\Mat(f)$ preserves the biproduct and
  monoidal structure. As to the dagger,
$$
\Mat(f)(h^{\dagger})(i,j) = h(f(j,i)^*) = (h(f(j,i)))^* = (\Mat(h)(f))^{\dagger}(i,j).
$$ 
\end{enumerate}
}

For a dagger symmetric monoidal Lawvere theory $\cat{L}$ with dagger
biproduct, it easily follows from the properties of the dagger that
this functor induces an involution on the semiring
$\mathcal{H}(\cat{L}) = \cat{L}(1,1)$, namely via $s\mapsto s^{\dag}$.

\auxproof{
\begin{enumerate}
\item $* (=\dagger) \colon \cat{L}(I,I) \to \cat{L}(1,1)$ is an
  involution.\\ Preservation on the addition:
$$
\begin{array}{rcll}
(a+b)^* &=& (\nabla \after (a \oplus b) \after \Delta)^{\dagger}\\
&=&
\Delta^{\dagger} \after (a \oplus b)^{\dagger} \after \nabla^{\dagger}\\
&=&
\nabla \after (a^{\dagger} \oplus b^{\dagger}) \after \Delta\\
&=&
a^*+b^*
\end{array}
$$
Here we use the fact that in a dagger biproduct category: $\nabla^{\dagger} = \Delta$, $\Delta^{\dagger} = \nabla$ and $(a \oplus b)^{\dagger} = a^{\dagger} \oplus b^{\dagger}$.

Preservation of the multiplication:
$$
(a \cdot b)^* = (b \after a)^{\dagger} = a^{\dagger} \after b^{\dagger} = b^* \cdot a^* = a^* \cdot b^*
$$  
and involutiveness:
$$
a^{**} = a^{\dagger\dagger} = a
$$
\item For $F \colon \cat{L} \to \cat{D}$, $\mathcal{H}(F) \colon
  \cat{L}(1,1) \to \cat{D}(1,1)$ preserves the involution.
$$
\mathcal{H}(F)(a^*) = F(a^*) = F(a^{\dagger}) = (F(a))^{\dagger} = F(a)^*.
$$
\end{enumerate}
} 

The unit and the counit of the adjunction $\Mat \colon \CSRng
\leftrightarrows \SMBLaw$ from Lemma~\ref{CSRng2CatAdjLem} preserve
the involution and the dagger respectively. Hence, also the restricted
functors form an adjunction. \QED
\end{proof}

\auxproof{
\begin{enumerate}
\item $\overline{F} \colon S \to \cat{L}(I,I)$ preserves the involution.
$$
\overline{F}(s^*) = F(\underline{s^*}) = F((\underline{s})^{\dagger}) = (F(\underline{s}))^{\dagger} = (F(\underline{s}))^* = (\overline{F}(s))^*,
$$
where $\underline{s} = 1 \times 1 \xrightarrow{\lam{x}{s}} S$.

\item $\overline{f} \colon \Mat(S) \to \cat{L}$ preserves the dagger.\\
Let $h \colon n \to m$ in $\Mat(S)$. Then $h^{\dagger} \colon m \to n$ 
$$
\begin{array}{rcll}
\overline{f}(h^{\dagger}) &=& \big[\big\langle f(h^{\dagger}(j,i))\big\rangle_{i\in n}\big]_{j \in m}\\
&=&
\big[\big\langle f((h(i,j))^*)\big\rangle_{i\in n}\big]_{j \in m}&\text{definition of $\dagger$ on $\Mat(S)$}\\
&=&
\big[\big\langle (f((h(i,j)))^{\dagger}\big\rangle_{i\in n}\big]_{j \in m} &\text{$\dagger$ is the involution on $\cat{L}(I,I)$}\\
&=&
\big\langle\big[ (f((h(i,j)))\big]_{i\in n}\big\rangle_{j \in m}^{\dagger} &\text{$\dagger$ interacts with the biproduct structure}\\
&=&
\big[\big\langle (f((h(i,j)))\big\rangle_{j\in m}\big]_{i \in n}^{\dagger} &\text{property of the biproduct}\\
&=&
(\overline{f}(h))^{\dagger} 
\end{array}
$$
\end{enumerate}
}

To complete our last triangle of adjunctions, recall that, for the
Lawvere theory associated with a (involutive commutative additive)
monad $T$, $\Kl_{\NNO}(T) \cong \Mat(\Evc(T))$, see
Proposition~\ref{KleisliMatLem}. Hence, using the previous two
lemmas, the finitary Kleisli construction restricts to a functor
$\Kl_{\NNO} \colon \IACMnd \to \DSMBLaw$. For the other direction we
use the following result.

\begin{lemma}
\label{InvLMCommLem}
The functor $\LM\colon\SMBLaw\rightarrow\ACMnd$ from
Lemma~\ref{LMCommLem} restricts to $\DSMBLaw \rightarrow \IACMnd$,
and yields a left adjoint to $\Kl_{\NNO}\colon \IACMnd\rightarrow
\DSMBLaw$.
\end{lemma}

\begin{proof}
To start, for a Lawvere theory $\cat{L}\in\DSMBLaw$ with dagger $\dag$
we have to define an involution $\zeta\colon T_{\cat{L}} \rightarrow
T_{\cat{L}}$. For a set $X$ this involves a map
$$\xymatrix@R-1.8pc{
\llap{$T_{\cat{L}}(X) =\;$}
   \big(\coprod_{i\in\NNO}\cat{L}(1,i)\times X^{i}\Big)/\!\sim
      \ar[r]^-{\zeta_X} & 
   \big(\coprod_{i\in\NNO}\cat{L}(1,i)\times X^{i}\Big)/\!\sim
   \rlap{$\;=T_{\cat{L}}(X)$} \\
[\kappa_{i}(g,v)]\ar@{|->}[r] & 
   [\kappa_{i}(\langle g_{0}^{\dag}, \ldots, g_{i-1}^{\dag}\rangle, v)],
}$$

\noindent where $g\colon 1\rightarrow i$ is written as $g = \langle
g_{0}, \ldots, g_{i-1}\rangle$ using that $i = 1+\cdots+1$ is not only
a sum, but also a product. Clearly, $\zeta$ is natural, and satisfies
$\zeta\after\zeta = \idmap{}$.  This $\zeta$ is also a map of monads;
commutatution with multiplication $\mu$ requires commutativity of
composition in the homset $\cat{L}(1,1)$.

The unit of the adjunction $\eta\colon \cat{L} \rightarrow
\Kl_{\NNO}(T_{\cat{L}}) \cong \Mat(T_{\cat{L}}(1))$ commutes with
daggers, since for $f\colon n\rightarrow m$ in \cat{L} we get
$\eta(f)^{\dag} = \eta(f^{\dag})$ via the following argument in
$\Mat(T_{\cat{L}}(1))$. For $i<n$ and $j<m$,
$$\begin{array}[b]{rcll}
\eta(f)^{\dag}(i,j)
& = &
\eta(f)(i,j)^{*}
   & \mbox{by~\eqref{MatDagEqn}} \\
& = &
\pi_{j}\bc_{m}(\kappa_{m}(f \after \kappa_{i}, \idmap_{m}))^{*} 
   & \mbox{by~\eqref{Kl2MatEqn}} \\
& = &
\kappa_{1}(\pi_{j} \after f \after \kappa_{i}, \idmap_{1})^{*}
   & \mbox{by definition of $\bc$, see~\eqref{LawMndbcEqn}} \\
& = & 
\kappa_{1}(\pi_{j} \after f \after \kappa_{i}, \idmap)^{\dag} 
   & \mbox{since $(-)^{*} = (-)^{\dag}$ on $T_{\cat{L}}(1)$} \\
& = &
\kappa_{1}((\pi_{j} \after f \after \kappa_{i})^{\dag}, \idmap) \\
& = &
\kappa_{1}(\kappa_{i}^{\dag} \after f^{\dag} \after \pi_{j}^{\dag}, \idmap) \\
& = &
\kappa_{1}(\pi_{i} \after f^{\dag} \after \kappa_{j}, \idmap) \\
& = &
\eta(f^{\dag})(i,j).
\end{array}\eqno{\QEDbox}$$

\auxproof{
We check commutation with $\eta$ and $\mu$.
$$\begin{array}{rcl}
\lefteqn{\big(\zeta \after \eta\big)(x)} \\
& = &
\zeta(\kappa_{1}(\idmap_{1}, x) \\
& = &
\kappa_{1}((\idmap_{1})^{\dag}, x) \\
& = &
\kappa_{1}(\idmap_{1}, x) \\
& = &
\eta(x) \\
\lefteqn{\big(\mu \after T_{\cat{L}}(\zeta) \after 
   \zeta\big)(\kappa_{i}(g, v))} \\
& = &
\big(\mu \after T_{\cat{L}}(\zeta)\big)
   (\kappa_{i}(\langle (\pi_{a}\after g)^{\dag}\rangle_{a<i}, v)) \\
& = &
\mu(\kappa_{i}(\langle (\pi_{a}\after g)^{\dag}\rangle_{a<i}, 
   \zeta \after v)) \\
& = &
\mu(\kappa_{i}(\langle (\pi_{a}\after g)^{\dag}\rangle_{a<i}, 
   \lam{a<i}{\kappa_{j_a}(\langle (\pi_{b} \after h_{a,b})^{\dag} \rangle_{b<j_a}, 
    w_a)})) \\
& & \qquad \mbox{if }v(a) = \kappa_{j_a}(h_{a}, w_{a}) \\
& = &
\kappa_{j}((\langle (\pi_{b} \after h_{0})^{\dag} \rangle_{b<j_0} + \cdots +
   \langle (\pi_{b} \after h_{i-1})^{\dag} \rangle_{b<j_i-1}) \after
   \langle (\pi_{a}\after g)^{\dag}\rangle_{a<i}, \\
& & \qquad [w_{0}, \ldots, w_{i-1}]), \qquad 
   \mbox{where }j = j_{0} + \cdots + j_{i-1} \\
& \smash{\stackrel{(*)}{=}} &
\kappa_{j}(\langle (\pi_{b} \after \pi_{a} \after 
   (h_{0}+\cdots+h_{i-1})\after g)^{\dag}
   \rangle_{a<i, b<j_{a}}, [w_{0}, \ldots, w_{i-1}]) \\
& = &
\zeta(\kappa_{j}((h_{0}+\cdots+h_{i-1})\after g, [w_{0}, \ldots, w_{i-1}])) \\
& = &
\big(\zeta \after \mu)(\kappa_{i}(g, v)).
\end{array}$$

\noindent The marked equation holds by commutativity of composition
in $\cat{L}(1,1)$, see:
$$\begin{array}{rcl}
\lefteqn{\pi_{b} \after \pi_{a} \after 
   \langle (\pi_{b} \after \pi_{a} \after 
      (h_{0}+\cdots+h_{i-1})\after g)^{\dag} \rangle_{a<i, b<j_{a}}} \\
& = &
(\pi_{b} \after \pi_{a} \after (h_{0}+\cdots+h_{i-1})\after g)^{\dag} \\
& = &
(\pi_{b} \after h_{a} \after \pi_{a} \after g)^{\dag} \\
& = &
(\pi_{a} \after g)^{\dag} \after (\pi_{b} \after h)^{\dag} \\
& \smash{\stackrel{\textrm{comm}}{=}} &
(\pi_{b} \after h)^{\dag} \after (\pi_{a} \after g)^{\dag} \\
& = &
\pi_{b} \after \langle (\pi_{b} \after h_{a})^{\dag} \rangle_{b<j_a} \after 
   \pi_{a} \after \langle (\pi_{a}\after g)^{\dag}\rangle_{a<i} \\
& = &
\pi_{b} \after \pi_{a} \after 
   (\langle (\pi_{b} \after h_{0})^{\dag} \rangle_{b<j_0} + \cdots +
   \langle (\pi_{b} \after h_{i-1})^{\dag} \rangle_{b<j_i-1}) \after
   \langle (\pi_{a}\after g)^{\dag}\rangle_{a<i}.
\end{array}$$
}
\end{proof}

In the definition of the involution $\zeta$ on the monad $T_{\cat{L}}$
in this proof we have used that $+$ is a (bi)product in the Lawvere
theory \cat{L}, namely when we decompose the map $g\colon 1\rightarrow
i$ into its components $\pi_{a} \after g\colon 1\rightarrow 1$ for
$a<i$. We could have avoided this biproduct structure by first taking
the dagger $g^{\dag} \colon i\rightarrow 1$, and then precomposing
with coprojections $g^{\dag} \after \kappa_{a} \colon 1 \rightarrow
1$. Again applying daggers, cotupling, and taking the dagger one gets
the same result. This is relevant if one wishes to consider
involutions/daggers in the context of monoids, where products in the
corresponding Lawvere theories are lacking.

\auxproof{
These two approaches are the same since if we have biproducts:
$$\begin{array}{rcl}
\big[(g^{\dag} \after \kappa_{0})^{\dag}, \ldots, 
   (g^{\dag} \after \kappa_{i-1})^{\dag}\big]^{\dag}
& = &
\big[\kappa_{0}^{\dag} \after g^{\dag\dag}, \ldots, 
   \kappa_{i-1}^{\dag} \after g^{\dag\dag}\big]^{\dag} \\
& = &
\big[\pi_{0} \after g, \ldots, \pi_{i-1} \after g\big]^{\dag} \\
& = &
\langle (\pi_{0}\after g)^{\dag}, \ldots, (\pi_{i-1}\after g)^{\dag} \rangle.
\end{array}$$
}

By combining the previous three lemmas we obtain another triangle of
adjunctions in Figure~\ref{ICSRngTriangleFig}. This concludes our
survey of the interrelatedness of scalars, monads and categories.

\begin{figure}
$$\xymatrix@R-.5pc@C+.5pc{
& & \ICSRng\ar@/_2ex/ [ddll]_{\cal M}
   \ar@/_2ex/ [ddrr]_(0.3){\Mat \cong \Kl_{\NNO}\mathcal{M}\;\;} \\
& \dashv & & \dashv & \\
\IACMnd\ar @/_2ex/[rrrr]_{\Kl_\NNO}
   \ar@/_2ex/ [uurr]_(0.6){\;{\Ev} \cong \mathcal{H}\Kl_{\NNO}} & & \bot & &  
   \DSMBLaw\ar@/_2ex/ [uull]_{\mathcal{H}}\ar @/_2ex/[llll]_{\LM}
}$$
\caption{Triangle of adjunctions starting from involutive commutative
  semi\-rings, with involutive commutative additive monads, and dagger
  symmetric monoidal Lawvere theories with dagger biproducts.}
\label{ICSRngTriangleFig}
\end{figure}
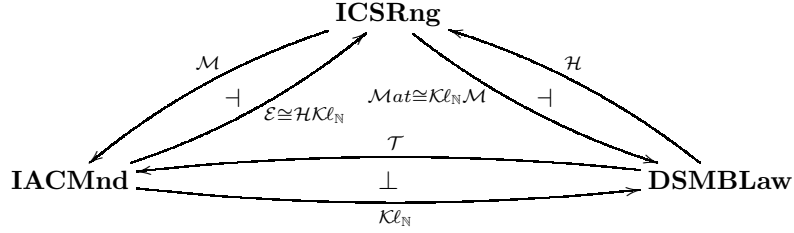

\bibliographystyle{plain}

\end{document}